\documentclass[11pt]{amsart}
\usepackage[latin1]{inputenc}
\usepackage{amsmath}
\usepackage{amsfonts}
\usepackage{amssymb}
\usepackage{graphicx}
\usepackage{fourier}
\usepackage{dsfont}
\usepackage{mathrsfs}
\usepackage{hyperref}
\usepackage{enumerate}
\usepackage{esint}
\usepackage{bm}
\usepackage{ulem}
\usepackage{xcolor}
\newtheorem{theorem}{Theorem}[section]
\newtheorem{lemma}[theorem]{Lemma}

\theoremstyle{definition}
\newtheorem{definition}[theorem]{Definition}

\newtheorem{prop}[theorem]{Proposition}
\newtheorem{claim}[theorem]{Claim}

\theoremstyle{remark}
\newtheorem{remark}[theorem]{Remark}

\newtheorem{ques}{Question}
\newtheorem*{ques*}{Question}

\newcommand{\norm}[1]{\left\lVert#1\right\rVert}
\newcommand{\abs}[1]{\left\lvert#1\right\rvert}
\newcommand{\pa}[1]{\left( #1 \right)}
\newcommand{\rpa}[1]{\left[ #1 \right]}
\newcommand{\br}[1]{\left\lbrace #1\right\rbrace}
\newcommand{\R}{\mathbb{R}}

\newcommand{\N}{\mathbb{N}}
\newcommand{\F}{\mathcal{F}}

\numberwithin{equation}{section}
\newcommand{\C}{\mathbb{C}}
\newcommand{\St}{\mathbb{S}}
\newcommand{\W}{\mathcal{W}}

\newcommand{\Li}{\mathcal{L}}
\newcommand{\rw}{r_\omega}

\definecolor{darkblue}{rgb}{0.0, 0.0, 0.55}

\definecolor{ForestGreen}{RGB}{034,139,034}

\title{Interpolation of weighted Sobolev spaces}
\author{Michael Cwikel and Amit Einav}
\address{Department of Mathematics, Technion - Israel Institute of Technology, Haifa 32000, Israel}
\email{mcwikel@math.technion.ac.il}
\address{Institut f\"ur Analysis und Scientific Computing,
  Technische Universit\"at Wien, Wiedner Hauptstrasse 8-10
  A-1040 Vienna, Austria}
\email{amit.einav@asc.tuwien.ac.at}
\thanks{The first named author's work was supported by the Technion V.P.R. Fund, and by the Fund for Promotion of Research at the Technion. The second named author was supported by the Austrian Science Fund (FWF) grant M 2104-N32.}
\begin{document}
\date{}
\begin{abstract}
In this work we present a newly developed study of the interpolation of weighted Sobolev spaces by the complex method. We show that in some cases, one can obtain an analogue of the famous Stein-Weiss theorem for weighted $L^{p}$ spaces. We consider an example which gives some indication that this may not be possible in all cases. Our results apply in cases which cannot be treated by methods in earlier papers about interpolation of weighted Sobolev spaces. They include, for example, a proof that $\left[W^{1,p}\pa{\mathbb{R}^{d},\omega_{0}},W^{1,p}\pa{\mathbb{R}^{d},\omega_{1}}\right]_{\theta}=W^{1,p}\pa{\mathbb{R}^{d},\omega_{0}^{1-\theta}\omega_{1}^{\theta}}$ whenever $\omega_{0}$ and $\omega_{1}$ are continuous and their quotient is the exponential of a Lipschitz function. We also mention some possible applications of such interpolation in the study of convergence in evolution equations. 
\end{abstract}
\maketitle
\tableofcontents

\section{Introduction}\label{sec:into}
\subsection{Goals and aims}\label{subsec:goals}
Consider two weighted Sobolev spaces, $W^{s_{0},p_{0}}(U,\omega_{0})$ and $W^{s_{1},p_{1}}(U,\omega_{1})$ on an open subset $U$ of $\mathbb{R}^{d}$. (Precise definitions of these spaces and of other notions, which are mentioned only in general terms in this subsection, will be given later). A natural question one can ask oneself is the following:

\begin{ques}\label{ques:interpolation_sobolev}
If $T$ is a linear operator which is bounded on both $W^{s_{0},p_{0}}(U,\omega_{0})$ and $W^{s_{1},p_{1}}(U,\omega_{1})$, does this imply that $T$ is also bounded on various other weighted Sobolev spaces $W^{s,p}(U,\omega)$?
\end{ques}

For certain choices of $U$, weight functions $\omega_{0}$, $\omega_{1}$, and $s_{0}$, $s_{1}$, $p_{0}$ and $p_{1}$  the answer is known to be affirmative. In most cases, the affirmation of Question \ref{ques:interpolation_sobolev} is obtained with the explicit, or implicit, help of Alberto Calder\'on's complex interpolation spaces. As in many (but not all) papers dealing with these interpolation spaces, we will use Calder\'on's notation $\left[A_{0},A_{1}\right]_{\theta}$ for them. Thus, so far, the affirmation of Question \ref{ques:interpolation_sobolev} comes hand in hand with the affirmation of the following question:
\begin{ques}\label{ques:complex_method_sobolev}
Does the complex interpolation space $\left[W^{s_{0},p_{0}}(U,\omega_{0}),W^{s_{1},p_{1}}(U,\omega_{1})\right]_{\theta}$ coincide with $W^{s,p}(U,\omega)$ for some $s$, $p$ and $\omega$?
\end{ques}
Our goal in this work is to investigate Question \ref{ques:complex_method_sobolev} further, and provide some new affirmative answers to it - thereby also finding more cases where Question \ref{ques:interpolation_sobolev} has an affirmative answer. We consider that our results can be potentially useful in the study of asymptotic behaviour of solutions of certain evolution equations.

Before continuing, let us recall two previously treated cases of Questions \ref{ques:interpolation_sobolev} and \ref{ques:complex_method_sobolev}:

\textbf{The Stein-Weiss Theorem:} Suppose that $s_{0}=s_{0}=0$, i.e., suppose that our two weighted Sobolev spaces are simply weighted $L^{p}$ spaces. Then a celebrated result, often referred to as the Stein-Weiss Theorem, which goes back to the paper \cite{Stein} of Eli Stein and Stein's joint paper \cite{SW} with Guido Weiss, gives a positive answer to Question \ref{ques:interpolation_sobolev} for \textit{all} $p_{0}$ and $p_{1}$ in $\left[1,\infty\right]$ and \textit{all} choices of positive measurable functions $\omega_{0}$
and $\omega_{1}$ on $U$ (and even in a more general setting, when $(U,dx)$ is replaced by an arbitrary measure space). More explicitly, in this case, the operator $T$ is bounded on $W^{0,p}(U,\omega)$ whenever $p$ and $\omega$ satisfy 
\begin{equation}\label{eq:pwtheta}
\frac{1}{p}=\frac{1-\theta}{p_{0}}+\frac{\theta}{p_{1}},\quad \text{and} \quad \omega^{\frac{1}{p}}=\omega_{0}^{\frac{1-\theta}{p_0}}\omega_{1}^{\frac{\theta}{p_{1}}}.
\end{equation}
for some $\theta\in(0,1)$. This result was obtained before Calder\'on developed his theory of complex interpolation spaces, but in the light of that theory the Stein-Weiss Theorem can also be expressed by the formula 
\begin{equation}
\left[W^{0,p_{0}}(U,\omega_{0}),W^{0,p_{1}}(U,\omega_{1})\right]_{\theta}=W^{0,p}(U,\omega)\label{eq:SteinWeiss}
\end{equation}
which holds isometrically for each $\theta\in(0,1)$ and for $p$ and $\omega$ as in (\ref{eq:pwtheta}). (When $p_{0}=p_{1}=\infty$ the weight $\omega$ of course cannot be determined from \eqref{eq:pwtheta} but in that case \eqref{eq:SteinWeiss} holds for all choices of $\omega$ for a trivial reason related to the discussion below in Remarks \ref{rem:InfinityIsTrivial} and \ref{rem:WhyInfinityIsTrivial}.)\\
\textbf{A result of J\"orgen L\"ofstr\"om:} In his paper \cite{Lofstrom}, among various other results, L\"ofstr\"om gives an affirmative answer to Question \ref{ques:complex_method_sobolev} in the setting where $s_{0}$ and $s_{1}$ can be non-zero. However, the set $U$ must be all of $\mathbb{R}^{d}$ and the
powers $\omega_{0}^{1/p_{0}}$ and $\omega_{1}^{1/p_{1}}$ of the given weight functions are required to satisfy a special condition which he calls ``polynomial regularity''. With these conditions fulfilled, L\"ofstr\"om is able to use the Fourier transform and the Mihlin multiplier theorem and to define the weighted Sobolev spaces $W^{s_{0},p_{0}}\pa{\mathbb{R}^{d},\omega_{0}}$ and $W^{s_{1},p_{1}}\pa{\mathbb{R}^{d},\omega_{1}}$ via properties of the Fourier transforms of their elements. In this setting L\"ofstr\"om proves that
\begin{equation}\label{eq:Lofstrom}
\rpa{W^{s_{0},p_{0}}\pa{\R^d,\omega_{0}},W^{s_{1},p_{1}}\pa{\R^d,\omega_{1}}}_{\theta}=W^{s_{\theta},p}\pa{\R^d,\omega}
\end{equation}
to within equivalence of norms, for each $p_0$ and $p_1\in (1,\infty)$ and each $\theta\in(0,1)$ where $s_{\theta}=(1-\theta)s_{0}+\theta s_{1}$ and $p$ and $\omega$ are given by (\ref{eq:pwtheta}).  Taking into account the particular different notation and formulation of definitions used by L\"ofstr\"om, one can see that this result is expressed by the formula (5.3) in Theorem 4 on page 208 of \cite{Lofstrom}. An interesting advantage of L\"ofstr\"om's method is that he can extend his definition of weighted Sobolev spaces to the case where their orders of smoothness are not necessarily integers, and in fact he proves his formula (\ref{eq:Lofstrom}) in this more general setting.

\par Of course, in the special case where $s_{0}=s_{1}=0$, the formula (\ref{eq:Lofstrom}) becomes the Stein-Weiss formula (\ref{eq:SteinWeiss}).
\par In this paper, in contrast to L\"ofstr\"om's work, we shall only consider the case
$$s_0=s_1=1,\quad\quad p_0=p_1=p\in[1,\infty).$$
However, making these restrictions gives us much more flexibility in our choices of the set $U$ and the weight functions $\omega_{0}$ and $\omega_{1}$ than is available in \cite{Lofstrom}. We can also include the case where $p=1$. For a relatively large class of weight functions, and for $U=\R^d$, we obtain that
\begin{equation}\label{eq:OurFormula}
\left[W^{1,p}(U,\omega_{0}),W^{1,p}(U,\omega_{1})\right]_{\theta}=W^{1,p}(U,\omega_{\theta})
\end{equation}
to within equivalence of norms, for all $p\in [1,\infty)$ and $\theta\in(0,1)$, where here, and also in fact throughout this paper, $\omega_\theta$ is the weight function obtained, for each $\theta$, from whichever weight functions $\omega_0$ and $\omega_1$ are currently under consideration, by the formula 
\begin{equation}\label{eq:omegatheta}
\omega_{\theta}=\omega_{0}^{1-\theta}\omega_{1}^{\theta}.
\end{equation}
\begin{remark}\label{rem:InfinityIsTrivial}When $p=\infty$, a case which we have not formally included in our discussion, the formula \eqref{eq:OurFormula} is an uninteresting triviality (see Remark \ref{rem:WhyInfinityIsTrivial} below) which holds isometrically for all choices of the weight functions $\omega_{0}$ and $\omega_{1}$ (as can be justified later by part \eqref{item:inter_properties_A} of Theorem \ref{thm:OtherProperties}). 
\end{remark}
Here is one easily formulated special case of our results, which gives some indication of the considerable flexibility we have in our choices of the weight functions $\omega_{0}$ and $\omega_{1}$: When the set $U$ is chosen to be all of $\mathbb{R}^{d}$, we can prove that (\ref{eq:OurFormula}) holds whenever $\omega_{0}/\omega_{1}$ is the exponential of a Lipschitz function and $\omega_{0}$ and $\omega_{1}$ are continuous. In fact we can make do here with a much weaker condition than continuity which is satisfied by various functions which are not even equivalent to continuous functions.

In fact \eqref{eq:OurFormula} is valid not only for $U=\R^d$, but also for a certain collection of other open sets which satisfy properties which we shall specify below. Moreover, if we consider the natural subspaces $W^{1,p}_0\pa{U,\omega}$ of functions in $W^{1,p}\pa{U,\omega}$ that can be approximated by a natural class of compactly supported functions, an analogue of \eqref{eq:OurFormula} for these spaces is valid for an even larger collection of open sets $U$. 

Our strategy for obtaining (\ref{eq:OurFormula}) will be to show that, under somewhat milder conditions on $\omega_{0}$ and $\omega_{1}$ than those imposed in L\"ofstr\"om's theorem, and for appropriate choices of $U$, the interpolation space $\left[W^{1,p}(U,\omega_{0}),W^{1,p}(U,\omega_{1})\right]_{\theta}$ satisfies the continuous inclusions 
$$ \mathcal{W}^{p}\left(U,\theta,r_{\omega}\right)\subset\left[W^{1,p}(U,\omega_{0}),W^{1,p}(U,\omega_{1})\right]_{\theta}\subset W^{1,p}(U,\omega_{\theta})$$
for a certain space $\mathcal{W}^{p}\left(U,\theta,r_{\omega}\right)$ which is the intersection of $W^{1,p}(U,\omega_{\theta})$ with an appropriate weighted $L^{p}$
space. These inclusions may also be of independent interest. We can then deduce that (\ref{eq:OurFormula}) holds under some explicitly
formulated conditions (see Theorem \ref{thm:Stein_Weiss_for_W}), which ensure that $W^{1,p}(U,\omega_{\theta})$ and $\mathcal{W}^{p}\left(U,\theta,r_{\omega}\right)$ coincide.\\
Interestingly enough, we will also present a somewhat intricate example where these spaces do not coincide, but we cannot exclude the possibility that (\ref{eq:OurFormula}) might still hold for that example. \\
We are now ready, in the following subsections, to give more explicit and more detailed formulations of our results and of the definitions of the notions which play roles in them.
\begin{remark} \label{rem:OtherPapersEtcZZ}
There are a number of other papers and books which contain other interpolation results involving weighted Sobolev spaces (and also sometimes other related spaces such as weighted Besov spaces and/or weighted Triebel-Lizorkin spaces) but the results in them do not quite correspond to the kind of results that we are seeking here. 
Some of them deal with the real method, rather than the complex method of interpolation. Some deal with interpolation
between weighted Sobolev spaces and weighted $L^{p}$ spaces. Also, in some of them, the definitions of weighted Sobolev spaces differ in some non-trivial way from the definition that we have chosen to use here. We can refer, for example, to \cite{EdmundsDTriebelH1996,Grisvard,HaroskeDTriebelH1994A,HaroskeDTriebelH1994B,HaroskeDTriebelH2005,PyatkovS1995,Pya,PyatkovS2001,PyatkovS2010,SchmiesserHTriebelH1987,TriebelH1978}.
Among the many other papers and some books which deal with other aspects of weighted Sobolev spaces and some of their applications we can mention, for example, the works \cite{Kufner,KO,KufnerASandigA1987} of Kufner, Kufner-Opic, and Kufner-S\"andig.
\end{remark}
\subsection{The setting of the problem}\label{subsec:settingZZ} In what is to follow, $d$ will always be a positive integer, $p$ will always lie in the interval $[1,\infty]$, and $U$ will always denote a non-empty open subset of $\R^d$. For most of our results $p$ will be restricted to the range $[1,\infty)$. The letter $\omega$, with or without a subscript, will always denote a \textit{weight function on $U$}, i.e. a Lebesgue measurable function $\omega:U\to (0,\infty)$.\\
Often the weight functions which are of interest in various applications are continuous. However, for our results here which establish the formula \eqref{eq:OurFormula} in certain cases, we will not need the weight functions $\omega_0$ and $\omega_1$ to be continuous, nor even equivalent to continuous functions. Instead we will require them to satisfy a rather weaker condition which we are now about to define. (We will however need $\omega_0/\omega_1$ to have some smoothness properties.)\\
We will be considering weight functions $\omega:U\to(0,\infty)$ which, like continuous functions,
satisfy 
\begin{equation}\label{eq:compact_boundedness}
0<\inf_{x\in K}\omega(x)\leq\sup_{x\in K}\omega(x)<\infty,\quad\text{for every compact set }K\subset U.
\end{equation}
 Consequently, it will sometimes be convenient to use the notation
\begin{equation}
m(K,\omega):=\inf_{x\in K}\omega(x),\quad \text{and}\quad M(K,\omega):=\sup_{x\in K}\omega(x).\label{eq:Define_m_and_M}
\end{equation}
\begin{definition}\label{def:compact_boundednessNewCompactBoundednessZZ}
A weight function $\omega:U\to(0,\infty)$ on an open subset $U$ of $\mathbb{R}^{d}$ which satisfies (\ref{eq:compact_boundedness}) will be said to satisfy the \textit{compact boundedness condition on }$U$.
\end{definition}
\begin{remark}\label{rem:MuchMoreGeneralThanContinuous} A weight function $\omega$ satisfying this condition can be very far from being continuous. As we already intimated above, it can, for example, have the property that, for every choice of a pair of positive constants $C_{1}$ and $C_{2}$ and of a continuous function $f$, the inequality $C_{1}f(x)\le\omega(x)\le C_{2}f(x)$
cannot hold for all $x\in U$, not even for almost all $x\in U$. As an example of this, let $U=(0,\infty)$ and let $\omega$ be the weight function on $U$ defined by $\omega(x)=\sum_{n=0}^{\infty}n\chi_{(2n,2n+1]}+\sum_{n=0}^{\infty}n^{2}\chi_{(2n,+1,2n+2]}$.
\end{remark}
\begin{remark}\label{rem:CBC-obvious}
We will have occasions to use the obvious fact that, if $\omega_{0}$ and $\omega_{1}$ are weight functions which satisfy the compact boundedness condition on some set $U$, then the same is true for every power, product and quotient of these weights, and in particular for $\omega_{0}^{1-\theta}\omega_{1}^{\theta}$ and $\omega_{0}/\omega_{1}$. 
\end{remark}
In the literature there are two standard ways of defining weighted $L^{p}$ spaces with respect to a given weight $\omega$. In one of them the weighted space consists of all functions $f$ for which $\omega f$ is in the (unweighted) $L^{p}$ space. But here we shall choose the second way, which proceeds by replacing the ambient underlying measure, in our case Lebesgue measure $dx$, with the weighted measure $\omega(x)dx$. In fact we shall need to use weighted $L^{p}$ spaces of vector valued functions.
\begin{definition}\label{def:L^p_w}
Given a weight function $\omega$ on $U$ and $p\in[1,\infty)$, and $m\in \N$, we define the Banach space
$$L^p\pa{U,\omega,\C^m}=\br{\phi:U \to\C^m\;|\;\phi\text{ is measurable and } \norm{\phi}_{L^p\pa{U,\omega,\C^m}}<\infty}$$
where
$$\norm{\phi}_{L^p\pa{U,\omega,\C^m}}=\pa{\sum_{i=1}^m \int_{U}\abs{\phi_i(x)}^p\omega(x)dx}^{\frac{1}{p}},$$
for each $\phi=\pa{\phi_1,\dots,\phi_m}$. 
\par We will sometimes use the abbreviated notation $L^{p}\pa{U,\omega}$ for $L^{p}(U,\omega,\mathbb{C})$, i.e. when $m=1$. We will also use the notation $L^{p}(U,\mathbb{C}^{m})$ for the ``unweighted'' case, where $\omega$ is identically $1$.
\end{definition}
As is customary in such definitions, there is a natural ``limiting'' extension to the case where $p=\infty$, which in our context takes the following form:
\begin{definition} \label{def:When_p_is_infinite} For $U$, $\omega$ and $m$ as above we define the Banach space 
$$ L^{\infty}\left(U,\omega,\mathbb{C}^{m}\right)=\left\{ \phi:U\to\mathbb{C}^{m}\;|\;\phi\text{ is measurable and }\left\Vert \phi\right\Vert _{L^{\infty}\left(U,\omega,\mathbb{C}^{m}\right)}<\infty\right\} $$
where 
$$ \mathrm{\left\Vert \phi\right\Vert _{L^{\infty}\left(U,\omega,\mathbb{C}^{m}\right)}=ess\, sup}_{x\in U}\left(\max_{i\in\{1,2,\dots,m\}}\left|\phi_{i}(x)\right|\right)
$$
for each $\phi=\left(\phi_{1},\dots,\phi_{m}\right)$.
\end{definition}

\begin{remark} \label{rem:IndependentOfWeight} The consequence of the choice that we made in formulating Definition \ref{def:L^p_w} is that here the weighted $L^{\infty}$ space in fact does not depend on $\omega$ and coincides isometrically with the unweighted space $L^{\infty}(U,\mathbb{C}^{m})$ which we shall occasionally use below. Had we chosen the first way of defining weighted $L^{p}$ spaces we would have obtained a more ``interesting'' space here.
\end{remark}
We can now introduce the main spaces that we wish to study.
\begin{definition}\label{def:W^p_wZZ}
Given an open set $U$ of $\R^d$, a weight function $\omega$ on $U$ and $p\in [1,\infty]$, we define
\begin{equation}\nonumber
\begin{split}
W^{1,p}\pa{U,\omega}=\bigg\{\phi:U\to &\C\;|\;\phi\text{ has a weak gradient }\nabla \phi\text{ on }U\text{ and}\\
&\pa{\phi,\nabla \phi}\in L^p\pa{U,\omega,\C^{d+1}} \bigg\},
\end{split}
\end{equation}
where the weak gradient $\nabla \phi$ (i.e. the vector of its weak first order partial derivatives) is defined in the usual way via the theory of distributions, i.e. for every $\psi\in C^\infty_c\pa{U}$ it satisfies 
\begin{equation}\label{eq:WeakDeriv}
\int_{U}\nabla \psi (x) \phi(x)dx = - \int_{U} \psi (x) \nabla \phi(x)dx.
\end{equation}
The norm\footnote{Of course in this definition, and in Definitions \ref{def:L^p_w} and \ref{def:When_p_is_infinite}, we permit ourselves the usual abuse of terminology where the word ``function'' really means an equivalence class of functions which are equal to each other almost everywhere.} which we will define on $W^{1,p}\pa{U,\omega}$ is given by
$$\norm{\phi}_{W^{1,p}\pa{U,\omega}}=\norm{\pa{\phi,\nabla \phi}}_{L^p\pa{U,\omega,\C^{d+1}}}.$$
\end{definition} 
\begin{remark}\label{rem:LocalIntegrabilityZZ}The requirement that $\phi$ and $\nabla\phi$ satisfy the equation (\ref{eq:WeakDeriv}) in which the integrals are evidently understood to be Lebesgue integrals, forces both of these functions to be locally integrable on $U$ (with respect to (unweighted) Lebesgue measure). 
\end{remark}
\begin{remark}\label{rem:Unweighted}
In the case where the weight function $\omega$ is identically equal to $1$ the above weighted Sobolev spaces coincide with their original and more frequently studied ``unweighted'' analogues. We shall occasionally need to use these unweighted spaces here and (analogously to what we did for unweighted $L^p$ spaces) we shall use the same notation as above for them, but with $\omega$ simply omitted. Note (cf. Remark \ref{rem:IndependentOfWeight}) that, for every weight function $\omega$, $W^{1,\infty}(U,\omega)$ coincides isometrically with the unweighted Sobolev space $W^{1,\infty}(U)$. 
\end{remark}

The study of weighted Sobolev spaces is certainly not new. The definition which we have chosen to use for them here is only one of a number of various different definitions chosen by authors of the many papers and several books (already mentioned above in \S \ref{subsec:goals}) which deal with weighted Sobolev spaces and in some cases also give or indicate various applications of them to other topics in analysis.\\
Consider the normed space $W^{k,p}(\Omega,\sigma)$ introduced on page 11 of \cite{Kufner} and also, with slightly different notation, on page 537 of \cite{KO}. In the case where $\Omega=U$ and $k=1$ and all the elements of the vector $\sigma=\left\{ \sigma_{\alpha}:\left|\alpha\right|\le1\right\} $ are chosen to be $\omega$ this space coincides isometrically with our space $W^{1,p}(U,\omega)$. (In \cite{Kufner} it is assumed that $p\in[1,\infty)$ and that the open set $\Omega$ is also connected.)\\
For most of our purposes in this paper we need the weighted Sobolev spaces with which we are working to be complete. Accordingly,
most of our results here will be formulated subject to the ``abstract'' requirement that the weighted Sobolev spaces being considered in them are complete. When seeking to apply these results we can keep in mind that there are various ``concrete'' conditions on a weight function $\omega$ which are sufficient to ensure that $W^{1,p}(U,\omega)$ is complete. One of these is a condition which appears in Theorem 1.11 of \cite[pp. 540--541]{KO} in the case where $p>1$. More relevantly for us, and as we shall
show below in Lemma \ref{lem:banach_space_and_loc}, $W^{1,p}(U,\omega)$ is complete whenever $\omega$ satisfies the compact boundedness condition. This is a condition which will be imposed anyway, also for other purposes, on the weight functions appearing in our main theorem \ref{thm:main}, and in Theorems \ref{thm:main_for_W_0}, \ref{thm:main_for_U_including_R_d} and \ref{thm:Stein_Weiss_for_W} which follow from it. 
\begin{remark}\label{rem:NotCompleteZZ}
We have not attempted to find examples of $p$, $U$ and $\omega$ for which $W^{1,p}(U,\omega)$ is not complete. However Example 1.12 on pp.
541--543 of \cite{KO} shows explicitly that a certain variant of $W^{1,p}(U,\omega)$, where \textit{different} weight functions are used for the weighted $L^{p}$ norms of the function and of its weak first order partial derivatives, can fail to be complete.
\end{remark}
In some of our results we will also be requiring the weighted Sobolev spaces $W^{1,p}(U,\omega)$ appearing in them to have the property that each of their elements can be approximated by a sequence of compactly supported $C^{\infty}$ or Lipschitz functions. Here too there are concrete conditions on $U$ and $\omega$ which suffice to imply such properties. The interested reader can find more information about such issues in \cite{Kufner} and \cite{KO}. 
\subsection{Main results}\label{subsec:main_resultsZZ}
As we've stated in \S \ref{subsec:goals}, our main goal in this work is to find settings in which an analogue of the Stein-Weiss theorem holds for weighted Sobolev spaces, or for suitable substitutes for these spaces.\\
In what follows, we will denote by $Lip_{c}\pa{U}$, or $Lip_{c}$ for short, the space of all Lipschitz functions $\varphi:U\to\mathbb{C}$ with compact support. To avoid any
ambiguity here, we should point out that $\varphi\in Lip_{c}(U)$ means that the closure of $\left\{ x\in U:\varphi(x)\ne0\right\}$ in $\mathbb{R}^{d}$ is a compact subset, not only of $\mathbb{R}^{d}$, but also of $U$.
\begin{definition}\label{def:W_0} 
Let $\omega$ be a weight function such that $W^{1,p}\left(U,\omega\right)$ is a Banach space, containing $Lip_{c}\left(U\right)$. Define the Banach space $W_{0}^{1,p}\left(U,\omega\right)$ to be the subspace of $W^{1,p}\pa{U,\omega}$ equipped with the same norm as that of $W^{1,p}\pa{U,\omega}$ which is the closure of $Lip_{c}\left(U\right)$ with respect to that norm.
\end{definition}
In particular we should mention that if $\omega$ satisfies the compact boundedness condition on $U$ then, as we shall show below in Lemma \ref{lem:banach_space_and_loc}, $W^{1,p}\left(U,\omega\right)$ has the two above mentioned properties required to make the above definition of $W_{0}^{1,p}\left(U,\omega\right)$ applicable.\\
As certain steps in the proof of our main theorem will show, $Lip_c$ is a very natural space to consider for our purposes. In addition, for many choices of the open set $U$, the subspace $C_c^\infty \pa{U}$ of $Lip_c(U)$ can be shown to have the same closure in $W^{1,p}\pa{U,\omega}$ norm as that of $Lip_c\pa{U}$.\\
We will also sometimes need to consider functions which satisfy the following variant of the Lipschitz condition.
\begin{definition}\label{def:LocallyLipschitzZZ}
A function $\phi:U\to\mathbb{C}$ is said to be \textit{locally Lipschitz} if, for each compact subset $K$ of $U$, the restriction of $\phi$ to $K$ is Lipschitz.
\end{definition}
\begin{remark}\label{rem:E2VZZ}
It is easy to verify that $\phi:U\to\mathbb{C}$ is locally Lipschitz if and only if its restriction to the set $V$ is Lipschitz whenever $V$ is an open set whose closure is a compact
subset of $U$.
\end{remark}
Let us now state our main theorem: 
\begin{theorem}\label{thm:main} 
Let $\omega_{0}$ and $\omega_{1}$ be weight functions on $U$ that satisfy the compact boundedness condition \eqref{eq:compact_boundedness}. Assume furthermore that the function $\rw:U\to(0,\infty)$ defined by
\begin{equation}\label{eq:rw}
r_{\omega}(x):=\frac{\omega_{0}(x)}{\omega_{1}(x)}
\end{equation}
is locally Lipschitz on $U$. For each $p\in [1,\infty)$ and $\theta\in(0,1)$, let $\W^p\pa{U,\theta,\rw}$ be the space  
$$\mathcal{W}^{p}\pa{U,\theta,\rw}=\left\{ \phi:U\rightarrow\mathbb{C}\;|\; \phi\in W^{1,p}\left(U,\omega_{\theta}\right),\; \phi\left\vert \nabla\log\left(r_{\omega}\right)\right\vert \in L^{p}\left(U,\omega_{\theta}\right)\right\} $$
with the norm 
\begin{equation}
\left\Vert \phi\right\Vert _{\mathcal{W}^{p}\left(U,\theta,\rw\right)}=\left(\left\Vert \phi\right\Vert _{W^{1,p}\left(U,\omega_{\theta}\right)}^{p}+\int_{U}\left\vert \phi(x)\right\vert ^{p}\left\vert \nabla\log\left(r_{\omega}(x)\right)\right\vert ^{p}\omega_{\theta}(x)dx\right)^{\frac{1}{p}},\label{eq:norm_on_W}
\end{equation}
and define $\mathcal{W}_{0}^{p}\left(U,\theta,r_{\omega}\right)$ to be the closure of $Lip_{c}\left(U\right)$ in $\mathcal{W}^{p}\left(U,\theta,\rw\right)$ with respect to this norm. Then, for every $p\in [1,\infty)$ and $\theta\in (0,1)$, we have that
\begin{equation}\label{eq:easy_inclusion_thm}
\left[W^{1,p}\left(U,\omega_{0}\right),W^{1,p}\left(U,\omega_{1}\right)\right]_{\theta}\subset W^{1,p}\left(U,\omega_{\theta}\right),
\end{equation}
and
\begin{equation}\label{eq:hard_inclusion_thm}
\mathcal{W}_{0}^{p}\left(U,\theta,r_{\omega}\right)\subset \left[W_{0}^{1,p}\left(U,\omega_{0}\right),W_{0}^{1,p}\left(U,\omega_{1}\right)\right]_{\theta}.
\end{equation}
Moreover, there exists a positive constant $C_p$ which depends only on $p$, such that, for every $\phi\in \mathcal{W}_{0}^{p}\left(U,\theta,r_{\omega}\right)$,
\begin{equation}\label{eq:main_norm}
\left\Vert \phi\right\Vert _{W^{1,p}\pa{U,\omega_{\theta}}}\leq\left\Vert \phi\right\Vert _{\left[W^{1,p}\pa{U,\omega_{0}},W^{1,p}\pa{U,\omega_{1}}\right]_{\theta}}\leq C_p\left\Vert \phi\right\Vert _{\mathcal{W}^{p}\pa{U,\theta,\rw}}
\end{equation}
 \end{theorem}
 \begin{remark}\label{rem:ValueOfCpZZ}
The exact value of the constant in \eqref{eq:main_norm} interests us rather less than the fact that it is independent of our choices of weight functions and $\theta$, and depends only on $p$. But we also note that the arguments that will eventually combine to give us a proof of Theorem \ref{thm:main} will also show that the said constant can be chosen to be 
\begin{equation}\label{eq:ValueOfCp}
C_{p}=\min_{\beta>0}2e^{\beta}\max\left\{ 1,\frac{e^{\beta}}{p\sqrt{2\beta e}}\right\} .
\end{equation}
\end{remark}
\begin{remark}\label{rem:chain_of_inclusionZZ}
From the definition of $\rpa{A_0,A_1}_{\theta}$ it is obvious (cf. part \eqref{item:inter_properties_inclusion_of_norms_and_subspaces} of Theorem \ref{thm:OtherProperties}) that 
$$\left[W_{0}^{1,p}\left(U,\omega_{0}\right),W_{0}^{1,p}\left(U,\omega_{1}\right)\right]_{\theta} \subset \left[W^{1,p}\left(U,\omega_{0}\right),W^{1,p}\left(U,\omega_{1}\right)\right]_{\theta}.$$
We can thus combine this inclusion together with \eqref{eq:easy_inclusion_thm} and \eqref{eq:hard_inclusion_thm} to obtain one chain of inclusions. 
\end{remark}
Although the non-negative function $\omega_{\theta}\left\vert \nabla\log\left(r_{\omega}\right)\right\vert ^{p}$ can assume the value $0$, it will be helpful for us to permit ourselves the following slight abuse of notation:
\begin{definition}\label{def:SemiNorm}
Let $L^{p}\left(U,\omega_{\theta}\left\vert \nabla\log\left(r_{\omega}\right)\right\vert ^{p}\right)$ denote the space of measurable functions $\phi:U\to\mathbb{C}$ for which the \textit{semi}-norm, which we shall denote by 
\begin{equation}\label{eq:SemiNorm}
\left\Vert \phi\right\Vert _{L^{p}\left(U,\omega_{\theta}\left\vert \nabla \log\left(r_{\omega}\right)\right\vert ^{p}\right)}=\left(\int_{\mathbb{R}^{d}}\left\vert \phi(x)\right\vert ^{p}\left\vert \nabla\log\left(r_{\omega}(x)\right)\right\vert ^{p}\omega_{\theta}(x)dx\right)^{\frac{1}{p}}, 
\end{equation}
is finite.
\end{definition}
While Theorem \ref{thm:main} is quite general, it might also be desirable to have a ``cleaner'' version of it which deals only with spaces of the kind ``$W^{1,p}$'' or only with spaces of the kind ``$W_{0}^{1,p}$'' rather than a mixture of these two kinds. Such versions can indeed be obtained when suitable additional conditions are imposed on $U$ and/or on the weight functions $\omega_{0}$ and $\omega_{1}$. The following two theorems give instances of this.
\begin{theorem}\label{thm:main_for_W_0} 
Let $U$, $\omega_0$ and $\omega_1$ be as in Theorem \ref{thm:main}. Assume furthermore that, for some $p\in[1,\infty)$, whenever $\phi$ is an element of $W_{0}^{1,p}\left(U,\omega_{0}\right)\cap W_{0}^{1,p}\left(U,\omega_{1}\right)$ there exists a sequence $\left\{ \phi_{n}\right\} _{n\in\mathbb{N}}$ of functions in $Lip_{c}(U)$ which converges to $\phi$ in $W^{1,p}\left(U,\omega_{0}\right)$ or in $W^{1,p}\left(U,\omega_{1}\right)$ and is a bounded sequence in the other space. Then, as well as the inclusions \eqref{eq:easy_inclusion_thm} and \eqref{eq:hard_inclusion_thm}, we also have that 
\begin{equation}\label{eq:inclusion_for_W_0_main_simplified}
\mathcal{W}_{0}^{p}\left(U,\theta,r_{\omega}\right)\subset\left[W_{0}^{1,p}\left(U,\omega_{0}\right),W_{0}^{1,p}\left(U,\omega_{1}\right)\right]_{\theta}\subset W_{0}^{1,p}\left(U,\omega_{\theta}\right)
\end{equation}
for that same value of $p$ and for all $\theta\in(0,1)$.
\end{theorem}
\begin{theorem}\label{thm:main_for_U_including_R_d}
Let $U$, $\omega_0$ and $\omega_1$ be as in Theorem \ref{thm:main}. Assume furthermore that, for some $p\in[1,\infty)$, 
\begin{equation}\label{eq:Wdensity01}
W_{0}^{1,p}\left(U,\omega_{j}\right)=W^{1,p}\left(U,\omega_{j}\right),\quad\quad j=0,1,
\end{equation}
and
\begin{equation}\label{eq:Wdensity02}
\mathcal{W}_{0}^{p}\left(U,\theta,r_{\omega},\right)=\mathcal{W}^{p}\left(U,\theta,r_{\omega}\right),\quad\quad\theta\in(0,1).
\end{equation}
Then, as well as the inclusions \eqref{eq:easy_inclusion_thm} and \eqref{eq:hard_inclusion_thm}, we also have that 
\begin{equation}\label{eq:main_for_U_including_R_d}
\mathcal{W}^{p}\left(U,\theta,r_{\omega}\right)\subset\left[W^{1,p}\left(U,\omega_{0}\right),W^{1,p}\left(U,\omega_{1}\right)\right]_{\theta}\subset W^{1,p}\left(U,\omega_{\theta}\right)
\end{equation}
for that same value of $p$ and for all $\theta\in(0,1)$.\\ 
In particular if $U=\R^d$ and if $\omega_0$ and $\omega_1$ satisfy the conditions imposed in Theorem \ref{thm:main}, then \eqref{eq:Wdensity01} and \eqref{eq:Wdensity02} both hold for every $p\in[1,\infty)$ and therefore so does \eqref{eq:main_for_U_including_R_d}.
\end{theorem}
With the help of Theorems \ref{thm:main_for_W_0} and \ref{thm:main_for_U_including_R_d} we are able to obtain the following theorem which identifies new conditions which suffice for a Stein-Weiss like theorem to hold for pairs of $W^{1,p}$ spaces: 
\begin{theorem}\label{thm:Stein_Weiss_for_W}
Let $U$ be a non-empty open subset of $\mathbb{R}^{d}$ and let $g:U\to\mathbb{R}$ be a Lipschitz function. Let $\omega_{0}$ and $\omega_{1}$ be weight functions on $U$ such that at least one of them satisfies the compact boundedness condition and 
$$
\omega_0(x)=\omega_1(x)e^{g(x)}.
$$
\begin{enumerate}[(i)]
\item\label{item:conditions_are_valid} Then the conditions of Theorem \ref{thm:main} are satisfied.
\item\label{item:SW_on_Rd} If $U=\mathbb{R}^{d}$ then
\begin{equation}\label{eq:SteinWeissOnRd}
\left[W^{1,p}\left(\mathbb{R}^{d},\omega_{0}\right),W^{1,p}\left(\mathbb{R}^{d},\omega_{1}\right)\right]_{\theta}=W^{1,p}\left(\mathbb{R}^{d},\omega_{\theta}\right)
\end{equation}
to within equivalence of norms, for every $p\in [1,\infty)$ and $\theta\in (0,1)$.
\item\label{item:SW_on_W0} If the open set $U$ and the above two weight functions $\omega_{0}$ and $\omega_{1}$ satisfy the conditions of Theorem \ref{thm:main_for_W_0} then 
$$ \left[W_{0}^{1,p}\left(U,\omega_{0}\right),W_{0}^{1,p}\left(U,\omega_{1}\right)\right]_{\theta}=W_{0}^{1,p}\left(U,\omega_{\theta}\right)$$
to within equivalence of norms, for every $p\in [1,\infty)$ and $\theta\in (0,1)$.
\item\label{item:SW_on_U} If the number $p$, the open set $U$ and the above two weight functions $\omega_{0}$ and $\omega_{1}$ satisfy the conditions of Theorem \ref{thm:main_for_U_including_R_d} then 
$$ \left[W^{1,p}\left(U,\omega_{0}\right),W^{1,p}\left(U,\omega_{1}\right)\right]_{\theta}=W^{1,p}\left(U,\omega_{\theta}\right)$$
to within equivalence of norms, for that value of $p$ and for every $\theta\in (0,1)$.
\end{enumerate}
\end{theorem}
\subsection{Structure of the paper}\label{subsec:structureZZ} In \S\ref{sec:pre} we will recall some basic definitions and facts about interpolation spaces, and in particular Calder\'on's complex interpolation spaces. In \S\ref{sec:simple_inclusion} we will show how a generalisation of the Stein-Weiss theorem easily implies the inclusion of $\rpa{W^{1,p}\pa{U,\omega_0},W^{1,p}\pa{U,\omega_1}}_{\theta} $ in $W^{1,p}\pa{U,\omega_\theta}$ and we shall also consider the analogous inclusion for the $W^{1,p}_0$ spaces. \S\ref{sec:difficult_inclusion} is the main section of this work. In it we will discuss the more difficult inclusion of $\W^p_{0}\pa{U,\theta,\rw}$ in $\rpa{W^{1,p}\pa{U,\omega_0},W^{1,p}\pa{U,\omega_1}}_{\theta} $, which will eventually yield our main result. With the help of some approximation theorems which we will present in \S\ref{sec:approximation}, we will prove our main theorem and several interesting consequences of it in \S\ref{sec:proof}. In \S \ref{sec:more_about_WW_p} we will discuss further properties of $\W^{p}\pa{U,\theta,\rw}$, such as the fact that it can sometimes happen that $\W^{p}\pa{U,\theta,\rw}$ is strictly smaller than $W^{1,p}\pa{U,\omega_\theta}$. In the final section \S\ref{sec:final} we will share some final thoughts about this, and future, work. These will include a discussion of potential applications to rates of convergence in evolution equations.\\
There is also an appendix at the end of the paper. It contains some auxiliary proofs which we felt should not hinder the flow of our presentation in the preceding sections.
\subsection*{Acknowledgement} A. Einav would like to thank Anton Arnold and Tobias W\"ohrer for stimulating discussions about using interpolation theory to estimate rates of convergence to equilibrium in Fokker-Planck equations, which led him to start investigating the main problem presented in this work.
\section{Preliminaries: Interpolation spaces and the complex interpolation method}\label{sec:pre}
This short section has been included for the convenience of those readers who may not be familiar with the theory of interpolation spaces, in particular of those which are generated by the \textit{complex interpolation method} (sometimes referred to simply as the \textit{complex method}). We briefly recall some basic definitions and some basic properties of these interpolation spaces. For more information about this subject we refer the reader to Calder\'on's fundamental paper \cite{Calder} or to \cite{BL} or many other excellent textbooks. The notations and definitions which we shall use are fairly standard.\footnote{But note that there are some differences between the notation used in \cite{Calder} and in \cite{BL}.}\\
The basic setting of interpolation theory is a pair of normed vector spaces $A_0$ and $A_1$ which have the property that there exists a Hausdorff topological vector space $\mathcal{A}$ satisfying 
$$A_0\subset \mathcal{A},\quad\quad A_1\subset \mathcal{A},$$
where the inclusions of $A_0$ and $A_1$ in $\mathcal{A}$ are linear and continuous. If $A_0$ and $A_1$ are also Banach spaces, then it is customary to say that $\pa{A_0,A_1}$ is a \textit{Banach couple}, and also to sometimes use the abbreviated notation $\overline{A}$ instead of $(A_0,A_1)$. 
\begin{definition}\label{def:interpolationZZ}
Let $\overline{A}=(A_0,A_1)$ be a Banach couple. A normed vector space $A$ is called an \textit{intermediate space} with respect to $(A_0,A_1)$ if 
$$A_0\cap A_1 \subset A \subset A_0+A_1$$
with continuous inclusions, where the norms on $A_0\cap A_1$ and $A_0+A_1$ are given by
$$\norm{a}_{A_0\cap A_1}= \max\pa{\norm{a}_{A_0},\norm{a}_{A_1}}$$
and
$$\norm{a}_{A_0+ A_1}= \inf\br{\norm{a_0}_{A_0}+\norm{a_1}_{A_1}\;|\; a_0+a_1=a}.$$
The notation $T:\left(A_{0},A_{1}\right)\to\left(A_{0},A_{1}\right)$ (or $T:\overline{A}\to\overline{A}$ for short) is taken to mean that $T$ is a linear operator which maps $A_{0}+A_{1}$ into itself and also satisfies $T(A_{0})\subset A_{0}$ and $T(A_{1})\subset A_{1}$ and is bounded from $A_{0}$ into $A_{0}$ and from $A_{1}$ into $A_{1}$.\\
An intermediate space $A$ is called an \textit{interpolation space} with respect to $(A_0,A_1)$ if it has the additional property that every linear operator $T:(A_0,A_1)\to(A_0,A_1)$ is also a bounded map of $A$ into itself. 
\end{definition}
We often need to consider a more elaborate version of the previous definition, for operators which map between the spaces of two possibly different Banach couples:
\begin{definition}\label{def:RelInterpZZ}
Let $\overline{A}=\left(A_{0},A_{1}\right)$ and $\overline{B}=\left(B_{0},B_{1}\right)$ be Banach couples. The notation $T:\left(A_{0},A_{1}\right)\to\left(B_{0},B_{1}\right)$ (or $T:\overline{A}\to\overline{B}$ for short) is taken to mean that $T$ is a linear operator which maps $A_{0}+A_{1}$ into $B_{0}+B_{1}$ and also satisfies $T(A_{0})\subset B_{0}$ and $T(A_{1})\subset B_{1}$ and is bounded from $A_{0}$ into $B_{0}$ and from $A_{1}$ into $B_{1}$. Let $A$ be an intermediate space with respect to $\left(A_{0},A_{1}\right)$ and let $B$ be an intermediate space with respect to $\left(B_{0},B_{1}\right)$. Then we say that $A$ and $B$ are relative interpolation spaces with respect to $\left(A_{0},A_{1}\right)$ and $\left(B_{0},B_{1}\right)$ if every linear operator $T :\left(A_{0},A_{1}\right)\to\left(B_{0},B_{1}\right)$ is a bounded map of $A$ into $B$.
\end{definition}
There are many methods for creating interpolation spaces with respect to a Banach couple and relative interpolation spaces with respect to a pair of Banach couples. In our work here we will only concern ourselves with the so-called \textit{complex method}, which was introduced and developed by Calder\'on in \cite{Calder}, and which we will now describe.\\ 
In what follows we will let $\St$ denote the infinite strip $\br{z\in \C\; | \; 0\leq \text{Re} z \leq 1}$ in the complex plane.
\begin{definition}\label{def:complex_method}
Let $\overline{A}=(A_0,A_1)$ be a Banach couple of complex Banach spaces. We define $\F\pa{A_0,A_1}$, also denoted by $\F\pa{\overline{A}}$, to be the space 
\begin{equation}\nonumber
\begin{split}
\F\pa{\overline{A}}=\Bigg\{f:\St\rightarrow &A_0+A_1\;|\; f\text{ is continuous and bounded on }\St\text{ and analytic on }\text{int}\pa{\St}\\
&t \rightarrow f(j+it)\text{ is a continuous bounded map of }\R\text{ into }A_j,\text{ for }j=0,1 \Bigg\},
\end{split}
\end{equation}
and define a norm on $\F\pa{\overline{A}}$ by
$$\norm{f}_{\F\pa{\overline{A}}}=\max \pa{\sup_{t\in\R}\norm{f(it)}_{A_0},\sup_{t\in\R}\norm{f(1+it)}_{A_1}}.$$
It can very readily be shown (using e.g., obvious very minor modifications of the argument in paragraph 22 of \cite[p. 129]{Calder}) that $\F\pa{\overline{A}}$ with the above norm is a Banach space, and we will make use of this fact.\\
For each given $\theta\in[0,1]$ we define $\rpa{A_0,A_1}_\theta$ to be the linear space of all elements of $a\in A_0+A_1$ for which there exists $f\in\F\pa{\overline{A}}$ such that $a=f(\theta)$. This space, when equipped with the norm 
$$\norm{a}_{[A_0,A_1]_\theta}=\inf \br{\norm{f}_{\F\pa{\overline{A}}}\;|\;f\in \F\pa{\overline{A}},\; f(\theta)=a},$$
is a Banach space.
\end{definition}
The basic and very useful interpolation property of the spaces defined by the previous definition is expressed by the following theorem of Calder\'on. (See \S4 of \cite[p. 115]{Calder}, also e.g., Theorem 4.1.2 of \cite[p. 88]{BL}.)
\begin{theorem}\label{thm:interpolation_thm_basic} 
Let $\overline{A}=(A_{0},A_{1})$ and $\overline{B}=\left(B_{0},B_{1}\right)$ be Banach couples of complex Banach spaces. Then $\left[A_{0},A_{1}\right]_{\theta}$ and $\left[B_{0},B_{1}\right]_{\theta}$ are relative interpolation spaces with respect to $\overline{A}$ and $\overline{B}$ for each $\theta\in(0,1)$. Furthermore, for each linear operator $T:\overline{A}\to\overline{B}$, the norms of $T$ as a operator from $A_{0}$ to $B_{0}$, from $A_{1}$ to $B_{1}$ and from $\left[A_{0},A_{1}\right]_{\theta}$ to $\left[B_{0},B_{1}\right]_{\theta}$ satisfy the inequality
$$\left\Vert T\right\Vert _{\mathcal{B}\left(\left[A_{0},A_{1}\right]_{\theta},\left[B_{0},B_{1}\right]_{\theta}\right)}\leq\left\Vert T\right\Vert _{\mathcal{B}\left(A_{0},B_{0}\right)}^{1-\theta}\left\Vert T\right\Vert _{\mathcal{B}\left(A_{1},B_{1}\right)}^{\theta}.$$
\end{theorem}
\begin{remark}
The above definition, describing the complex method for creating interpolation spaces, indicates why one may think of such spaces as 'in between' spaces. The original spaces $A_0$ and $A_1$ are somehow identified with elements of $\F\pa{\overline{A}}$ when restricted to $\text{Re} z=0$ or $\text{Re} z=1$. The $\theta -$interpolation space $\rpa{A_0,A_1}_\theta$ is identified with elements of $\F\pa{\overline{A}}$ restricted to the \textit{intermediate line} $\text{Re} z=\theta$.
\end{remark}
\begin{remark}\label{rem:difference_in_def}
Our definition here of the space $\F\pa{\overline{A}}$, which is often, but not always, the definition chosen to appear in papers about the complex interpolation method is in fact slightly different from the original definition appearing in \cite{Calder} in that we have not required that $\lim_{\left|t\right|\to\infty}\left\Vert f(j+it)\right\Vert _{A_{j}}=0$ for $j\in\left\{ 0,1\right\} $. However it is well known (implicit already in \cite{Calder} and mentioned e.g. in \cite[p.~1007]{Cwikel}) and very easy to check (by multiplying functions in $\mathcal{F}(\overline{A})$ by $e^{\beta(z-\theta)^{2}}$ for arbitrarily small positive $\beta$) that this difference does not effect the definition of $\left[A_{0},A_{1}\right]_{\theta}$ nor its norm.
\end{remark}
Here are some more properties of complex interpolation spaces which will be relevant for our purposes:
\begin{theorem}\label{thm:OtherProperties}
Let $\left(A_{0},A_{1}\right)$ be a Banach couple of complex Banach spaces and let $\theta$ be a number in $(0,1)$. Then 
\begin{enumerate}[(i)]
\item\label{item:inter_properties_dense} $A_{0}\cap A_{1}$ is a dense subspace of $\left[A_{0},A_{1}\right]_{\theta}$,\\ 
and
\item\label{item:inter_properties_norm_control} each element $a$ of $A_{0}\cap A_{1}$ satisfies the inequality 
\begin{equation}\label{eq:A-Theta}
\left\Vert a\right\Vert _{\left[A_{0},A_{1}\right]_{\theta}}\le\left\Vert a\right\Vert _{A_{0}}^{1-\theta}\left\Vert a\right\Vert _{A_{1}}^{\theta}.
\end{equation} 
\item\label{item:inter_properties_inclusion_of_norms_and_subspaces} Suppose, furthermore, that, for $j\in\left\{ 0,1\right\} $, $B_{j}$ is a closed subspace of $A_{j}$ equipped with the same norm as $A_{j}$. Then $\left(B_{0},B_{1}\right)$ is a Banach couple, and $\left[B_{0},B_{1}\right]_{\theta}\subset\left[A_{0},A_{1}\right]_{\theta}$, and each $b\in\left[B_{0},B_{1}\right]_{\theta}$ satisfies 
$$
\left\Vert b\right\Vert _{\left[A_{0},A_{1}\right]_{\theta}}\le\left\Vert b\right\Vert _{\left[B_{0},B_{1}\right]_{\theta}}.
$$  
\item\label{item:inter_properties_A} If $A_{0}=A_{1}$ with equality of norms, then $\left[A_{0},A_{1}\right]_{\theta}=A_{0}=A_{1}$ with equality of norms.
\end{enumerate}
\end{theorem}
For the statement and proof of \eqref{item:inter_properties_dense} we refer to page 116 and then pages 132--133 of \cite{Calder}, or to Theorem 4.2.2 and Lemma 4.2.3 of \cite[pp. 91--92]{BL}. The rather obvious, immediate and well known proof of \eqref{item:inter_properties_norm_control} proceeds via the function $f(z)=a/\left\Vert a\right\Vert _{A_{0}}^{1-z}\left\Vert a\right\Vert _{A_{1}}^{z}$ which is a norm one element of $\mathcal{F}\left(A_{0},A_{1}\right)$. Part \eqref{item:inter_properties_inclusion_of_norms_and_subspaces} is an immediate consequence of Definition \ref{def:complex_method}. Part \eqref{item:inter_properties_A} is an obvious consequence of Definition \ref{def:complex_method} and part \eqref{item:inter_properties_dense}.
\begin{remark}\label{rem:WhyInfinityIsTrivial}
We note that part \eqref{item:inter_properties_A} of this theorem combined with Remark \ref{rem:Unweighted}, justifies the claim made in Remark \ref{rem:InfinityIsTrivial}. \end{remark}
With the basics of interpolation theory and the complex method in hand, we can start proving our main theorem.
\section{The simpler inclusion}\label{sec:simple_inclusion}
Before beginning our systematic study of the interpolation spaces 
$$\left[W^{1,p}\left(U,\omega_{0}\right),W^{1,p}\left(U,\omega_{1}\right)\right]_{\theta}$$ 
we of course need to be sure that 
\begin{equation}\label{eq:IsBanachCouple}
\left(W^{1,p}\left(U,\omega_{0}\right),W^{1,p}\left(U,\omega_{1}\right)\right)\mbox{ is a Banach couple,}
\end{equation}
when we either assume that $W^{1,p}\left(U,\omega_{0}\right)$ and $W^{1,p}\left(U,\omega_{1}\right)$ are both Banach spaces or impose explicit conditions on $U$, $\omega_{0}$ and $\omega_{1}$ which guarantee that. In fact (\ref{eq:IsBanachCouple}) holds under such assumptions, because both of the spaces $W^{1,p}\left(U,\omega_{0}\right)$ and $W^{1,p}\left(U,\omega_{1}\right)$ are continuously embedded in the Banach space $L^{p}(U,\omega)$ when $\omega$ is chosen to be the weight function $\omega=\min\left\{ \omega_{0},\omega_{1}\right\}$.\\
Our main theorem, Theorem \ref{thm:main}, establishes two inclusions. The reader has probably already noticed that one of them, \eqref{eq:easy_inclusion_thm}, is actually an obvious and almost immediate consequence of the Stein-Weiss theorem, or of a slightly generalised version of it due to Calder\'on. Nevertheless, we will provide an explicit proof of that inclusion in this section. We will also consider the case where this inclusion can be generalised to the $W_0^{1,p}$ spaces.\\
Here is the version that we need of Calder\'on's generalization of the Stein-Weiss theorem.
\begin{theorem}\label{thm:calder}
The formula
\begin{equation}\label{eq:calder_thm}
\rpa{L^p\pa{U,\omega_0,\C^m},L^p\pa{U,\omega_1,\C^m}}_{\theta}=L^p\pa{U,\omega_\theta,\C^m},
\end{equation}
holds with equality of norms for every $m\in\N$, every $p\in [1,\infty)$, every $\theta\in (0,1)$, every open subset $U$ of $\R^d$, and all pairs of weight functions $\omega_0$ and $\omega_1$ on $U$, when $\omega_\theta$ is defined by \eqref{eq:omegatheta}.
\end{theorem}
\begin{remark}
In fact the formula (\ref{eq:calder_thm}) is a special case of the isometric formula 
\begin{equation}\label{eq:GenCalderThm}
\left[L^{p_{0}}(U,\omega_{0},\mathbb{C}^{m}),L^{p_{1}}\left(U,\omega_{1},\mathbb{C}^{m}\right)\right]_{\theta}=L^{p_\theta}\left(U,\omega_{\theta,p_\theta},\mathbb{C}^{m}\right)
\end{equation}
 which holds for every $p_{0}$ and $p_{1}$ in $[1,\infty)$ and each $\theta\in(0,1)$ when we take $p_\theta$ and $\omega_{\theta,p_\theta}$ to be the number $p$ and the function $\omega$ specified by \eqref{eq:pwtheta}. This can be proved by a straightforward variant of the proof of Theorem 2 in \cite{Stein}. It can also be seen to be a very special case of a result\footnote{It is a special case of the isometric formula  $\left[X_{0}(B_{0}),X_{1}(B_{1})\right]_{\theta}=X_{0}^{1-\theta}X_{1}^{\theta}\left(\left[B_{0},B_{1}\right]_{\theta}\right)$ which, as stated as the second part of the result labelled ``i)''
on p. 125 of \cite{Calder}, holds for all choices of the Banach couple $\left(B_{0},B_{1}\right)$ whenever $X_{0}$ and $X_{1}$ are Banach lattices of measurable functions on the same underlying $\sigma$-finite measure space and $X_{0}^{1-\theta}X_{1}^{\theta}$ has a certain property. That property holds in our case in view of the Lebesgue Dominated Convergence Theorem.}  in \cite{Calder}.
\end{remark}
We shall now use this theorem to easily show the aforementioned inclusion, as well as the associated norm inequality.
\begin{theorem}\label{thm:easy_inclusion}
Let $U$ be an open subset of $\mathbb{R}^{d}$ and suppose that $\omega_{0}$ and $\omega_{1}$ are weight functions which both satisfy the compact boundedness condition on $U$. Then, for each $p\in[1,\infty)$, we obtain that $\left(W^{1,p}\left(U,\omega_{0}\right),W^{1,p}\left(U,\omega_{1}\right)\right)$ is a Banach couple and that, for each $\theta\in(0,1)$, its complex interpolation spaces satisfy 
$$
\rpa{W^{1,p}\pa{U,\omega_0},W^{1,p}\pa{U,\omega_1}}_{\theta} \subset W^{1,p}\pa{U,\omega_\theta},
$$
and 
\begin{equation}\label{eq:easy_inclusion_norm}
\norm{\phi}_{W^{1,p}\pa{U,\omega_\theta}} \leq \norm{\phi}_{\rpa{W^{1,p}\pa{U,\omega_0},W^{1,p}\pa{U,\omega_1}}_{\theta} },
\end{equation}
for every $\phi\in \rpa{W^{1,p}\pa{U,\omega_0},W^{1,p}\pa{U,\omega_1}}_{\theta}$.
\end{theorem}
\begin{proof}
In view of Lemma \ref{lem:banach_space_and_loc}, $W^{1,p}\left(U,\omega_{0}\right)$ and $W^{1,p}\left(U,\omega_{1}\right)$ are both Banach spaces, and therefore also, as explained at the beginning of this section, $\left(W^{1,p}\left(U,\omega_{0}\right),W^{1,p}\left(U,\omega_{1}\right)\right)$ is a Banach couple.\\
Define a linear operator 
$$T:W^{1,p}\pa{U,\omega_0}+W^{1,p}\pa{U,\omega_1}\rightarrow L^p\pa{U,\omega_0,\C^{d+1}}+ L^p\pa{U,\omega_1\C^{d+1}}$$
by 
$$T\phi:=\pa{\phi,\nabla \phi}.$$
By the definition of $W^{1,p}\pa{U,\omega}$ we have that $T\phi$ is a norm $1$ map from $W^{1,p}\pa{U,\omega_j}$ into $L^p\pa{U,\omega_j,\C^{d+1}}$ for $j=0,1$. Using the standard interpolation theorem, Theorem \ref{thm:interpolation_thm_basic}, and Theorem \ref{thm:calder}, we conclude that 
$$T:\rpa{W^{1,p}\pa{U,\omega_0},W^{1,p}\pa{U,\omega_1}}_{\theta} \rightarrow L^p\pa{U,\omega_\theta,\C^{d+1}}$$
is bounded with norm less than, or equal to, $1$. Thus, if $\phi\in \rpa{W^{1,p}\pa{U,\omega_0},W^{1,p}\pa{U,\omega_1}}_{\theta}$ then $\pa{\phi,\nabla \phi}\in L^p\pa{U,\omega_\theta,\C^{d+1}}$, i.e. $\phi\in W^{1,p}\pa{U,\omega_\theta}$. Moreover,  
$$\norm{\phi}_{W^{1,p}\pa{U,\omega_\theta}}=\norm{\pa{\phi,\nabla \phi}}_{L^{p}\pa{U,\omega_\theta,\C^{d+1}}}=\norm{T\phi}_{L^{p}\pa{U,\omega_\theta,\C^{d+1}}}$$
$$=\norm{T\phi}_{\rpa{L^{p}\pa{U,\omega_0,\C^{d+1}},L^{p}\pa{U,\omega_1,\C^{d+1}}}_{\theta}}
\leq  \norm{\phi}_{\rpa{W^{1,p}\pa{U,\omega_0},W^{1,p}\pa{U.\omega_1}}_{\theta}}.$$
This concludes the proof of this theorem. Thus it also establishes (\ref{eq:easy_inclusion_thm}) and the left side of the norm inequality (\ref{eq:main_norm}) in the statement of Theorem \ref{thm:main}.
\end{proof}
\begin{remark}\label{rem:Different_values_of_pZZ}
It should perhaps also be pointed out that the proof of Theorem \ref{thm:easy_inclusion} does not require the weight functions $\omega_{0}$ and $\omega_{1}$ to satisfy the
compact boundedness condition. In fact it suffices if they have any other property which ensures that $W^{1,p}(U,\omega_{0})$ and $W^{1,p}(U,\omega_{1})$ are both Banach spaces. (For example, some such properties are considered in the latter part of Lemma \ref{lem:banach_space_and_loc}.) Essentially the same proof as for Theorem \ref{thm:easy_inclusion} also shows that, for
any $p_{0}$ and $p_{1}$ in $[1,\infty)$, if $W^{1,p_{0}}(U,\omega_{0})$ and $W^{1,p_{1}}(U,\omega_{1})$ are both Banach spaces then they satisfy the norm one inclusion 
$$\left[W^{1,p_{0}}(U,\omega_{0}),W^{1,p_{1}}\left(U,\omega_{1}\right)\right]_{\theta}\subset W^{1,p}(U,\omega)$$
for each $\theta\in(0,1)$, where $p$ and $\omega$ are as in \eqref{eq:pwtheta}. This is an immediate consequence of the isometry (\ref{eq:GenCalderThm}). We have not yet considered under what conditions we can use some version of our approach to obtain the reverse of this inclusion when $p_{0}\not=p_{1}$. Such a result would complement the particular case considered in Formula (5.3) of Theorem 4 of \cite[p. 208]{Lofstrom} in which $U=\mathbb{R}^{d}$ and $\omega_{0}$, $\omega_{1}$ and $\omega_{\theta}$ are all required to be so-called polynomially regular weight functions. 
\end{remark}

We still have one more task in this section, which is to consider under what conditions we can replace the spaces $W^{1,p}(U,\omega)$ by their subspaces $W_{0}^{1,p}(U,\omega)$ in the inclusion \eqref{eq:easy_inclusion_thm}. We do this in the following theorem which will lead us eventually to Theorem \ref{thm:main_for_W_0}.

\begin{theorem} \label{thm:SimultaneousApprox}
Let $U$ be an open subset of $\mathbb{R}^{d}$ and suppose that $\omega_{0}$ and $\omega_{1}$ are weight functions which both satisfy
the compact boundedness condition on $U$. Suppose further that, for some value of $p\in[1,\infty)$,  whenever $\phi$ is an element of $W_{0}^{1,p}\left(U,\omega_{0}\right)\cap W_{0}^{1,p}\left(U,\omega_{1}\right)$, then there exists a sequence $\left\{ \phi_{n}\right\} _{n\in\mathbb{N}}$ of functions in $Lip_{c}(U)$ which converges to $\phi$ in $W^{1,p}\left(U,\omega_{0}\right)$ or $W^{1,p}\left(U,\omega_{1}\right)$ and is a bounded sequence in the other space. Then, for that value of $p$ and for each $\theta\in(0,1)$, 
$$
\left[W_{0}^{1,p}\left(U,\omega_{0}\right),W_{0}^{1,p}\left(U,\omega_{1}\right)\right]_{\theta}\subset W_{0}^{1,p}\left(U,\omega_{\theta}\right)
$$
and, furthermore,
\begin{equation}\label{eq:easy_inequality_W_0}
\left\Vert \phi\right\Vert _{W_{0}^{1,p}\left(U,\omega_{\theta}\right)}\le\left\Vert \phi\right\Vert _{\left[W_{0}^{1,p}\left(U,\omega_{0}\right),W_{0}^{1,p}\left(U,\omega_{1}\right)\right]_{\theta}}
\end{equation}
for each $\phi\in\left[W_{0}^{1,p}\left(U,\omega_{0}\right),W_{0}^{1,p}\left(U,\omega_{1}\right)\right]_{\theta}$.
\end{theorem}
\begin{remark} \label{rem:SimpleConditionZZ}We
suspect that, given a pair of weight functions $\omega_{0}$ and $\omega_{1}$ which are known to both satisfy the compact boundedness condition, quite mild additional conditions might suffice to ensure that they also satisfy the second hypothesis required in Theorem \ref{thm:SimultaneousApprox}. Perhaps it is even not necessary to impose any additional conditions for this to happen. Pending further investigation of this, we simply point out that this second hypothesis is obviously satisfied whenever one of the weight functions $\omega_{0}$ and $\omega_{1}$ is dominated by some constant multiple of the other. In that case there even exists a sequence $\left\{ \phi_{n}\right\} _{n\in\mathbb{N}}$ in $Lip_c\pa{U}$ which converges to $\phi$ in \textit{both} $W^{1,p}\left(U,\omega_{0}\right)$ and $W^{1,p}\left(U,\omega_{1}\right)$.
\end{remark} 
\begin{proof}[Proof of Theorem \ref{thm:SimultaneousApprox}]
As pointed out in Remark \ref{rem:CBC-obvious}, the fact that $\omega_{0}$ and $\omega_{1}$ both satisfy the compact boundedness condition immediately implies that $\omega_{\theta}$ also satisfies this condition for each $\theta\in(0,1)$. Consequently, by Lemma \ref{lem:banach_space_and_loc}, $W^{1,p}\left(U,\omega\right)$ is a Banach space which contains $Lip_{c}\left(U\right)$ whenever $\omega$ is any one of the weight functions $\omega_{0}$, $\omega_{1}$ or $\omega_{\theta}$. This enables us to define $W_{0}^{1,p}(U,\omega)$ in accordance with Definition \ref{def:W_0} for each of these choices of $\omega$.
Then the relevant definitions immediately imply (cf. part \eqref{item:inter_properties_inclusion_of_norms_and_subspaces} of Theorem \ref{thm:OtherProperties}) that $\left(W_{0}^{1,p}(U,\omega_{0}),W_{0}^{1,p}(U,\omega_{1})\right)$ is a Banach couple, and that 
$$
\left[W_{0}^{1,p}(U,\omega_{0}),W_{0}^{1,p}(U,\omega_{1})\right]_{\theta}\subset\left[W^{1,p}(U,\omega_{0}),W^{1,p}(U,\omega_{1})\right]_{\theta}
$$ 
 and that 
\begin{equation}\label{eq:W_0-InequalityB}
\left\Vert \phi\right\Vert _{\left[W^{1,p}(U,\omega_{0}),W^{1,p}(U,\omega_{1})\right]_{\theta}}\le\left\Vert \phi\right\Vert _{\left[W_{0}^{1,p}(U,\omega_{0}),W_{0}^{1,p}(U,\omega_{1})\right]_{\theta}}.
\end{equation}
for each $\phi\in\left[W_{0}^{1,p}(U,\omega_{0}),W_{0}^{1,p}(U,\omega_{1})\right]_{\theta}$.\\
By definition, $Lip_{c}(U)$ is contained in $W_{0}^{1,p}(U,\omega_{0})$ and in $W_{0}^{1,p}(U,\omega_{1})$ and therefore also in $W_{0}^{1,p}(U,\omega_{0})\cap W_{0}^{1,p}(U,\omega_{1})$. This implies (cf. part \eqref{item:inter_properties_dense} of Theorem \ref{thm:OtherProperties}) that 
$$ Lip_{c}(U)\subset\left[W_{0}^{1,p}(U,\omega_{0}),W_{0}^{1,p}(U,\omega_{1})\right]_{\theta}.$$
In view of this inclusion and since we also know that $Lip_{c}(U)$ is also contained in $W_{0}^{1,p}(U,\omega_{\theta})$, we are now ready to proceed to our next step of the proof, which will be to show that the inequality \eqref{eq:easy_inequality_W_0} holds for every $\phi\in Lip_{c}(U)$. By definition $\left\Vert \phi\right\Vert _{W_{0}^{1,p}(U,\omega)}=\left\Vert \phi\right\Vert _{W^{1,p}(U,\omega)}$ for such functions $\phi$ and this, combined with the inequality \eqref{eq:easy_inclusion_norm} of Theorem \ref{thm:easy_inclusion} and \eqref{eq:W_0-InequalityB}, indeed accomplishes this step.\\
In view of the standard fact recalled in part \eqref{item:inter_properties_dense} of Theorem \ref{thm:OtherProperties}, we can assert that $W^{1,p}_0\pa{U,\omega_{0}}\cap W^{1,p}_0\pa{U,\omega_{1}}$ is dense in $\left[W_{0}^{1,p}\left(U,\omega_{0}\right),W_{0}^{1,p}\left(U,\omega_{1}\right)\right]_{\theta}$. Therefore, since $W^{1,p}_0\pa{U,\omega_\theta}$ is a Banach space, the proof of this theorem will be complete if we show that each element $\phi$ in $W_{0}^{1,p}\left(U,\omega_{0}\right)\cap W_{0}^{1,p}\left(U,\omega_{1}\right)$ is also in $W^{1,p}_0\pa{U,\omega_\theta}$ and satisfies the inequality \eqref{eq:easy_inequality_W_0}.\\
Let $\phi$ be an arbitrary element of $W_{0}^{1,p}\left(U,\omega_{0}\right)\cap W_{0}^{1,p}\left(U,\omega_{1}\right)$. We keep in mind that Theorem \ref{thm:easy_inclusion} combined with parts \eqref{item:inter_properties_dense} and \eqref{item:inter_properties_inclusion_of_norms_and_subspaces} of Theorem \ref{thm:OtherProperties} guarantee
that $\phi$ must also be in $W^{1,p}(U,\omega_{\theta})$. Now let $\left\{ \phi_{n}\right\} _{n\in\mathbb{N}}$ be a sequence in $Lip_{c}(U)$ with the properties guaranteed by the hypotheses of the present theorem, so that at least one of the two numerical sequences $\left\{ \left\Vert \phi_{n}-\phi\right\Vert _{W^{1,p}\left(U,\omega_{0}\right)}\right\} _{n\in\mathbb{N}}$ and $\left\{ \left\Vert \phi_{n}-\phi\right\Vert _{W^{1,p}\left(U,\omega_{1}\right)}\right\} _{n\in\mathbb{N}}$
converges to $0$ and both are bounded. We now use the result about the complex interpolation method which is recalled in part \eqref{item:inter_properties_norm_control} and inequality \eqref{eq:A-Theta} of Theorem \ref{thm:OtherProperties}. In our current context the inequality \eqref{eq:A-Theta} becomes
\begin{align*}
\left\Vert \phi_{n}-\phi\right\Vert _{\left[W_{0}^{1,p}\left(U,\omega_{0}\right),W_{0}^{1,p}\left(U,\omega_{1}\right)\right]_{\theta}} & \le\left\Vert \phi_{n}-\phi\right\Vert _{W_{0}^{1,p}\left(U,\omega_{0}\right)}^{1-\theta}\cdot\left\Vert \phi_{n}-\phi\right\Vert _{W_{0}^{1,p}\left(U,\omega_{1}\right)}^{\theta}\\
 & =\left\Vert \phi_{n}-\phi\right\Vert _{W^{1,p}\left(U,\omega_{0}\right)}^{1-\theta}\cdot\left\Vert \phi_{n}-\phi\right\Vert _{W^{1,p}\left(U,\omega_{1}\right)}^{\theta}
\end{align*}
which gives us that
\begin{equation}\label{eq:phi_n_converges}
\lim_{n\to0}\left\Vert \phi-\phi_{n}\right\Vert _{\left[W_{0}^{1,p}\left(U,\omega_{0}\right),W_{0}^{1,p}\left(U,\omega_{1}\right)\right]_{\theta}}=0,
\end{equation}
and therefore also (cf. part \eqref{item:inter_properties_inclusion_of_norms_and_subspaces} of Theorem \ref{thm:OtherProperties}) that $\lim_{n\to0}\left\Vert \phi-\phi_{n}\right\Vert _{\left[W^{1,p}\left(U,\omega_{0}\right),W^{1,p}\left(U,\omega_{1}\right)\right]_{\theta}}=0$. This in turn implies, by Theorem \ref{thm:easy_inclusion} that
\begin{equation}\label{eq:WillBeInW_0}
\lim_{n\to0}\left\Vert \phi-\phi_{n}\right\Vert _{W^{1,p}\left(U,\omega_{\theta}\right)}=0
\end{equation}
and therefore $\phi$ is not only an element of $W^{1,p}(U,\omega_{\theta})$ but also of $W_{0}^{1,p}(U,\omega_{\theta})$. So we can deduce from \eqref{eq:WillBeInW_0} and Definition \ref{def:W_0} that 
$$\left\Vert \phi\right\Vert _{W_{0}^{1,p}\left(U,\omega_{\theta}\right)}=\left\Vert \phi\right\Vert _{W^{1,p}\left(U,\omega_{\theta}\right)}=\lim_{n\to0}\left\Vert \phi_{n}\right\Vert _{W^{1,p}\left(U,\omega_{\theta}\right)}=\lim_{n\to0}\left\Vert \phi_{n}\right\Vert _{W_{0}^{1,p}\left(U,\omega_{\theta}\right)}.$$
Since an earlier step of this proof gives us that each $\phi_{n}$ satisfies \eqref{eq:easy_inequality_W_0}, the above limit must be less than or equal to the limit $\lim_{n\to\infty}\left\Vert \phi_{n}\right\Vert _{\left[W_{0}^{1,p}\left(U,\omega_{0}\right),W_{0}^{1,p}\left(U,\omega_{1}\right)\right]_{\theta}}$
which exists and equals $\left\Vert \phi\right\Vert _{\left[W_{0}^{1,p}\left(U,\omega_{0}\right),W_{0}^{1,p}\left(U,\omega_{1}\right)\right]_{\theta}}$in
view of \eqref{eq:phi_n_converges}. Thus we have shown that our arbitrary element $\phi$ of $W_{0}^{1,p}\left(U,\omega_{0}\right)\cap W_{0}^{1,p}\left(U,\omega_{1}\right)$, as well as being in $W_{0}^{1,p}(U,\omega_{\theta})$, also satisfies \eqref{eq:easy_inequality_W_0}. As already explained above, this suffices to complete the proof of our theorem.
\end{proof}
In the next section we will make the significant step towards showing the other inclusion and inequality of our main theorem. This, however, will not be as simple. 

\section{Towards the difficult inclusion}\label{sec:difficult_inclusion}
The goal of this section is to consider the relationship between the two spaces $\W^p_0\pa{U,\theta,\rw}$ and $\rpa{W^{1,p}_0\pa{U,\omega_0},W^{1,p}_0\pa{U,\omega_1}}_{\theta}$. We do not find it surprising that in general we cannot obtain an analogous relationship between the ``full'' spaces $\W^p\pa{U,\theta,\rw}$ and $\rpa{W^{1,p}\pa{U,\omega_0},W^{1,p}\pa{U,\omega_1}}_{\theta}$, and can only deal with their above mentioned smaller but ``significant'' respective subspaces. However, as we have mentioned in \S\ref{sec:into}, in many cases these smaller spaces are either everything, or ``big'' enough to warrant a theorem on their own.\\
We begin with a definition which specifies some conditions which may or may not be satisfied by a given pair of weight functions $\omega_0$ and $\omega_1$. As we shall see, for most of the results of this section, we will require our weight functions to satisfy these conditions. 
\begin{definition}\label{def:admissibility}
We say that a pair of weight functions $\pa{\omega_0,\omega_1}$ is a $(\theta,p)-$\textit{admissible pair of weight functions} on $U$ (or, in short, a \textit{$(\theta,p)-$admissible pair}) for some $\theta\in(0,1)$ and $p\in[1,\infty]$ if the following holds:
\begin{enumerate}[(i)]
\item\label{item:W_is_banach_with_lip} $W^{1,p}\pa{U,\omega_0}$, $W^{1,p}\pa{U,\omega_1}$ and $W^{1,p}\pa{U,\omega_\theta}$ are Banach spaces that contain $Lip_c\pa{U}$.
\item\label{item:bounded_gradient} The function $\log\pa{\rw(x)}=\log\pa{\frac{\omega_0(x)}{\omega_1(x)}}$ is locally Lipschitz on $U$.
\end{enumerate} 
\end{definition} 
\begin{remark} \label{rem:AboutAdmissiblePairsZZ}
In view of the inequality $\omega_{0}^{1-\theta}\omega_{1}^{\theta}\le\max\left\{ \omega_{0},\omega_{1}\right\} $ it is easy to see that if $W^{1,p}(U,\omega_{0})$ and $W^{1,p}(U,\omega_{1})$ both contain $Lip_{c}(U)$, then in fact it follows automatically that $W^{1,p}(U,\omega_{\theta})$ also contains $Lip_{c}(U)$ for every $\theta\in(0,1)$. 
\end{remark}

From this point onward, till the end of this section, we will assume that we are dealing with a given fixed open set $U$ and often we will not mention $U$ in the notation that we use.
\begin{remark}\label{rem:W_(p,rw)_contains_lipZZ}
It is simple to observe that if $\pa{\omega_0,\omega_1}$ forms a $(\theta,p)-$admissible pair for some $\theta\in(0,1)$ and $p\in[1,\infty)$, then $\W^p\pa{\theta,\rw}$ also contains $Lip_c$. (Relevant definitions, i.e. of $\W^p\pa{\theta,\rw}$ and of $\rw$, appear in the statement of Theorem \ref{thm:main}.) Indeed, given any $\phi\in Lip_{c}$, we first note that it is contained in $W^{1,p}(\omega_{\theta})$, in view of condition \eqref{item:W_is_banach_with_lip} of Definition \ref{def:admissibility}. Then the other requirement for the membership of $\phi$ in $\mathcal{W}^p\pa{\theta,r_{\omega}}$ follows immediately from condition \eqref{item:bounded_gradient} of the same definition, which ensures that the auxiliary function $\phi(x)\left\vert \nabla\log\left(r_{\omega}(x)\right)\right\vert $ is continuous. Since that function also has compact support, it is obviously an element of $L^{p}(\omega_{\theta})$ as required.
\end{remark}
\begin{remark}
In most of our main results in this paper we require $\omega_{0}$ and $\omega_{1}$ to both satisfy the compact boundedness condition. In view of Remark \ref{rem:CBC-obvious}, Lemma \ref{lem:banach_space_and_loc} and Lemma \ref{lem:lip_in_W_p}, this implies that \eqref{item:W_1p_is_banach} holds for all $\theta\in(0,1)$ and $p\in[1,\infty)$.
From this we see that Definition \ref{def:admissibility}, and the results in this section which use it, are somewhat more general than they need to be for obtaining those main results. But this greater generality may be useful in future investigations where more general weight functions are considered. 
\end{remark}
Our main goal in this section is to prove the following result:
\begin{theorem}\label{thm:difficult_inclusion_on_W_0}
Let $\pa{\omega_0,\omega_1}$ form a $(\theta,p)-$admissible pair for some $\theta\in(0,1)$ and $p\in[1,\infty)$. Then, we have that 
\begin{equation}\label{eq:difficult_inclusion_subset}
\W^p_0\pa{\theta,\rw} \subset \rpa{W^{1,p}_0\pa{\omega_0},W^{1,p}_0\pa{\omega_1}}_{\theta}
\end{equation}
Moreover, there exists a positive constant $C_p$ which depends only on $p$, such that 
\begin{equation}\label{eq:difficult_inclusion_norm}
\norm{\phi}_{\rpa{W^{1,p}\pa{\omega_0},W^{1,p}\pa{\omega_1}}_{\theta} } \leq C_p \norm{\phi}_{\W^p\pa{\theta,\rw}} \quad\quad \text{for every }\;\phi\in \W^p_0\pa{\theta,\rw}.
\end{equation}
\end{theorem}
The proof of Theorem \ref{thm:difficult_inclusion_on_W_0} has two main parts and will occupy the rest of this section. In the first part, for each given $\phi\in Lip_c$, we will construct an element $f^{\phi}$ of $\mathcal{F}\left(W_{0}^{1,p}\left(\omega_{0}\right),W_{0}^{1,p}\left(\omega_{1}\right)\right)$ which satisfies $f^{\phi}(\theta)=\phi$ and also the norm estimates necessary for showing that $\phi$ satisfies \eqref{eq:difficult_inclusion_norm}. Then, in the second part, we will extend this construction to elements $\phi$ which are in the closure of $Lip_c$ in $\W^p\pa{\theta,\rw}$.
\begin{definition}\label{def:f_phiZZ}
Given $\theta\in(0,1)$, an arbitrary function $\phi\in Lip_c$ and a number $\beta>0$, we define the function
\begin{equation}\label{eq:def_of_f_phi}
f^{\phi}_{\beta,\theta}(x,z)=e^{\frac{\pa{z-\theta}}{p}\log \pa{\rw(x)}+\beta(z-\theta)^2}\phi(x),\quad\quad x\in U,\;z\in \St.
\end{equation}
(Recall, cf. \eqref{eq:rw} that $r_{\omega}$ denotes the ratio $\omega_{0}/\omega_{1}$.) Thus $f_{\beta,\theta}^{\phi}$ denotes the map from $\mathbb{S}$ into the space of measurable functions on $U$, which is such that the image $f_{\beta,\theta}^{\phi}(z)$ of each $z\in\mathbb{S}$ is the function of $x$ defined by \eqref{eq:def_of_f_phi}.
\end{definition}
The reader might be surprised by the inclusion of another parameter $\beta>0$ in the formula \eqref{eq:def_of_f_phi}. This parameter appears in the exponential factor $e^{\beta(z-\theta)^{2}}$ in the formula for $f_{\beta,\theta}^{\phi}$, and we need that factor in order to control the size of the $f_{\beta,\theta}^{\phi}(x,z)$ as $z$ tends to $\infty$ on the strip $\mathbb{S}$. (The same function $e^{\beta(z-\theta)^{2}}$ has a quite similar role elsewhere, as mentioned in Remark \ref{rem:difference_in_def}.) As will be evident in the course of the proofs of the results presented in this section, $\beta$ could be chosen arbitrarily, but we can also note that there is one choice of $\beta$ which will minimize the value that we can obtain for the above mentioned constant.\\
For convenience, we will often drop the subscripts of $\beta$ and $\theta$ and accordingly write $f^\phi$ to denote the function $f_{\beta,\theta}^\phi$.\\
We begin our preparations for achieving the first of the above mentioned main parts of our proof by recalling a rather standard definition and mentioning an important theorem which we will use in several places. 
\begin{definition}\label{def:LocalSobolev}
For each $p\in[1,\infty]$ and each open subset $U$ of $\mathbb{R}^{d}$, the space $W_{loc}^{1,p}(U)$ consists of all measurable functions $f:U\to\mathbb{C}$ which have a weak gradient $\nabla f$ in $U$ and for which, for every compact subset $K$ of $U$, the functions $\chi_{K}f$ and $\chi_{K}\nabla f$ are elements, respectively, of the (unweighted) spaces $L^{p}(U)$ and of $L^{p}(U,\mathbb{C}^{d})$.
\end{definition}
We remark that it is quite straightforward, but perhaps a little tedious, to check that a function $f:U\to\mathbb{C}$ is an element of $W_{loc}^{1,p}(U)$ if and only if, for every open subset $V$ of $U$ whose closure is a compact subset of $U$, the restriction of $f$ to $V$ is an element of $W^{1,p}(V)$. We refer the interested reader to the appendix for the proof of this fact.
\begin{theorem}[Rademacher's Theorem]\label{thm:rademacher}
Let $U$ be an open set. Then, 
\begin{enumerate}
\item\label{item:lip_iff_W_1infty} a function $\phi:U\to \C$ is locally Lipschitz if and only if $\phi \in W^{1,\infty}_{loc}\pa{U}$. 
\item\label{item:diff_ae_and_weak_derivative} Moreover, if $\phi$ is locally Lipschitz on $U$, then it is differentiable almost everywhere\footnote{When we say ``almost everywhere'' we always mean with respect to $d-$dimensional Lebesgue measure.} and its gradient equals its weak gradient almost everywhere\footnote{This equality is to be understood in the following way: each component of the pointwise gradient belongs to the equivalence class of the corresponding component of the weak gradient.}.
\end{enumerate}
\end{theorem}
The proof of this theorem can be obtained from material in Chapter 5 of \cite{Evans} (subchapter 5.8.2 b., pages 279-281). Part \eqref{item:lip_iff_W_1infty}, as pointed out in the remark on p. 280 of \cite{Evans}, can be obtained by an easy adaptation of the arguments of the proof of Theorem 4 on pp. 279--280. Part \eqref{item:diff_ae_and_weak_derivative} is the case $p=\infty$ of Theorem 5 (\cite[pp. 280--281]{Evans}) combined with Theorem 6 (\cite[p. 281]{Evans}). The alternative characterisation of $W^{1,p}_{loc}\pa{U}$ mentioned just after Definition \ref{def:LocalSobolev} is relevant here.
\begin{remark}\label{rem:UtterlyObvious}
When we use Theorem \ref{thm:rademacher}, it is sometimes helpful to also keep in mind the very obvious fact that, at every point $x\in U$ where the pointwise gradient $\nabla\phi(x)$ of a locally Lipschitz function $\phi$ exists, it satisfies $\left|\nabla\phi(x)\right|\le L$ where $L$ is the Lipschitz constant of the restriction of $\phi$ to a bounded open set containing $x$.
\end{remark}
\begin{lemma}\label{lem:range}
Let $\pa{\omega_0,\omega_1}$ form a $(\theta,p)-$admissible pair for some $\theta\in(0,1)$ and $p\in[1,\infty)$, and let $\phi\in Lip_c$. Then, for each fixed $z\in\St$, we have that $f^\phi(x,z)\in Lip_c$ and consequently
\begin{equation}\label{eq:InclusionOf_FphiAtZ}
f^\phi(\cdot,z)\in W^{1,p}_0(\omega_0)\cap W^{1,p}_0(\omega_1) \subset W^{1,p}_0(\omega_0)+W^{1,p}_0(\omega_1). 
\end{equation}
\end{lemma}
Before we begin proving this lemma, it is convenient to present a few simple facts that we will use in the proof of this lemma, as well as in subsequent proofs:
\begin{enumerate}[(a)]
\item \label{fact:LocLipschitzGoesGlobal} If the function $\psi:U\to\mathbb{C}$ is locally Lipschitz and has compact support, then it is Lipschitz on all of $U$. This rather standard result is perhaps slightly less obvious than may seem at first glance. For the sake of completeness, we include its quite straightforward and short proof in the appendix.
\item\label{item:loc_lip_bounded_on_compact} Every Lipschitz function is bounded on every compact set. Therefore every locally Lipschitz function also has this property.
\item\label{item:pw_lip_multi} The pointwise product of two bounded Lipschitz functions is itself also a bounded Lipschitz function. So, by \eqref{item:loc_lip_bounded_on_compact}, the pointwise product of two locally Lipschitz functions is also locally Lipschitz.
\item\label{item:composition} The composition of two Lipschitz functions is a Lipschitz function. Therefore the composition of two locally Lipschitz functions is a locally Lipschitz function. 
\item\label{item:bounded_derivative_is_lip} A function mapping $\mathbb{R}$ to $\mathbb{C}$ whose real and imaginary parts have bounded derivatives on every bounded interval is a locally Lipschitz function.
\end{enumerate}
\begin{proof}[Proof of Lemma \ref{lem:range}]
Let us define $g_{1}:U\to\mathbb{R}$ by $g_{1}(x)=\log\left(r_{\omega}(x)\right)$ and define $g_{2}:\mathbb{R}\to\mathbb{C}$ by $g_{2}(t)=e^{\frac{\left(z-\theta\right)}{p}t+\beta(z-\theta)^{2}}$. Then $f^{\phi}(x,z)=\phi(x)g_{2}\left(g_{1}(x)\right)$ for all $x\in U$. Since $\phi$ has compact support so does $\phi(x)g_{2}\left(g_{1}(x)\right)$.
So, in view of fact \eqref{fact:LocLipschitzGoesGlobal} it suffices to show that $\phi(x)g_{2}\left(g_{1}(x)\right)$ is locally Lipschitz, which we will now do with the help of the other above mentioned obvious facts.\\
We first use \eqref{item:bounded_derivative_is_lip} to show that $g_{2}$ is locally Lipschitz. The $(\theta,p)-$admissibility of $\left(\omega_{0},\omega_{1}\right)$ implies that $g_{1}$ is also locally Lipschitz. So \eqref{item:composition} gives us that $g_{2}\circ g_{1}$ is locally Lipschitz. Finally, since $\phi$ is obviously locally Lipschitz,
\eqref{item:pw_lip_multi} completes the proof that $\phi(x)g_{2}\left(g_{1}(x)\right)$ is locally Lipschitz and therefore also the proof of this lemma.
\end{proof}
\begin{lemma}\label{lem:analyticity_continuity_and_boundedness}
Let $\pa{\omega_0,\omega_1}$ form a $(\theta,p)-$admissible pair for some $\theta\in(0,1)$ and $p\in[1,\infty)$, and let $\phi\in Lip_c$. Then the function $f^\phi$ is an element of the space $\F\pa{W^{1,p}_0\pa{\omega_0},W^{1,p}_0\pa{\omega_1}}$.
\end{lemma}
\begin{proof} 
For reasons which the reader can quite possibly guess and which anyway will become apparent later, we begin this proof by establishing some properties of the functions $\left(\log\left(r_{\omega}(x)\right)\right)^{k}$ and $\left(\log\left(r_{\omega}(x)\right)\right)^{k}\phi(x)$ for each non-negative integer $k$. 
We first note that the $(\theta,p)-$admissibility of $\left(\omega_{0},\omega_{1}\right)$, together with facts \eqref{item:composition} and \eqref{item:bounded_derivative_is_lip} that were stated after Lemma \ref{lem:range}, ensure that, for each such $k$, the function $\left(\log\left(r_{\omega}(x)\right)\right)^{k}$ is locally Lipschitz. Consequently, using fact \eqref{item:pw_lip_multi} and then fact \ref{fact:LocLipschitzGoesGlobal}, we obtain that $\left(\log\left(r_{\omega}(\cdot)\right)\right)^{k}\phi(\cdot)$ is in $Lip_{c}$. This in turn implies, again because $\left(\omega_{0},\omega_{1}\right)$ is $(\theta,p)-$admissible, that this same function is also an element of $W_{0}^{1,p}(\omega_{0})$ and of $W_{0}^{1,p}(\omega_{1})$.\\
Our next task will be to estimate the norm of this function in each of these spaces.\\
The fact that $\phi$ has compact support ensures the existence of an open set $V_{\phi}$ whose closure is a compact subset of $U$ such that at every point $x\in U\setminus V_{\phi}$ we have that $\phi(x)=0$ and also that the pointwise gradient $\nabla\phi(x)$ exists and also equals $0$. \\
We observe that, for $j\in\left\{ 0,1\right\}$, 
$$ \left\Vert \left(\log\left(r_{\omega}(\cdot)\right)\right)^{k}\phi(\cdot)\right\Vert _{L^{p}\left(\omega_{j}\right)}^{p}=\int_{U}\left\vert \phi(x)\right\vert ^{p}\left\vert \log\left(r_{\omega}(x)\right)\right\vert ^{pk}\omega_{j}(x)dx$$
$$ \leq\left(\sup_{x\in V_{\phi}}\left\vert \log\left(r_{\omega}(x)\right)\right\vert \right)^{pk}\left\Vert \phi\right\Vert _{L^{p}\left(\omega_{j}\right)}^{p}=\Lambda^{pk}\left\Vert \phi\right\Vert _{L^{p}\left(\omega_{j}\right)}^{p},$$
where $\Lambda$ denotes the quantity
$$\Lambda:=\sup_{x\in V_{\phi}}\left\vert \log\left(r_{\omega}(x)\right)\right\vert $$
which is finite, as the supremum of a Lipschitz function on a bounded set. \\
Next we need to make some more preparations for performing a somewhat more elaborate calculation to estimate the size of $\left\Vert \nabla\left(\left(\log\left(r_{\omega}(\cdot)\right)\right)^{k}\phi(\cdot)\right)\right\Vert _{L^{p}\left(\omega_{j}\right)}^{p}$.\\
We note the following:
\begin{itemize}
\item For any $a,b\in\R$ we have that 
\begin{equation}\label{eq:known_ineq}
\left(\left\vert a\right\vert +\left\vert b\right\vert \right)^{p}\leq2^{p}\left(\left\vert a\right\vert ^{p}+\left\vert b\right\vert ^{p}\right),\quad\quad p\geq0.
\end{equation}
\item By part \eqref{item:diff_ae_and_weak_derivative} of Theorem \ref{thm:rademacher} we know that the pointwise gradients $\nabla\log\left(r_{\omega}(x)\right)$
and $\nabla\phi(x)$ of the locally Lipschitz functions $\log\left(r_{\omega}(x)\right)$ and $\phi(x)$ both exist for all $x$ in a certain subset $U_{\#}$ of $U$ which is such that $U\setminus U_{\#}$ has measure $0$. We shall let $\Xi$ denote the supremum 
\begin{equation}\label{eq:DefineXo}
\Xi:= \sup_{x\in V_{\phi}\cap U_{\#}}\abs{\nabla \log\pa{\rw(x)}},
\end{equation}
which is obviously finite (cf. Remark \ref{rem:UtterlyObvious}).
\item Furthermore it is also obvious that the pointwise gradient of $\left(\log\left(r_{\omega}(x)\right)\right)^{k}$ exists and equals $k\left(\log\left(r_{\omega}(x)\right)\right)^{k-1}\nabla\log\left(r_{\omega}(x)\right)$ for every $x\in U_{\#}$. As already observed above, the function
$\left(\log\left(r_{\omega}(x)\right)\right)^{k}$ is locally Lipschitz. So another application of Rademacher's theorem tells us that the pointwise gradient of $\left(\log\left(r_{\omega}(x)\right)\right)^{k}$ on $U_{\#}$ is also its weak gradient on $U$. Analogous statements hold for the pointwise
and weak gradients of the functions $\phi$ and $\left(\log\left(r_{\omega}(x)\right)\right)^{k}\phi(x)$.
\end{itemize}
Having made these preparations, we can now see that, for $j\in\left\{ 0,1\right\} $ and for each $k\in\mathbb{N}$,
$$ \left\Vert \nabla\left(\left(\log\left(r_{\omega}(\cdot)\right)\right)^{k}\phi(\cdot)\right)\right\Vert _{L^{p}\left(\omega_{j}\right)}^{p}\leq2^{p}\int_{V_\phi}\left\vert \nabla\phi(x)\right\vert ^{p}\left\vert \log\left(r_{\omega}(x)\right)\right\vert ^{pk}\omega_{j}(x)dx$$
$$ +2^{p}k^{p}\int_{V_{\phi}}\left\vert \phi(x)\right\vert ^{p}\left\vert \nabla\log\left(r_{\omega}(x)\right)\right\vert ^{p}\left\vert \log\left(r_{\omega}(x)\right)\right\vert ^{p(k-1)}\omega_{j}(x)dx$$
$$\leq2^{p}\Lambda^{p(k-1)}\left(\Lambda^{p}\left\Vert \nabla\phi\right\Vert _{L^{p}\left(\omega_{j}\right)}^{p}+k^{p}\Xi^{p}\left\Vert \phi\right\Vert _{L^{p}\left(\omega_{j}\right)}^{p}\right)$$
$$\leq2^{p}\Lambda^{p(k-1)}\left(\Lambda^{p}+k^{p}\Xi^{p}\right)\left\Vert \phi\right\Vert _{W^{1,p}\left(\omega_{j}\right)}^{p},$$
and for $k=0$ we obviously obtain that 
$$\left\Vert \nabla\left(\left(\log\left(r_{\omega}(\cdot)\right)\right)^{k}\phi(\cdot)\right)\right\Vert _{L^{p}\left(\omega_{j}\right)}^{p}\le\left\Vert \phi\right\Vert _{W^{1,p}\left(\omega_{j}\right)}^{p}.$$
We already noted above that $\left(\log\left(r_{\omega}(\cdot)\right)\right)^{k}\phi(\cdot)\in W_{0}^{1,p}(\omega_{j})$ for $j=0,1$. Now we can conclude furthermore, from the preceding estimates, that, for each $k\in\mathbb{N}$, 
$$\left\Vert \left(\log\left(r_{\omega}(\cdot)\right)\right)^{k}\phi(\cdot)\right\Vert _{W_{0}^{1,p}\left(\omega_{j}\right)}=\left(\left\Vert \left(\log\left(r_{\omega}(\cdot)\right)\right)^{k}\phi(\cdot)\right\Vert _{L^{p}\left(\omega_{j}\right)}^{p}+\left\Vert \nabla\left(\left(\log\left(r_{\omega}(\cdot)\right)\right)^{k}\phi(\cdot)\right)\right\Vert _{L^{p}\left(\omega_{j}\right)}^{p}\right)^{\frac{1}{p}}$$
$$\leq\left(\Lambda^{pk}\left\Vert \phi\right\Vert _{L^{p}\left(\omega_{j}\right)}^{p}+2^{p}\Lambda^{p(k-1)}\left(\Lambda^{p}+k^{p}\Xi^{p}\right)\left\Vert \phi\right\Vert _{W^{1,p}\left(\omega_{j}\right)}^{p}\right)^{\frac{1}{p}}$$
$$\leq \left(2^{p}\Lambda^{p(k-1)}\left(2^{-p}\Lambda^{p}+\Lambda^{p}+k^{p}\Xi^{p}\right)\left\Vert \phi\right\Vert _{W^{1,p}\left(\omega_{j}\right)}^{p}\right)^{\frac{1}{p}}$$
$$\leq 2\Lambda^{k-1}\left(2\Lambda^{p}+k^{p}\Xi^{p}\right)^{1/p}\left\Vert \phi\right\Vert _{W^{1,p}(\omega_{j})}. $$
We estimate this last expression, using \eqref{eq:known_ineq} again, but this time with $\frac{1}{p}$ in place of $p$, and thus
obtain that 
\begin{equation}
\left\Vert \left(\log\left(r_{\omega}(\cdot)\right)\right)^{k}\phi(\cdot)\right\Vert _{W_{0}^{1,p}\left(\omega_{j}\right)}\le2^{1+\frac{1}{p}}\Lambda^{k-1}\left(2^{\frac{1}{p}}\Lambda+k\Xi\right)\left\Vert \phi\right\Vert _{W^{1,p}(\omega_{j})}.\label{eq:NiceIneq}
\end{equation}
The analogue of this inequality for $k=0$ is just
the trivial fact that 
$$\left\Vert \phi\right\Vert _{W_{0}^{1,p}(\omega_{j})}=\left\Vert \phi\right\Vert _{W^{1,p}(\omega_{j})}.$$

For each $N\in\mathbb{N}$ we define the function $f_{N}^{\phi}:\mathbb{S}\to Lip_{c}$
by setting 
\begin{equation}
f_{N}^{\phi}(x,z)=e^{\beta\left(z-\theta\right)^{2}}\sum_{k=0}^{N}\frac{\left(\log\left(r_{\omega}(x)\right)\right)^{k}\left(z-\theta\right)^{k}}{p^{k}k!}\phi(x).\label{eq:definefp}
\end{equation}
Here, analogously to the convention adopted in Definition \ref{def:f_phiZZ}, it must be understood that $f_{N}^{\phi}$ denotes the map from $\mathbb{S}$ into the space of measurable functions on $U$, which is such that the image $f_{N}^{\phi}(z)$ of each $z\in\mathbb{S}$ is the function of $x$ defined by \eqref{eq:definefp}. 
For each $k\in\mathbb{N}\cup\left\{ 0\right\} $, the function $\left(\log\left(r_{\omega}(\cdot)\right)\right)^{k}\phi(\cdot)$ is known to be an element of $Lip_{c}$ and therefore also of $W_{0}^{1,p}\left(\omega_{0}\right)\cap W_{0}^{1,p}\left(\omega_{0}\right)$.
Also, a simple calculation shows that the entire function 
$$z\mapsto e^{\beta(z-\theta)^{2}}(z-\theta)^{k}$$
is bounded on $\mathbb{S}$. (Later we will do this calculation in more detail, to find an explicit bound for the modulus of this function on $\St$.) From these two facts it is obvious that $f_{N}^{\phi}$ indeed maps $\mathbb{S}$ into $Lip_{c}$ and, furthermore,
that it is an element of $\mathcal{F}\left(A_{0},A_{1}\right)$, where, throughout this proof, we will take $\left(A_{0},A_{1}\right)$ to be the Banach couple obtained by setting $A_{0}=A_{1}=W_{0}^{1,p}\left(\omega_{0}\right)\cap W_{0}^{1,p}\left(\omega_{0}\right)$ and letting $A_{0}$ and $A_{1}$ be normed by $\left\Vert \psi\right\Vert _{A_{0}}=\left\Vert \psi\right\Vert _{A_{1}}=\max\left\{ \left\Vert \psi\right\Vert _{W_{0}^{1,p}(\omega_{0})},\left\Vert \psi\right\Vert _{W_{0}^{1,p}(\omega_{1})}\right\}$.
Our next step will be to use the inequalities \eqref{eq:NiceIneq} to show that $\left\{ f_{N}^{\phi}\right\} _{N\in\mathbb{N}}$ is a Cauchy sequence in $\mathcal{F}\left(A_{0},A_{1}\right)$. For that it will clearly suffice to consider the quantity 
$$E(M,N):=\left\Vert e^{\beta\left(z-\theta\right)^{2}}\sum_{k=M}^{N}\frac{\left(\log\left(r_{\omega}(x)\right)\right)^{k}\left(z-\theta\right)^{k}}{p^{k}k!}\phi(x)\right\Vert _{\mathcal{F}\left(A_{0},A_{1}\right)}$$
for all integers $M$ and $N$ which satisfy $0\le M\le N$ and to show that $E(M,N)$ tends to $0$ when $M$ and $N$ tend to $\infty$. \\
We see that 
\begin{equation}\nonumber
\begin{gathered}
E(M,N)  \le\sum_{k=M}^{N}\left\Vert e^{\beta\left(z-\theta\right)^{2}}\frac{\left(\log\left(r_{\omega}(x)\right)\right)^{k}\left(z-\theta\right)^{k}}{p^{k}k!}\phi(x)\right\Vert _{\mathcal{F}\left(A_{0},A_{1}\right)}\\
=\sum_{k=M}^{N}\max_{j=0,1}\sup_{t\in\mathbb{R}}\frac{\left|e^{\beta\left(j+it-\theta\right)^{2}}\left(j+it-\theta\right)^{k}\right|}{p^{k}k!}\left\Vert \left(\log\left(r_{\omega}(x)\right)\right)^{k}\phi(x)\right\Vert _{A_{j}}.
\end{gathered}
\end{equation}
In view of \eqref{eq:NiceIneq} we have that 
\begin{equation}\nonumber
\begin{gathered}
\left\Vert \left(\log\left(r_{\omega}(x)\right)\right)^{k}\phi(x)\right\Vert _{A_{0}} =\left\Vert \left(\log\left(r_{\omega}(x)\right)\right)^{k}\phi(x)\right\Vert _{A_{1}}\\
\le2^{1+\frac{1}{p}}\Lambda^{k-1}\left(2^{\frac{1}{p}}\Lambda+k\Xi\right)\max_{j\in\{0,1\}}\left\Vert \phi\right\Vert _{W^{1,p}\left(\omega_{j}\right)}
\end{gathered}
\end{equation}
for all $k\in\mathbb{N}$. The analogue of this for $k=0$ is simply and obviously 
$$\left\Vert \phi\right\Vert _{A_{0}}=\left\Vert \phi\right\Vert _{A_{i}}=\max_{j\in\{0,1\}}\left\Vert \phi\right\Vert _{W^{1,p}\left(\omega_{j}\right)}$$
We also have, for each $k\in\mathbb{N}$, that
\begin{equation}\label{eq:PleaseGatherThis}
\begin{gathered}
\max_{j=0,1}\sup_{t\in\mathbb{R}}\left|e^{\beta\left(j+it-\theta\right)^{2}}\left(j+it-\theta\right)^{k}\right|\\
\le\sup_{t\ge0}e^{\beta(1-t^{2})}\pa{1+t^{2}}^{\frac{k}{2}}=e^{2\beta}\sup_{t\ge0}e^{-\beta(1+t^{2})}\pa{1+t^{2}}^{\frac{k}{2}}\\
=e^{2\beta}\sup_{y\ge1}e^{-\beta y^{2}}y^{k},
\end{gathered}
\end{equation}
and a simple computation shows that 
\begin{equation}\label{eq:known_inequality_II-1}
\max_{y\geq0}y^{k}e^{-\beta y^{2}}=\left(\frac{k}{2\beta e}\right)^{\frac{k}{2}}\quad \text{for each }\; k>0.
\end{equation}
These preceding inequalities show
that 
$$E(M,N)\le\max_{j\in\{0,1\}}\left\Vert \phi\right\Vert _{W^{1,p}\left(\omega_{j}\right)}\sum_{k=M}^{N}\gamma_{k}$$
where, for each $k\in\mathbb{N}$, 
$$\gamma_{k}=\frac{2^{1+\frac{1}{p}}\Lambda^{k-1}\left(2^{\frac{1}{p}}\Lambda+k\Xi\right)}{p^{k}k!}\left(\frac{k}{2\beta e}\right)^{\frac{k}{2}}e^{2\beta},$$
and for $k=0$, since $\max_{j=0,1}\sup_{t\in\mathbb{R}}\left|e^{\beta\left(j+it-\theta\right)^{2}}\right|\le e^{\beta}$,
$$\gamma_{0}=e^{\beta}.$$
For each $k\in\mathbb{N}$ we have that 
\begin{equation}\nonumber
\begin{gathered}
\frac{\gamma_{k+1}}{\gamma_{k}} =\frac{\Lambda}{p(k+1)}\cdot\frac{\left(2^{\frac{1}{p}}\Lambda+(k+1)\Xi\right)}{\left(2^{\frac{1}{p}}\Lambda+k\Xi\right)}\cdot\frac{1}{\sqrt{2\beta e}}\cdot\left(\frac{k+1}{k}\right)^{\frac{k}{2}}\sqrt{k+1}\\
=\frac{\Lambda}{p\sqrt{2\beta e}}\cdot\frac{\left(2^{\frac{1}{p}}\Lambda+(k+1)\Xi\right)}{\left(2^{\frac{1}{p}}\Lambda+k\Xi\right)}\cdot\left(1+\frac{1}{k}\right)^{\frac{k}{2}}\frac{1}{\sqrt{k+1}},
\end{gathered}
\end{equation}
from which it is clear that $\lim_{k\to\infty}\frac{\gamma_{k+1}}{\gamma_{k}}=0$. This shows that the positive term series $\sum_{k=0}^{\infty}\gamma_{k}$ is convergent and this implies that $E(M,N)$ indeed has the behaviour required to show that $\left\{ f_{N}^{\phi}\right\} _{N\in\mathbb{N}}$ is a Cauchy sequence in $\mathcal{F}\left(A_{0},A_{1}\right)$. 
Therefore $\left\{ f_{N}^{\phi}\right\} _{N\in\mathbb{N}}$ converges in $\mathcal{F}\left(A_{0},A_{1}\right)$ norm to an element $g\in\mathcal{F}\left(A_{0},A_{1}\right)$. In our setting, where $A_{0}$ and $A_{1}$ are the same space with the same norm, each element $h\in\mathcal{F}\left(A_{0},A_{1}\right)$ is a continuous bounded $A_{0}$ valued function on $\mathbb{S}$ which is analytic in the interior of $\mathbb{S}$. Therefore an obvious extension of the Phragmen-Lindel\"of theorem to Banach space valued
analytic functions shows that 
$$\left\Vert h(z)\right\Vert _{A_{0}}\le\sup_{\zeta\in\partial\mathbb{S}}\left\Vert h(\zeta)\right\Vert _{A_{0}}=\left\Vert h\right\Vert _{\mathcal{F}\left(A_{0},A_{1}\right)}$$
for each $z\in\mathbb{S}$. In particular, this means that, for each fixed $z\in\mathbb{S}$, $g(z)\in A_{0}$ and 
\begin{equation}\label{eq:fNtoG}
\lim_{N\to\infty}\left\Vert f_{N}^{\phi}(z)-g(z)\right\Vert _{A_{0}}=0.
\end{equation}
Since $g(z)$ is an element of $A_{0}$ and therefore a measurable function, or rather an equivalence class of measurable functions on $U$, it will be convenient to let $g(x,z)$ denote the value of $g(z)$ at each (or almost every) point $x$ of $U$. It follows from (\ref{eq:fNtoG}) that, for each fixed $z\in\mathbb{S}$ the sequence $\left\{ f_{N}^{\phi}(z)\right\} _{N\in\mathbb{N}}$ converges to $g(z)$ also in $L^{p}(\omega)$ when $\omega$ is either one of the weight functions $\omega_{0}$ and $\omega_{1}$. By standard results, a norm convergent sequence in any (weighted or unweighted) $L^{p}$ space must have a subsequence which converges pointwise at almost every point of the underlying measure space. So there exists a subset $U_{z}$ of $U$ possibly depending on $z$, and an unbounded subsequence $\left\{ N_{k}\right\} _{k\in\mathbb{N}}$ of positive integers, also possibly depending on $z$, such that $U\setminus U_{z}$ has measure zero and $g(x,z)=\lim_{k\to\infty}f_{N_{k}}^{\phi}(x,z)$
for all $x\in U_{z}$. \\
But, from the formula \eqref{eq:definefp} we see that the ``whole'' sequence $\left\{ f_{N}^{\phi}(x,z)\right\} _{N\in\mathbb{N}}$ converges pointwise for \textit{every} $x\in U$ and that its pointwise limit is 
$$e^{\frac{\left(z-\theta\right)}{p}\log\left(r_{\omega}(x)\right)+\beta(z-\theta)^{2}}\phi(x)=f_{\beta,\theta}^{\phi}(x,z).$$
This shows that $f^{\phi}(x,z)=g(x,z)$ for every $x\in U_{z}$ which means that $f^{\phi}(z)$ and $g(z)$ are the same element of $W_{0}^{1,p}(\omega_{0})\cap W_{0}^{1,p}(\omega_{1})$ for each $z\in\mathbb{S}$. Consequently $f^{\phi}$ coincides with $g$ and is therefore an element of $\mathcal{F}\left(A_{0},A_{1}\right)$.
Since $A_{0}$ is continuously embedded in $W_{0}^{1,p}(\omega_{0})$ and $A_{1}$ is continuously embedded in $W_{0}^{1,p}(\omega_{1})$ it follows that $\mathcal{F}\left(A_{0},A_{1}\right)$ is continuously embedded in $\mathcal{F}\left(W_{0}^{1,p}\left(\omega_{0}\right),W_{0}^{1,p}\left(\omega_{1}\right)\right)$
and this completes our proof of Lemma \ref{lem:analyticity_continuity_and_boundedness}. 
\end{proof} 
\begin{remark}\label{rem:AlternativProofOfInclusionZZ}
Note that the fact $f^{\phi}(z)$ and $g(z)$ are the same element of $W_{0}^{1,p}(\omega_{0})\cap W_{0}^{1,p}(\omega_{1})$ for each $z\in\mathbb{S}$ provides an alternative proof of the inclusion \eqref{eq:InclusionOf_FphiAtZ}.
\end{remark}
As well as showing that $f^{\phi}$ is an element of $\mathcal{F}\left(W_{0}^{1,p}\left(\omega_{0}\right),W_{0}^{1,p}\left(\omega_{1}\right)\right)$,
Lemma \ref{lem:analyticity_continuity_and_boundedness}, via its proof, also provides the means to obtain a expression which is an upper bound for $\left\Vert f^{\phi}\right\Vert _{\mathcal{F}\left(W_{0}^{1,p}\left(\omega_{0}\right),W_{0}^{1,p}\left(\omega_{1}\right)\right)}$.
But that very complicated expression\footnote{It depends in quite complicated ways on $\phi$, $\beta$, $p$, $\omega_{0}$ and $\omega_{1}$ including dependence on expressions involving the set $V_{\phi}$ and the numbers $\Lambda$ and $\Xi$ which themselves have complicated dependence on $\phi$, $\omega_{0}$ and $\omega_{1}$.}
is not useful for our needs here. Instead we need the upper bound which will be provided by the next result:
\begin{prop}\label{prop:f_phi_in_interpolation}
Let $\pa{\omega_0,\omega_1}$ form a $(\theta,p)-$admissible pair for some $\theta\in(0,1)$ and $p\in[1,\infty)$, and let $\phi$ be an element of $Lip_c$.\\ 
Then the element $f^\phi_{\beta,\theta}$ of $\mathcal{F}\left(W_{0}^{1,p}\left(\omega_{0}\right),W_{0}^{1,p}\left(\omega_{1}\right)\right)$ defined by \eqref{eq:def_of_f_phi} satisfies
\begin{equation}\label{eq:norm_of_f_phi_in_interpolation}
\norm{f^\phi_{\beta,\theta}}_{\F\pa{W^{1,p}_0\pa{\omega_0},W^{1,p}_0\pa{\omega_1}}} \leq 2e^\beta \max\br{1,\frac{e^\beta}{p\sqrt{2\beta e}}}\norm{\phi}_{\W^p\pa{\theta,\rw}}.
\end{equation}
Consequently, we also obtain that 
\begin{equation}\label{eq:MoreProp}
\left\Vert \phi\right\Vert _{\left[W_{0}^{1,p}\left(\omega_{0}\right),W_{0}^{1,p}\left(\omega_{1}\right)\right]_{\theta}}\leq C_{p}\left\Vert \phi\right\Vert _{\mathcal{W}^{p}(\theta,r_{\omega})}\mbox{ for every \,}\phi\in Lip_{c},
\end{equation}
for a positive constant $C_{p}$ depending only on $p$, which can be chosen to be defined as in \eqref{eq:ValueOfCp}.
\end{prop}
\begin{remark}\label{rem:optimisationZZ}
In the statement of Proposition \ref{prop:f_phi_in_interpolation} we have used the notation $f_{\beta,\theta}^{\phi}$ instead of its abbreviated variant $f^{\phi}$, and we will sometimes also do this in the following proof of this proposition, in some places where it is necessary to emphasise the dependence of this element on the parameter $\beta$. 
\end{remark} 
\begin{proof}[Proof of Proposition \ref{prop:f_phi_in_interpolation}]
We first notice that
$$\abs{f^\phi(x,z)}=\rw(x)^{\frac{\text{Re}z-\theta}{p}}e^{\beta\pa{\pa{\text{Re}z-\theta}^2-\pa{\text{Im}z}^2}}\abs{\phi(x)}.$$
Then we recall that, as shown just before \eqref{eq:DefineXo} in the proof of Lemma \ref{lem:analyticity_continuity_and_boundedness}, there exists a set $U_{\#}$ (which of course depends on our given function $\phi\in Lip_{c}$) which is almost all of $U$ and consists of all points $x$ for which the pointwise gradients of $\log\left(r_{\omega}(x)\right)$ and of $\phi(x)$ both exist. 
It obviously follows that, at every point $x$ in this same set $U_{\#}$ and for each fixed $z\in\mathbb{S}$, the pointwise gradient with respect to $x$ of $f^{\phi}(x,z)$ exists
and equals 
$$\nabla_{x}f^{\phi}(x,z)=e^{\frac{\left(z-\theta\right)}{p}\log\left(r_{\omega}(x)\right)+\beta(z-\theta)^{2}}\left(\nabla_{x}\phi(x)+\frac{z-\theta} {p}\phi(x)\nabla_{x}\log\left(r_{\omega}(x)\right)\right),$$
and therefore also
$$\left\vert \nabla_{x}f^{\phi}(x,z)\right\vert =r_{\omega}(x)^{\frac{\text{Re}z-\theta}{p}}e^{\beta\left(\left(\text{Re}z-\theta\right)^{2}-\left(\text{Im}z\right)^{2}\right)}\left\vert \nabla_{x}\phi(x)+\frac{z-\theta}{p}\phi(x)\nabla_{x}\log\left(r_{\omega}(x)\right)\right\vert. $$
Since $\phi\in Lip_{c}$, Lemma \ref{lem:range} ensures that we also have $f^{\phi}\left(\cdot,z\right)\in Lip_{c}$. Therefore, by \eqref{item:diff_ae_and_weak_derivative} of Theorem \ref{thm:rademacher}, the weak gradients of each of the functions $\phi$ and $f^{\phi}(\cdot,z)$ exist and coincide almost everywhere with their respective pointwise gradients. We will need to use this fact because the norms appearing in \eqref{eq:norm_of_f_phi_in_interpolation} are defined via weak gradients, rather than pointwise gradients.\\
We will also need to use the fact that, for all $x\in U$ and $j\in\left\{ 0,1\right\}$,
$$\rw(x)^{j-\theta}\omega_j(x) = \pa{\frac{\omega_0(x)}{\omega_1(x)}}^{j-\theta}\omega_j(x)=\omega_\theta(x).$$
We compute
$$\norm{f^\phi(\cdot, j+it)}_{L^p\pa{\omega_j}}^p=\int_{U}\rw(x)^{j-\theta}e^{p\beta\pa{\pa{j-\theta}^2-t^2}}\abs{\phi(x)}^p\omega_j(x)dx$$
$$=\int_{U}e^{p\beta\pa{\pa{j-\theta}^2-t^2}}\abs{\phi(x)}^p\omega_\theta(x)dx\leq e^{p\beta}\norm{\phi}^p_{L^p\pa{\omega_\theta}}.$$
Then we also see that the weak gradients of $f^\phi(\cdot,z)$ and of $\phi$ satisfy
$$\norm{\nabla_x f^\phi(\cdot,j+it)}^p_{L^p\pa{\omega_j}}=\int_{U}\rw(x)^{j-\theta}e^{p\beta\pa{\pa{j-\theta}^2-t^2}}\abs{\nabla_x \phi(x) + \frac{z-\theta}{p}\phi(x)\nabla_x \log\pa{\rw(x)}}^p\omega_j(x)dx$$
$$=\int_{U}e^{p\beta\pa{\pa{j-\theta}^2-t^2}}\abs{\nabla_x \phi(x) + \frac{z-\theta}{p}\phi(x)\nabla_x \log\pa{\rw(x)}}^p\omega_\theta(x)dx$$
which, by \eqref{eq:known_ineq}, does not exceed
\begin{equation}\label{eq:Stum365}
2^{p}\int_{U}e^{p\beta\left(\left(j-\theta\right)^{2}-t^{2}\right)}\left(\left\vert \nabla_{x}\phi(x)\right\vert ^{p}+\frac{\left(\left(j-\theta\right)^{2}+t^{2}\right)^{\frac{p}{2}}}{p^{p}}\left\vert \phi(x)\right\vert ^{p}\left\vert \nabla_{x}\log\left(r_{\omega}(x)\right)\right\vert ^{p}\right)\omega_{\theta}(x)dx.
\end{equation}
In turn, in order to estimate this expression we use the fact that $\left\vert j-\theta\right\vert \leq1$, for one of its terms. For the other term we recall and use the calculations of
\eqref{eq:PleaseGatherThis} and \eqref{eq:known_inequality_II-1} which together, when $k=1$, give us that 
$$e^{\beta\left(\left(j-\theta\right)^{2}-t^{2}\right)}\left(\left(j-\theta\right)^{2}+t^{2}\right)^{\frac{1}{2}}\le\left(\frac{1}{2\beta e}\right)^{\frac{1}{2}}e^{2\beta}.$$
These considerations enable us to deduce that the expression \eqref{eq:Stum365} is dominated by
$$
2^{p}e^{p\beta}\left\Vert \nabla_{x}\phi\right\Vert _{L^{p}\left(\omega_{\theta}\right)}^{p}+2^{p}\frac{e^{2p\beta}}{\left(2\beta e\right)^{\frac{p}{2}}p^{p}}\left\Vert \phi\right\Vert _{L^{p}\left(\omega_{\theta}\left\vert \nabla\log\left(r_{\omega}\right)\right\vert ^{p}\right)}^{p}.
$$
We conclude that
\begin{equation}\nonumber
\begin{gathered}
\left\Vert f^{\phi}\right\Vert _{\mathcal{F}\left(W_{0}^{1,p}\left(\omega_{0}\right),W_{0}^{1,p}\left(\omega_{1}\right)\right)}=\sup_{t\in\mathbb{R}}\left(\left\Vert f^{\phi}(\cdot,it)\right\Vert _{W_{0}^{1,p}\left(\omega_{0}\right)},\left\Vert f^{\phi}(\cdot,1+it)\right\Vert _{W_{0}^{1,p}\left(\omega_{1}\right)}\right)\\
\le\left(e^{p\beta}\left\Vert \phi\right\Vert _{L^{p}\left(\omega_{\theta}\right)}^{p}+2^{p}e^{p\beta}\left\Vert \nabla_{x}\phi\right\Vert _{L^{p}\left(\omega_{\theta}\right)}^{p}+2^{p}\frac{e^{2p\beta}}{\left(2\beta e\right)^{\frac{p}{2}}p^{p}}\left\Vert \phi\right\Vert _{L^{p}\left(\omega_{\theta}\left\vert \nabla\log\left(r_{\omega}\right)\right\vert ^{p}\right)}^{p}\right)^{\frac{1}{p}}\\
\le\left(2^pe^{p\beta}\left\Vert \phi\right\Vert _{W^{1,p}\left(\omega_{\theta}\right)}^{p}+2^{p}\frac{e^{2p\beta}}{\left(2\beta e\right)^{\frac{p}{2}}p^{p}}\left\Vert \phi\right\Vert _{L^{p}\left(\omega_{\theta}\left\vert \nabla\log\left(r_{\omega}\right)\right\vert ^{p}\right)}^{p}\right)^{\frac{1}{p}}.
\end{gathered}
\end{equation}

Recalling the notation introduced in \eqref{eq:norm_on_W}, we see that the preceeding expression is dominated by 
$$
\left(\max\left\{ 2^{p}e^{p\beta},2^{p}\frac{e^{2p\beta}}{\left(2\beta e\right)^{\frac{p}{2}}p^{p}}\right\} \right)^{\frac{1}{p}}\left\Vert \phi\right\Vert _{\mathcal{W}^{p}(\theta,r_{\omega})}=2e^{\beta}\max\left\{ 1,\frac{e^{\beta}}{p\sqrt{2\beta e}}\right\} \left\Vert \phi\right\Vert _{\mathcal{W}^{p}(\theta,r_{\omega})}.$$
This establishes \eqref{eq:norm_of_f_phi_in_interpolation}. Since $f_{\beta,\theta}^{\phi}(x,\theta)=\phi(x)$, we conclude that $\phi\in\left[W_{0}^{1,p}\left(\omega_{0}\right),W_{0}^{1,p}\left(\omega_{1}\right)\right]_{\theta}$ with $\left\Vert \phi\right\Vert _{\left[W_{0}^{1,p}\left(\omega_{0}\right),W_{0}^{1,p}\left(\omega_{1}\right)\right]_{\theta}}\le\left\Vert f_{\beta,\theta}^{\phi}\right\Vert _{\mathcal{F}\left(W_{0}^{1,p}\left(\omega_{0}\right),W_{0}^{1,p}\left(\omega_{1}\right)\right)}$.
Combining this inequality with \eqref{eq:norm_of_f_phi_in_interpolation} and taking the infimum over all values of $\beta$ in the interval $(0,\infty)$, (an infimum which is obviously attained) establishes \eqref{eq:MoreProp} when $C_{p}$ is the constant defined in \eqref{eq:ValueOfCp}. This completes the proof of the proposition.
\end{proof}
We are finally ready to present the proof of Theorem \ref{thm:difficult_inclusion_on_W_0}. In fact nearly all of the tools needed for the proof have already been developed by the preceding results of this section. It remains only to use an approximation argument to extend the result which we already have obtained for functions in $Lip_{c}$ to all functions in $\mathcal{W}_{0}^{p}\left(\theta,r_{\omega}\right)$.
\begin{proof}[Proof of Theorem \ref{thm:difficult_inclusion_on_W_0}]
Let $\phi$ be an arbitrary function in the space $\mathcal{W}_{0}^{p}(\theta,r_{\omega})$. In view of the definition of this space, there exists a sequence of functions $\left\{ \phi_{n}\right\} _{n\in\mathbb{N}}\in Lip_{c}$
such that 
$$\left\Vert \phi-\phi_{n}\right\Vert _{\mathcal{W}^{p}(\theta,r_{\omega})}\underset{n\rightarrow\infty}{\longrightarrow}0.$$
Since $Lip_{c}\subset\left[W_{0}^{1,p}\left(\omega_{0}\right),W_{0}^{1,p}\left(\omega_{1}\right)\right]_{\theta}$ and since, in view of \eqref{eq:MoreProp},
$$\left\Vert \phi_{n}-\phi_{m}\right\Vert _{\left[W_{0}^{1,p}\left(\omega_{0}\right),W_{0}^{1,p}\left(\omega_{1}\right)\right]_{\theta}}\leq C_p\left\Vert \phi_{n}-\phi_{m}\right\Vert _{\mathcal{W}^{p}(\theta,r_{\omega})}\underset{n,m\rightarrow\infty}{\longrightarrow}0,$$
we find that $\left\{ \phi_{n}\right\} _{n\in\mathbb{N}}$ is a Cauchy sequence in $\left[W_{0}^{1,p}\left(\omega_{0}\right),W_{0}^{1,p}\left(\omega_{1}\right)\right]_{\theta}$. 
As $\left[W_{0}^{1,p}\left(\omega_{0}\right),W_{0}^{1,p}\left(\omega_{1}\right)\right]_{\theta}$ is a Banach space we conclude that there exists $g\in\left[W_{0}^{1,p}\left(\omega_{0}\right),W_{0}^{1,p}\left(\omega_{1}\right)\right]_{\theta}$ such that 
$$\left\Vert \phi_{n}-g\right\Vert _{\left[W_{0}^{1,p}\left(\omega_{0}\right),W_{0}^{1,p}\left(\omega_{1}\right)\right]_{\theta}}\underset{n\rightarrow\infty}{\longrightarrow}0.$$
The relevant definitions immediately imply\footnote{In fact we already took note of and used this simple
fact while proving Theorem \ref{thm:SimultaneousApprox}.}
(cf. part \eqref{item:inter_properties_inclusion_of_norms_and_subspaces} of Theorem \ref{thm:OtherProperties}) that 
$$
\left[W_{0}^{1,p}\left(\omega_{0}\right),W_{0}^{1,p}\left(\omega_{1}\right)\right]_{\theta}\subset\left[W^{1,p}\left(\omega_{0}\right),W^{1,p}\left(\omega_{1}\right)\right]_{\theta},
$$
and
$$\left\Vert \cdot\right\Vert _{\left[W^{1,p}\left(\omega_{0}\right),W^{1,p}\left(\omega_{1}\right)\right]_{\theta}}\leq\left\Vert \cdot\right\Vert _{\left[W_{0}^{1,p}\left(\omega_{0}\right),W_{0}^{1,p}\left(\omega_{1}\right)\right]_{\theta}}.$$

According to Theorem \ref{thm:easy_inclusion}, $g\in W^{1,p}\left(\omega_{\theta}\right)$, and \
\begin{align*}
\left\Vert \phi_{n}-g\right\Vert _{W^{1,p}\left(\omega_{\theta}\right)} & \leq\left\Vert \phi_{n}-g\right\Vert _{\left[W^{1,p}\left(\omega_{0}\right),W^{1,p}\left(\omega_{1}\right)\right]_{\theta}}\\
 & \le\left\Vert \phi_{n}-g\right\Vert _{\left[W_{0}^{1,p}\left(\omega_{0}\right),W_{0}^{1,p}\left(\omega_{1}\right)\right]_{\theta}}\underset{n\rightarrow\infty}{\longrightarrow}0.
\end{align*}
Therefore $\left\{ \phi_{n}\right\} _{n\in\mathbb{N}}$ converges to $g$ in $W^{1,p}\left(\omega_{\theta}\right)$. However, since $\left\{ \phi_{n}\right\} _{n\in\mathbb{N}}$ converges to $\phi$ in $\mathcal{W}_{0}^{p}\left(\theta,r_{\omega}\right)$ it must also converge to $\phi$ in $W^{1,p}\left(\omega_{\theta}\right)$. As $W^{1,p}\left(\omega_{\theta}\right)$ is a Banach space due to the $(\theta,p)-$admissibility of the pair $\left(\omega_{0},\omega_{1}\right)$, we conclude that $g=\phi$. 
Thus $\phi$ is an element of $\left[W_{0}^{1,p}\left(\omega_{0}\right),W_{0}^{1,p}\left(\omega_{1}\right)\right]_{\theta}$ and $\left\{ \phi_{n}\right\} _{n\in\mathbb{N}}$ converges to it in the norm of this space.\\
To complete this proof it remains only to observe that inequality \eqref{eq:difficult_inclusion_norm} follows immediately from the facts that 
$$ \lim_{n\rightarrow\infty}\left\Vert \phi_{n}\right\Vert _{\left[W_{0}^{1,p}\left(\omega_{0}\right),W_{0}^{1,p}\left(\omega_{1}\right)\right]_{\theta}}=\left\Vert \phi\right\Vert _{\left[W_{0}^{1,p}\left(\omega_{0}\right),W_{0}^{1,p}\left(\omega_{1}\right)\right]_{\theta}}$$
and
$$\lim_{n\rightarrow\infty}\left\Vert \phi_{n}\right\Vert _{\mathcal{W}^{p}(\theta,r_{\omega})}=\left\Vert \phi\right\Vert _{\mathcal{W}^{p}(\theta,r_{\omega})},$$
together with \eqref{eq:MoreProp}.
\end{proof}
\section{Approximation theorems}\label{sec:approximation}
Theorems  \ref{thm:easy_inclusion} and \ref{thm:difficult_inclusion_on_W_0} are the principal ingredients that we require for proving our main theorem, Theorem \ref{thm:main},
as well as for most of the proof of Theorem \ref{thm:main_for_U_including_R_d}.
As further preparation for these tasks, we require some approximation results that will allow us to show that the conditions imposed in Theorems \ref{thm:main} and  \ref{thm:main_for_U_including_R_d}  in fact also suffice to ensure that $Lip_c\pa{\R^d}$ and even also $C_c^\infty\pa{\R^d}$ are dense in $W^{1,p}\pa{\R^d,\omega}$ and also in $\W^p\pa{\R^d,\theta,\rw}$.\\
These approximation results are the focus of this section, with the following as our main goal:
\begin{theorem}\label{thm:W_p_is_W_p_0}
Let $\omega_0$ and $\omega_1$ be weight functions that satisfy the compact boundedness condition on $\R^d$. Assume in addition that the function $\rw(x)$ defined in \eqref{eq:rw} is locally Lipschitz on $\R^d$. Then, for each $p\in[1,\infty)$,
\begin{equation}\label{eq:W_p_is_Etc-1}
W^{1,p}\left(\mathbb{R}^{d},\omega_{j}\right)=W_{0}^{1,p}\left(\mathbb{R}^{d},\omega_{j}\right)\text{ for }\;j\in\left\{ 0,1\right\} 
\end{equation}
and
\begin{equation}\label{eq:W_p_is_Etc-2}
\mathcal{W}^{p}\left(\mathbb{R}^{d},\theta,r_{\omega}\right)=\mathcal{W}_{0}^{p}\left(\mathbb{R}^{d},\theta,r_{\omega}\right)\text{ for each }\;\theta\in(0,1).
\end{equation}
\end{theorem}
The proof of Theorem \ref{thm:W_p_is_W_p_0} relies on several approximation results, which we will establish via a series of lemmas. We start by showing that the space $W^{1,p}\pa{U,\omega}$ has two important properties whenever $\omega$ satisfies the compact boundedness condition.
\begin{lemma}\label{lem:banach_space_and_loc}
Let $U$ be an open subset of $\mathbb{R}^{d}$, let $p\in[1.\infty)$ and let $\omega$ be a weight function which satisfies the compact boundedness condition
on $U$. Then
\begin{enumerate}[(i)]
\item\label{item:W_1p_is_banach} $W^{1,p}\left(U,\omega\right)$ is a Banach space,\\
and
\item\label{item:W_1p_in_loc} $W^{1,p}\left(U,\omega\right)$ is contained in $W_{loc}^{1,p}\left(U\right)$.
\end{enumerate}
In fact these two properties also hold if our hypothesis on the weight function $\omega$ is replaced by the weaker requirement that
\begin{equation}\label{eq:HalfCBC}
\inf_{x\in K}\omega(x)>0\,\text{ for every compact subset }K\text{ of }\;U.
\end{equation}
Furthermore, in order to obtain that \eqref{item:W_1p_is_banach} holds, it suffices to assume that $\omega$ only satisfies an even weaker condition, namely that 
\begin{equation}\label{eq:KO-Property}
\frac{1}{\omega^{1/p}}\chi_{K}\in L^{p'}(U)\ \text{ for every compact subset }\;K\;\text{ of }\;U.
\end{equation}
where, as usual, $1/p+1/p'=1$. 
\end{lemma}
\begin{remark}\label{rem:When_p_is_infiniteZZ}
For $p=\infty$ the space $W^{1,\infty}(U,\omega)$ has both of the properties \eqref{item:W_1p_is_banach} and \eqref{item:W_1p_in_loc} in the statement of Lemma \ref{lem:banach_space_and_loc} for \textit{all} choices of the weight function $\omega$. This is simply because, as already pointed out in Remark \ref{rem:Unweighted}, $W^{1,\infty}(U,\omega)$ coincides isometrically with the unweighted space $W^{1,\infty}(U)$, for which \eqref{item:W_1p_in_loc} is a triviality and \eqref{item:W_1p_is_banach} is a standard result. (Cf. e.g., Theorem 2 of \cite[p. 249]{Evans}.)
\end{remark}
\begin{remark}\label{rem:CompareWithKOZZ}  
In  fact, one can also deduce that $W^{1,p}\left(\omega\right)$ is a Banach space for $p\in(1,\infty)$ from Theorem 1.11 of \cite{KO}.
For $p>1$ the condition (\ref{eq:KO-Property}) is just a reformulation of the condition $\omega\in B_{p}(\Omega)$ introduced in Definition 1.4 on page 538 of \cite{KO} with $\Omega$ in the role of our set $U$. So part \eqref{item:W_1p_is_banach} of this lemma and our proof of it are very closely related to the statement and proof of Theorem 1.11 of \cite[pp. 540--541]{KO}. Here, unlike in that theorem, we also deal with the case where $p=1$ and we can observe that the condition \eqref{eq:KO-Property}, when $p=1$, would be the natural substitute for the condition $\omega\in B_{p}(\Omega)$ of \cite{KO} and in fact is exactly the same as \eqref{eq:HalfCBC}.
\end{remark}

\begin{proof}[Proof of Lemma \ref{lem:banach_space_and_loc}]
One should keep in mind that we are adopting the frequently used convention that, for each point $a$ of $\mathbb{C}^{d}$, $\left|a\right|$ denotes the $\ell^{2}$ norm of $a$. Since the $\ell^{2}$ and $\ell^{p}$ norms are equivalent on $\mathbb{C}^{d}$ we have, for every $\bm{f}\in L^{p}(U,\omega,\mathbb{C}^{d})$, that 
\begin{equation}
r\left\Vert \bm{f}\right\Vert _{L^{p}(U,\omega,\mathbb{C}^{d})}\le\left\Vert \left|\bm{f}\right|\right\Vert _{L^{p}(U,\omega)}=\left(\int_{U}\left|\bm{f}(x)\right|^{p}\omega(x)dx\right)^{1/p}\le R\left\Vert \bm{f}\right\Vert _{L^{p}(U,\omega,\mathbb{C}^{d})}\label{eq:ellPell2}
\end{equation}
 for some constants $r$ and $R$ depending only on $d$ and $p$.\\
Using H\"older's inequality and then (\ref{eq:ellPell2}), for each $\bm{f}\in L^{p}\pa{U,\omega,\mathbb{C}^{d}}$, for each compact subset $K$ of $U$ and for each $g\in L^{\infty}(U)$ which vanishes on $U\setminus K$ we obtain that 
\begin{align}
\left|\int_{U}\bm{f}(x)g(x)dx\right| & \le\int_{U}\left|\bm{f}(x)\right|\left|g(x)\right|dx\le\left\Vert g\right\Vert _{L^{\infty}(U)}\int_{U}\left|\bm{f}(x)\right|\chi_{K}(x)dx\nonumber \\
 & \le\left\Vert g\right\Vert _{L^{\infty}(U)}\left\Vert \left|\bm{f}\right|\omega^{1/p}\right\Vert _{L^{p}(U)}\left\Vert \frac{\chi_{K}}{\omega^{1/p}}\right\Vert _{L^{p'}(U)}\nonumber \\
 & =\left\Vert g\right\Vert _{L^{\infty}(U)}\left\Vert \left|\bm{f}\right|\right\Vert _{L^{p}(U,\omega)}\left\Vert \frac{\chi_{K}}{\omega^{1/p}}\right\Vert _{L^{p'}(U)}\nonumber \\
 & \le R\left\Vert g\right\Vert _{L^{\infty}(U)}\left\Vert \bm{f}\right\Vert _{L^{p}(U,\omega,\mathbb{C}^{d})}\left\Vert \frac{\chi_{K}}{\omega^{1/p}}\right\Vert _{L^{p'}(U)}\label{eq:HolderAA}
\end{align}
Then, by quite similar reasoning, for $K$ as above, but when the function $\bm{f}$ is in $L^{\infty}\pa{U,\mathbb{C}^{d}}$ and vanishes on $U\setminus K$ and when $g$ is an arbitrary element of $L^{p}\pa{U,\omega}$, we have that 
\begin{align}
\left|\int_{U}\bm{f}(x)g(x)dx\right| & \le\int_{U}\left|\bm{f}(x)\right|\left|g(x)\right|dx\le\sqrt{d}\left\Vert \bm{f}\right\Vert _{L^{\infty}(U,\omega)}\int_{U}\left|g(x)\right|\chi_{K}(x)dx\nonumber \\
 & \le\sqrt{d}\left\Vert \bm{f}\right\Vert _{L^{\infty}(U,\omega)}\left\Vert g\right\Vert _{L^{p}(U,\omega)}\left\Vert \frac{\chi_{K}}{\omega^{1/p}}\right\Vert _{L^{p'}(U)}.\label{eq:HolderB}
\end{align}
After these preparations we turn to showing that \eqref{item:W_1p_is_banach} holds: Since $\omega$ is measurable and strictly positive and since $\left\Vert \phi\right\Vert _{L^{p}(U,\omega)}\le\left\Vert \phi\right\Vert _{W^{1,p}(U,\omega)}$, it is obvious that $\left\Vert \cdot\right\Vert _{W^{1,p}\left(\omega\right)}$
is a norm. Now let $\left\{ \phi_{n}\right\} _{n\in\mathbb{N}}$ be an arbitrary Cauchy sequence in $W^{1,p}(U,\omega)$. This obviously implies that $\left\{ \phi_{n}\right\} _{n\in\mathbb{N}}$ is also a Cauchy sequence in $L^{p}(U,\omega)$ and that $\left\{ \nabla\phi_{n}\right\} _{n\in\mathbb{N}}$
is a Cauchy sequence in $L^{p}(U,\omega,\mathbb{C}^{d})$. Since $L^{p}(U,\omega)$ and $L^{p}(U,\omega,\mathbb{C}^{d})$ are both Banach spaces, these sequences $\left\{ \phi_{n}\right\} _{n\in\mathbb{N}}$ and $\left\{ \nabla\phi_{n}\right\} _{n\in\mathbb{N}}$ converge respectively, in the norms of these spaces, to an element
$\phi\in L^{p}(U,\omega)$ and an element $\bm{\xi}\in L^{p}(U,\omega,\mathbb{C}^{d})$. Therefore, in order to show that $W^{1,p}(U,\omega)$ is complete, it remains only to show that $\bm{\xi}$ is the weak gradient of $\phi$ on $U$.\\
Let $\psi$ be an arbitrary function in $C_{c}^{\infty}\left(U\right)$ and let $K_{\psi}$ denote a compact subset of $U$ on which $\psi$ is supported. We claim that
\begin{equation}
\lim_{n\to\infty}\int_{U}\psi(x)\left(\nabla\phi_{n}(x)-\bm{\xi}(x)\right)dx=0.\label{eq:02-Lim}
\end{equation}
and also that 
\begin{equation}
\lim_{n\to\infty}\int_{U}\nabla\psi(x)\left(\phi(x)-\phi_{n}(x)\right)dx=0\label{eq:01-Lim}
\end{equation}
We shall obtain these two formulae with the help of the inequalities (\ref{eq:HolderAA}) and (\ref{eq:HolderB}). We shall choose $K$ in both of these inequalities to be $K_{\psi}$. If we require $\omega$ to satisfy the compact boundedness condition or if we merely make do with requiring one of the weaker variants (\ref{eq:HalfCBC}) and (\ref{eq:KO-Property}) of this condition, we will obtain that the quantity $\left\Vert \frac{\chi_{K_{\psi}}}{\omega^{1/p}}\right\Vert _{L^{p'}(U)}$ is finite. We use this fact and substitute $\bm{f}(x)=\nabla\phi_{n}(x)-\bm{\xi}(x)$ and $g(x)=\psi(x)$ in (\ref{eq:HolderAA}). Since $\lim_{n\to\infty}\left\Vert \nabla\phi_{n}-\bm{\xi}\right\Vert _{L^{p}\pa{U,\omega,\mathbb{C}^{d}}}=0$, this gives us (\ref{eq:02-Lim}). Then we substitute $\bm{f}(x)=\nabla\psi(x)$ and $g(x)=\phi(x)-\phi_{n}(x)$ in (\ref{eq:HolderB}). Since $\lim_{n\to\infty}\left\Vert \phi-\phi_{n}\right\Vert _{L^{p}\pa{U,\omega}}=0$, this gives us (\ref{eq:01-Lim}). \\
Using (\ref{eq:02-Lim}) and then (\ref{eq:01-Lim}), we see that 
\begin{align*}
\int_{U}\nabla\psi(x)\phi(x)dx & =\lim_{n\rightarrow\infty}\int_{U}\nabla\psi(x)\phi_{n}(x)dx\\
 & =-\lim_{n\rightarrow\infty}\int_{U}\psi(x)\nabla\phi_{n}(x)dx\\
 & =-\int_{U}\psi(x)\bm{\xi}(x)dx
\end{align*}
for each $\psi\in C_{c}^{\infty}(U)$. This shows that $\bm{\xi}$ is indeed the weak gradient of $\phi$ on $U$ and completes the proof of part \eqref{item:W_1p_is_banach} of the lemma, also when $\omega$ only satisfies (\ref{eq:HalfCBC}) or (\ref{eq:KO-Property}).\\
We now turn our attention to part \eqref{item:W_1p_in_loc}. Here our hypothesis on $\omega$ is either that it satisfies the compact boundedness condition, or merely one ``half'' of that condition, namely (\ref{eq:HalfCBC}). \\ 
For each compact subset $K$ of $U$ and each $\phi\in W^{1,p}(U,\omega)$ we have, recalling the notation introduced in (\ref{eq:Define_m_and_M}), that 
\begin{align*}
\left\Vert \chi_{K}\phi\right\Vert _{L^{p}(U)}^{p}+\left\Vert \chi_{K}\nabla\phi\right\Vert _{L^{p}\pa{U,\mathbb{C}^{d}}}^{p} & \le\frac{1}{m(K,\omega)^{p}}\left(\left\Vert \chi_{K}\phi\right\Vert _{L^{p}\pa{U,\omega}}^{p}+\left\Vert \chi_{K}\nabla\phi\right\Vert _{L^{p}\pa{U,\omega,\mathbb{C}^{d}}}^{p}\right)\\
 & \le\frac{1}{m(K,\omega)^{p}}\left(\left\Vert \phi\right\Vert _{L^{p}\pa{U,\omega}}^{p}+\left\Vert \nabla\phi\right\Vert _{L^{p}\pa{U,\omega,\mathbb{C}^{d}}}^{p}\right)\\
 & =\frac{1}{m(K,\omega)^{p}}\left\Vert \phi\right\Vert _{W^{1,p}\pa{U,\omega}}^{p}.
\end{align*}
This establishes the inclusion claimed in part \eqref{item:W_1p_in_loc}, and therefore completes the proof of the lemma.
\end{proof} 
\begin{remark}\label{rem:UnnecessaryHypothesesZZ}
In the formulations of several of the following results of this section we will find it convenient to require the relevant weight function to satisfy the compact boundedness condition, even though, just as for the preceding lemma, some of these results hold under weaker hypotheses. 
\end{remark}
\begin{lemma}\label{lem:lip_in_W_p}
Let $\omega$ be a weight function that satisfies the compact boundedness condition on an open set $U$. Then $Lip_c\pa{U} \subset W^{1,p}\pa{U,\omega}$ for each $p\in[1,\infty]$.
\end{lemma}
\begin{proof}
Recalling Rademacher's theorem, Theorem \ref{thm:rademacher}, we know that if $\phi\in Lip_c\pa{U}$ then it is in $W^{1,\infty}_{loc}\pa{U}$, and therefore has a weak gradient $\nabla \phi$ on $U$. Obviously (cf. Remark \ref{rem:UtterlyObvious}) $\nabla\phi\in L^{\infty}(U,\mathbb{C}^{d})$ and $\left\Vert \nabla\phi\right\Vert _{L^{\infty}(U,\mathbb{C}^{d})}$ is bounded by the Lipschitz constant of $\phi$. Also $\phi$, like every other compactly supported Lipschitz function, is bounded. This completes the proof in the case where $p=\infty$.\\
To deal with the case $p\in[1,\infty)$, we note that $\phi$ and $\nabla \phi$ both vanish almost everywhere on $U\setminus \text{supp}\phi$, and that the Lebesgue measure of $\text{supp}\phi$ is finite. 
Then, using the notation of \eqref{eq:Define_m_and_M}, and once more the boundedness of $\phi$ we conclude that
$$\int_{U}\abs{\phi(x)}^p \omega(x)dx \leq \norm{\phi}_{L^\infty\pa{U}}^p \abs{\text{supp}\phi}M\pa{\text{supp}\phi,\omega}<\infty\footnote{Here (and also elsewhere) we use the standard notation $\abs{E}$ for the $d-$dimensional Lebesgue measure of any measurable subset $E$ of $\R^d$.}$$
and
$$\int_{U}\abs{\nabla \phi(x)}^p \omega(x)dx \leq \norm{\nabla \phi}_{L^\infty\pa{U,\C^d}}^p \abs{\text{supp}\phi}M\pa{\text{supp}\phi,\omega}<\infty.$$
This shows that $\phi\in W^{1,p}\pa{U,\omega}$ also for $p\in[1,\infty)$, completing the proof.
\end{proof}
\begin{lemma}\label{lem:product_rule} 
Let $\omega$ be a weight function that satisfies the compact boundedness condition on an open set $U$. Then, given $p\in[1,\infty)$, and functions $\phi\in W^{1,p}\left(U,\omega\right)$ and $\xi:U\rightarrow \C$ such that $\xi$ is Lipschitz on $U$, we have that their pointwise product $\phi\xi$ has a weak gradient which can be expressed in terms of the weak gradients of $\phi$ and $\xi$ by the formula 
$$
\nabla\left(\phi\xi\right)(x)=\xi(x)\nabla\phi(x)+\phi(x)\nabla\xi(x) \quad \text{for almost every }\;x\in U.
$$
\end{lemma}
\begin{proof}
Much as in the proof of the previous lemma, we know, due to Rademacher's Theorem and Remark \ref{rem:UtterlyObvious}, that $\xi\in W^{1,\infty}_{loc}\pa{U}$ (and in fact $\xi\in W^{1,\infty}\pa{U}$), and consequently $\xi\in W^{1,r}_{loc}\pa{U}$ for every $r\geq 1$. We also note that, due to Lemma \ref{lem:banach_space_and_loc}, $\phi\in W^{1,p}_{loc}\pa{U}$. Let $\psi$ be an arbitrary function in $C_c^\infty\pa{U}$ and let $K_\psi=\text{supp}\psi$. Standard approximation theorems which are used in the theory of Sobolev spaces (see for example Theorem 1 of \cite[pp. 250-251]{Evans} and its proof, which is based on Theorem 6 of \cite[pp. 630-631]{Evans}) enable us to find an open set $V$ containing $K_\psi$, whose closure is compact and contained in $U$, and functions $\xi_n:U\rightarrow\C$ whose restrictions to $V$ are in $C^\infty\pa{V}$ and which satisfy
\begin{equation}\label{eq:NiceXi}
\lim_{n\to\infty}\int_{V}\left|\xi_{n}(x)-\xi(x)\right|^{q}dx=0
\end{equation}
 and 
\begin{equation}\label{eq:NiceNablaXi}
\lim_{n\to\infty}\int_{V}\left|\nabla\xi_{n}(x)-\nabla\xi(x)\right|^{q}dx=0.
\end{equation}
where $q$ is the H\"older conjugate of $p$. Since $\psi(x)\xi_n(x)\in C_c^\infty\pa{U}$ for each $n\in\N$ we see that
$$\int_{U} \phi(x)\xi_n(x)\nabla \psi(x)dx = \int_{U}\phi(x) \nabla \pa{\psi(x) \xi_n(x)}dx - \int_{U}\phi(x)\psi(x)\nabla \xi_n(x)dx$$
$$=-\int_{U}\psi(x)\pa{\xi_n(x)\nabla \phi(x)+\phi(x)\nabla \xi_n(x) }dx.$$
We recall that $\psi$ and $\nabla\psi$ both vanish on $U\setminus V$. Therefore, in all the integrals in the preceding calculation, we can replace the domain of integration $U$ by $V$. Since $\phi\in W^{1,p}_{loc}\pa{U}$ (due to \eqref{item:W_1p_in_loc} of Lemma \ref{lem:banach_space_and_loc}), we also observe that $\phi\nabla\psi$ and $\psi\nabla\phi$ are both elements of $L^{p}(U,\mathbb{C}^{d})$ and $\psi\phi\in L^{p}(U)$.
So we can use \eqref{eq:NiceXi}, \eqref{eq:NiceNablaXi} and H\"older's inequality to pass to the limit in the preceding calculation and thus to obtain that
$$\int_{U}\phi(x)\xi(x)\nabla \psi(x) dx =-\int_{U}\psi(x)\pa{\xi(x)\nabla \phi(x)+\phi(x)\nabla \xi(x)}dx.$$
Since this holds for every $\psi\in C^\infty_c\pa{U}$, the proof is complete.
\end{proof}
From this point onward, till the end of this section, we will only deal with the case where $U=\R^d$, and, accordingly, we shall usually remove any mention of the underlying set (namely $\R^d$) from our notation.\\
Note that in the following lemma we will be permitting a function $\rho$ which seems to be playing the role of a weight function to possibly also take the value $0$. So, analogously to the conventions adopted in Definition \ref{def:SemiNorm} and in \eqref{eq:SemiNorm}, it may happen that $L^{p}(\rho)$ and $\left\Vert \cdot\right\Vert _{L^{p}(\rho)}$ have to be understood to be, respectively, a semi-normed space and the semi-norm on that space. 
\begin{lemma}\label{lem:compact_apporximation_general}
Let $\rho:\mathbb{R}^{d}\to[0,\infty)$ and $\phi:\mathbb{R}^{d}\to\mathbb{C}$ be arbitrary measurable functions which satisfy 
$$\int_{\R^d}\abs{\phi(x)}^p \rho(x)dx <\infty,$$
for some $p\in [1,\infty)$. For each $n\in\N$ define the function $\xi_n:\R^d\rightarrow[0,\infty)$ by 
\begin{equation}\label{eq:xi_n}
\xi_n(x) = \begin{cases}
1 & \abs{x}<n \\
2-\frac{\abs{x}}{n} & n<\abs{x} < 2n \\
0 & \abs{x}\geq 2n
\end{cases}
\end{equation}
Then, 
\begin{enumerate}[(i)]
\item\label{item:xi_is_lip} $\xi_{n}\in Lip_{c}\left(\mathbb{R}^{d}\right)$ with Lipschitz constant $\frac{1}{n}$, and at almost every point $x\in\mathbb{R}^{d}$ the $d$-dimensional $\ell^{p}$norm of its weak gradient satisfies 
\begin{equation}\label{eq:ptwiseGradient-EllP_estimate}
\left\Vert \nabla\xi{}_{n}(x)\right\Vert _{\ell^{p}}\le\frac{d^{\frac{1}{p}}}{n}.
\end{equation}
\item\label{item:phi_times_xi_approximates_phi} $\phi\xi_{n}\in L^{p}\left(\rho\right)$ and 
\begin{equation}\label{eq:compact_approximation_L_p}
\int_{\mathbb{R}^{d}}\left\vert \phi(x)\xi_{n}(x)-\phi(x)\right\vert ^{p}\rho(x)dx\underset{n\rightarrow\infty}{\longrightarrow}0.
\end{equation}
\end{enumerate}
\end{lemma}
\begin{proof}
For part \eqref{item:xi_is_lip} the fact that $\xi_n$ is in $Lip_c\pa{\R^d}$ with Lipschitz constant $\frac{1}{n}$ is immediate. It is also obvious that, at each point $x\in\mathbb{R}^{d}$ such that $\left|x\right|\ne n$ and $\left|x\right|\ne2n$, the pointwise gradient $\nabla\xi_{n}(x)$ exists and equals either $0$ or $-\frac{x}{n\left|x\right|}.$\\
This gives us \eqref{eq:ptwiseGradient-EllP_estimate} since every $x\in\R^d$ satisfies 
$$\left\Vert x\right\Vert _{\ell^{p}}\le d^{\frac{1}{p}}\left\Vert x\right\Vert _{\ell^{\infty}}.$$
For part \eqref{item:phi_times_xi_approximates_phi} we note that $\left\vert \phi(x)\xi_{n}(x)\right\vert \leq\left\vert \phi(x)\right\vert $ which ensures that $\phi\xi_{n}\in L^{p}(\rho)$. We also have
$$\left\vert \phi(x)\xi_{n}(x)-\phi(x)\right\vert ^{p}=\left|\phi(x)\right|^{p}\left|1-\xi_{n}(x)\right|^{p}\le\left|\phi(x)\right|^{p}$$
for all $n\in\mathbb{N}$ and for almost every $x\in\mathbb{R}^{d}$. Since $\left|\phi\right|^{p}\rho\in L^{1}$ and since $\phi(x)\xi_n(x)$ converges pointwise to $\phi(x)$ we obtain \eqref{eq:compact_approximation_L_p} as a consequence of the Dominated Convergence Theorem.
\end{proof}
\begin{lemma}\label{lem:compact_approximation_W_p}
Let $\nu$ be a weight function that satisfies the compact boundedness condition and let $\xi_n$ be as in \eqref{eq:xi_n}. Then, for each $p\in[1,\infty)$ and for every $\phi\in W^{1,p}\pa{\nu}$, we have that $\phi\xi_n\in W^{1,p}\pa{\nu}$ and
\begin{equation}\label{eq:compact_approximation_W_p}
\norm{\phi-\phi\xi_n}_{W^{1,p}\pa{\nu}}\underset{n\rightarrow\infty}{\longrightarrow}0.
\end{equation}
\end{lemma}
\begin{proof}
Our hypothesis on $\nu$ enables us to use Lemma \ref{lem:product_rule} and so we see that $\phi\xi_n$ has a weak gradient on $\R^d$, which is given by
$$\xi_n(x)\nabla \phi(x)+ \phi(x) \nabla \xi_n(x).$$
Since $\xi_n$ and $\nabla \xi_n$ are bounded, for each fixed $n$, and since $\phi\in W^{1,p}\pa{\nu}$ we see that $\phi\xi_{n}\in L^{p}\left(\nu\right)$
and $\nabla\left(\phi\xi_{n}\right)\in L^{p}\left(\nu,\mathbb{C}^{d}\right)$. This implies that $\phi\xi_n\in W^{1,p}\pa{\nu}$. 
Moreover, with the help of \eqref{eq:ptwiseGradient-EllP_estimate}, we have that 
$$\norm{\nabla \phi- \nabla \pa{\phi\xi_n}}_{L^p\pa{\nu,\C^d}} \leq \norm{\nabla \phi - \xi_n\nabla \phi}_{L^p\pa{\nu,\C^d}}+\norm{\phi \nabla \xi_n}_{L^p\pa{\nu,\C^d}}$$
$$\leq \pa{\int_{\R^d} \abs{1-\xi_n(x)}^p \abs{\nabla\phi(x)}^p_{\ell^p}\nu(x)dx}^{\frac{1}{p}}+\frac{d^{\frac{1}{p}}}{n}\norm{\phi}_{L^p\pa{\nu}}.$$
So, as in the proof of part \eqref{item:phi_times_xi_approximates_phi} of Lemma \ref{lem:compact_apporximation_general}, the Dominated Convergence Theorem obviously provides what is needed to show that
$$\lim_{n\to\infty}\left\Vert \nabla\phi-\nabla\left(\phi\xi_{n}\right)\right\Vert _{L^{p}\left(\nu,\mathbb{C}^{d}\right)}=0.$$
And of course Lemma \ref{lem:compact_apporximation_general} also gives us that 
$$\lim_{n\rightarrow\infty}\norm{\phi-\phi\xi_n}_{L^p\pa{\nu}}=0.$$
 These two conditions show that \eqref{eq:compact_approximation_W_p} holds, and thus complete the proof of the present lemma. 
\end{proof}
A very useful feature of Lemmas \ref{lem:compact_apporximation_general} and \ref{lem:compact_approximation_W_p} is the fact that the approximating sequence in the relevant semi-normed weighted $L^{p}$ space or weighted Sobolev space is always constructed in the same way, regardless of \textit{which $p$ or weight function we use!} This fact immediately proves the following obvious extension of these two lemmas which will be essential for us when we seek  to approximate functions in the normed space $\mathcal{W}^{p}(\theta,r_{\omega})$, precisely because its norm is the sum of the (semi-)norms of two \textit{differently weighted} spaces. As in Lemma \ref{lem:compact_apporximation_general} and Definition \ref{def:SemiNorm}, here again we will encounter non-negative functions acting essentially as weights. This time there will be a finite collection of possibly different non-negative functions $\rho_{\beta}$ and semi-normed spaces $L^{q_{\beta}}(\rho_{\beta})$ associated with them.
\begin{lemma}\label{lem:multiple_compact_apporimation}
Let $\nu_{1},\dots,\nu_{j}$ be weight functions that satisfy the compact boundedness condition, and let $\rho_{1},\dots,\rho_{k}$ be arbitrary non-negative measurable functions. Given $j+k$ exponents $p_{1},...,p_{j},q_{1},....,q_{k}$ in $[1,\infty)$, we let $X$ be the linear space $X=\left(\bigcap_{\alpha=1}^{j}W^{1,p_{\alpha}}\left(\nu_{\alpha}\right)\right)\cap\left(\bigcap_{\beta=1}^{k}L^{q_{\beta}}\left(\rho_{\beta}\right)\right)$ equipped with the norm 
$$
\left\Vert \phi\right\Vert _{X}:=\sum_{\alpha=1}^{j}\left\Vert \phi\right\Vert _{W^{1,p_{\alpha}}\left(\nu_{\alpha}\right)}+\sum_{\beta=1}^{k}\left\Vert \phi\right\Vert _{L^{q_{\beta}}\left(\rho_{\beta}\right)}.$$
For each $n\in\mathbb{N}$, let $\xi_{n}$ be the function defined by \eqref{eq:xi_n}. Then, for every $\phi\in X$, the function $\phi\xi_{n}$ is also in $X$ and 
$$
\lim_{n\to\infty}\left\Vert \phi-\phi\xi_{n}\right\Vert _{X}=0.
$$
\end{lemma}
Now that we know how to approximate functions of interest to us by compactly supported functions, our next step will be to approximate those approximating functions in turn by smoother functions, not merely Lipschitz functions but even elements of $C_{c}^{\infty}\pa{\mathbb{R}^{d}}$.
An essential ingredient for doing this is the following standard result:
\begin{lemma}\label{lem:mollifier}
Let $\eta \in C_c^{\infty}\pa{\R^d}$ be a given non-negative function which is supported in $B_1(0)$ and satisfies $\int_{\R^d}\eta(x)dx=1$. For $p\in[1,\infty)$ let $\phi$ be a given function in $L^p\pa{\R^d}$ that is compactly supported. Define:
$$\phi_n(x) := n^{d}\int_{\R^d} \eta\pa{n(x-y)}\phi(y)dy = \eta_n \ast \phi (x),$$
where $\eta_n(x)=n^d\eta\pa{nx}$. Then
\begin{enumerate}[(i)]
\item\label{item:C_c_molZZ} $\phi_n\in C_c^\infty\pa{\R^d}$ and is supported in $D_n=\br{x\in \R^d\;|\; \text{dist}\pa{x,\text{supp}\phi}\leq \frac{1}{n}}$.
\item\label{item:convergence_in_L_p_molZZ} $\phi_n$ converges to $\phi$ in $L^p\pa{\R^d}$.
\item\label{item:convergence_in_W_p_molZZ} If in addition $\phi\in W^{1,p}\pa{\R^d}$ we have that $\phi_n\in W^{1,p}\pa{\R^d}$ and 
$$\lim_{n\to\infty}\norm{\phi-\phi_n}_{W^{1,p}\pa{\R^d}}=0.$$
\end{enumerate}
\end{lemma}
We refer to standard textbooks for the standard proof \footnote{For example, it is a straightforward variant of the proof of Theorem 6 of \cite[p. 630]{Evans}. See relevant material on pp. 629-631 of \cite{Evans} and also, for example, on pp. 148-149 of \cite{WheedenRZygmundA1977}.} of this result.\\
We will now obtain a variant of Lemma \ref{lem:multiple_compact_apporimation} as our next step towards proving our principal result in this section, Theorem \ref{thm:W_p_is_W_p_0}. This variant will be stated in much greater generality than we need for our immediate purposes here - but that comes at no extra cost in the proof. That greater generality may well turn out to be useful in future extensions of this research.\\
\begin{theorem}\label{thm:multiple_compact_smooth_apporimation}
Let $X$ be the normed space introduced in Lemma \ref{lem:multiple_compact_apporimation} and suppose that each of the weight functions $\nu_{1},\dots,\nu_{j}$
appearing in its definition satisfies the compact boundedness condition. Suppose also that each of the non-negative measurable functions $\rho_{1},\dots,\rho_{k}$
appearing in its definition is bounded from above on every compact set and that 
\begin{equation}\label{eq:MaxPandMaxQ}
\max\left\{ q_{1},....,q_{k}\right\} \le\max\left\{ p_{1},...,p_{j}\right\} .
\end{equation}
Then $C_{c}^{\infty}(\mathbb{R}^{d})$ is a dense subspace of $X$.
\end{theorem}
\begin{proof} 
The inclusion $C_{c}^{\infty}(\mathbb{R}^{d})\subset X$ follows immediately from the fact that all the functions $\nu_{1},\dots,\nu_{j}$ and $\rho_{1},\dots,\rho_{k}$ are bounded above on every compact set. \\
Let $\phi$ be an arbitrary compactly supported element of $X$ and let $\phi_{n}=\eta_{n}\ast\phi$ for each $n\in\mathbb{N}$, where the functions $\eta_{n}$ are as specified in Lemma \ref{lem:mollifier}. Then, by Lemma \ref{lem:mollifier}, all of the functions $\phi_{n}$ are in $C_{c}^{\infty}(\mathbb{R}^{d})$ and they and their gradients vanish on the complement of the compact set $D_{1}=\left\{ x\in\mathbb{R}^{d}:\mathrm{dist}(x,\mathrm{supp}\phi)\le2\right\} $.
Furthermore, our hypotheses ensure that the supremum 
$$M:=\sup_{x\in D_1}\left\{ \max_{\alpha\in\{1,...,j\}, \beta\in\left\{ 1,...,k\right\}}\left\{ \nu_{\alpha}(x),\rho_{\beta}(x)\right\} \right\} $$
is finite. This enables us to use the inequalities 
\begin{equation}\label{eq:rhoStuff}
\left\Vert \phi-\phi_{n}\right\Vert _{L^{q_{\beta}}(\rho_{\beta})}\le M^{\frac{1}{q_{\beta}}}\left\Vert \phi-\phi_{n}\right\Vert _{L^{q_{\beta}}(\mathbb{R}^{d})}
\end{equation}
and 
\begin{equation}\label{eq:omegaStuff}
\left\Vert \phi-\phi_{n}\right\Vert _{W^{1,p_{\alpha}}(\nu_{\alpha})}\le M^{\frac{1}{p_{\alpha}}}\left\Vert \phi-\phi_{n}\right\Vert _{W^{1,p_{\alpha}}(\mathbb{R}^{d})}
\end{equation}
which hold for all $n\in\mathbb{N}$, $\alpha\in\{1,...,j\}$ and $\beta\in\left\{ 1,...,k\right\} $.\\
The fact that $\inf_{x\in D_{1}}\nu_{\alpha}(x)>0$ ensures that $\phi\in W^{1,p_{\alpha}}\pa{\mathbb{R}^{d}}$ and therefore also $\phi\in L^{p_{\alpha}}(\mathbb{R}^{d})$ for each $\alpha\in\left\{ 1,...,j\right\} $.
Consequently, in view of \eqref{eq:MaxPandMaxQ}, H\"older's inequality and the fact that the support of $\phi$ is compact, we also have that $\phi\in L^{q_{\beta}}(\mathbb{R}^{d})$ for each $\beta\in\left\{ 1,..,k\right\} $. The membership of $\phi$ in all of these unweighted $L^{p}$ spaces permits us to use Lemma \ref{lem:mollifier} to obtain that the right sides of \eqref{eq:rhoStuff} and \eqref{eq:omegaStuff} tend to $0$ as $n$ tends to $\infty$ for all $\alpha\in\{1,...,j\}$ and $\beta\in\left\{ 1,...,k\right\}$. 
Consequently we also have that 
$$\lim_{n\to\infty}\left\Vert \phi-\phi_{n}\right\Vert _{X}=0.$$
This shows that every compactly supported element in $X$ can be approximated in $X$ norm by a sequence of functions in $C^\infty_{c}\pa{\mathbb{R}^{d}}$. Since we already know from Lemma \ref{lem:multiple_compact_apporimation} that every element in $X$ can be approximated in $X$ norm by a sequence of compactly supported functions in $X$, this completes the proof of the theorem.
\end{proof}
We are now ready to prove Theorem \ref{thm:W_p_is_W_p_0}:
\begin{proof}[Proof of Theorem \ref{thm:W_p_is_W_p_0}]
We shall in fact prove a stronger result, namely that $C_{c}^{\infty}(\mathbb{R}^{d})$ is dense in each of the spaces $W^{1,p}\left(\mathbb{R}^{d},\omega_{0}\right)$, $W^{1,p}\left(\mathbb{R}^{d},\omega_{1}\right)$ and $\mathcal{W}^{p}\left(\mathbb{R}^{d},\theta,r_{\omega}\right)$. This of course will imply \eqref{eq:W_p_is_Etc-1} and \eqref{eq:W_p_is_Etc-2}, since $C_{c}^{\infty}\left(\R^d\right)\subset Lip_{c}\left(\R^d\right)$. We begin by applying Theorem \ref{thm:multiple_compact_smooth_apporimation} in the special case where
$$j=k=1,\quad p_1=q_1=p,\quad \nu_1=\omega_m,\quad \rho_1=0$$
and where $m=0$ or $m=1$. In this case the theorem shows that $C_c^\infty\pa{\R^d}$ is dense in $X=W^{1,p}\left(\mathbb{R}^{d},\omega_{m}\right)$. This establishes \eqref{eq:W_p_is_Etc-1}.\\
It remains to prove \eqref{eq:W_p_is_Etc-2}. We start by recalling (cf. Remark \ref{rem:CBC-obvious}) that if $\omega_{0}$ and $\omega_{1}$ satisfy the compact boundedness condition, then obviously so do $\omega_{\theta}=\omega_{0}^{1-\theta}\omega_{1}^{\theta}$ and $r_{\omega}$.
As $r_{\omega}$ is locally Lipschitz, we use Rademacher's theorem to conclude that it is differentiable almost everywhere with essentially bounded gradient on each compact
set. Consequently, due to the strict positivity of  $r_{\omega}$, the function $\log r_{\omega}$ is also differentiable almost everywhere and its gradient, given by 
$$\nabla\log\left(r_{\omega}(x)\right)=\frac{\nabla r_{\omega}(x)}{r_{\omega}(x)},$$
is also essentially bounded on each compact set.\\
We summarise the above:
\begin{itemize}
\item $\omega_\theta$ satisfies the compact boundedness condition.
\item $\omega_\theta \abs{\nabla \log\pa{\rw}}^p$ is essentially bounded from above on every compact set.
\end{itemize}
These two properties are exactly what we need to enable us to apply Theorem \ref{thm:multiple_compact_smooth_apporimation} once more, where this time we choose $j=k=1$, $p_{1}=q_{1}=p$, $\nu_{1}=\omega_{\theta}$ and $\rho_{1}=\omega_{\theta}\left\vert \nabla\log r_{\omega}\right\vert ^{p}$.
These choices make $X$ coincide with $\mathcal{W}^{p}\left(\theta,\omega_{\theta}\right)$ to within equivalence of norms, and so the density of $C_{c}^{\infty}(\mathbb{R}^{d})$ in $X$, which is again guaranteed by Theorem \ref{thm:multiple_compact_smooth_apporimation}, establishes \eqref{eq:W_p_is_Etc-2} and therefore completes the proof of the theorem. 
\end{proof}
In all of our results in this section from Lemma \ref{lem:compact_apporximation_general} onwards we have been dealing only with the case where the underlying open set $U$ is all of $\mathbb{R}^{d}$. But let us conclude this section with the following remark about the possibility of extending our approximation theorems to the case where $U$
is a proper open subset of $\mathbb{R}^{d}$: 
\begin{remark}\label{rem:extension_of_approximationZZ}
As the reader may have noticed from the various steps leading to the proof of Theorem \ref{thm:multiple_compact_smooth_apporimation}, the two main ingredients which we used to show the theorem were the fact that we can approximate a function in $W^{1,p}\pa{\R^d,\omega}$ by a compactly supported function and Lemma \ref{lem:mollifier}. As a more general version of Lemma \ref{lem:mollifier} can be shown to be valid for \textit{any} open set $U\subset \R^d$ instead of $\R^d$ (see e.g. \cite{Evans} pp.
629-631), we see that if the open set $U$ and the weight function $\omega$ are such that any $\phi \in W^{1,p}\pa{U,\omega}$ can be approximated by compactly supported functions, then we will have that
$$W^{1,p}_0\pa{U,\omega}=W^{1,p}\pa{U,\omega}, \quad\quad \W^p_0\pa{U,\theta,\rw}=\W^p\pa{U,\theta,\rw},$$
and also that $C_c^\infty\pa{U}$ is dense in $W^{1,p}\pa{U,\omega}$ and in $\W^p\pa{U,\theta,\rw}$.\\
This raises the question of what properties of $U$ and $\omega$ might suffice to ensure that compactly supported functions are dense in $W^{1,p}(U,\omega)$. It would seem that this density property will hold if $\omega(x)$ is required to tend to $0$ sufficiently rapidly when $x$ approaches the boundary of $U$. Thus it might be natural to consider this question for weight functions which depend solely on the distance to the boundary, such as those which appear in the weighted Sobolev spaces discussed on pages 17-20 of \cite{Kufner}.
\end{remark}

\section{Proofs of the main theorem, and of its main consequences}\label{sec:proof}
In this short section we pool all the resources that we have developed in Sections \S\ref{sec:simple_inclusion}, \S\ref{sec:difficult_inclusion} and \S\ref{sec:approximation} to prove our main theorem, Theorem \ref{thm:main} and its main consequences: Theorems \ref{thm:main_for_W_0}, \ref{thm:main_for_U_including_R_d} and \ref{thm:Stein_Weiss_for_W}.
\begin{proof}[Proof of Theorem \ref{thm:main}]
Using Theorem \ref{thm:easy_inclusion} we obtain the inclusion \eqref{eq:easy_inclusion_thm} and one side of the inequality \eqref{eq:main_norm}. To complete the proof it remains to establish the inclusion \eqref{eq:hard_inclusion_thm} and the other side of \eqref{eq:main_norm}.
We start doing this by observing that, since $\omega_{0}$, $\omega_{1}$, and therefore obviously also $\omega_{\theta}$, each satisfy the compact boundedness condition (cf. Remark \ref{rem:CBC-obvious}), we can apply Lemma \ref{lem:banach_space_and_loc} and Lemma \ref{lem:lip_in_W_p} to obtain that 
\begin{enumerate}
\item\label{item:FirstThing} For every $p\in[1,\infty)$, each of the spaces $W^{1,p}\left(\omega_{0}\right)$, $W^{1,p}\left(\omega_{1}\right)$ and also $W^{1,p}\left(\omega_{\theta}\right)$, for every $\theta\in(0,1)$, is a Banach space that contains $Lip_{c}$.
\end{enumerate}
One of the hypotheses of Theorem \ref{thm:main} is that $r_{\omega}$ is locally Lipschitz on $U$. The strict positivity of $r_{\omega}$ therefore implies that 
\begin{enumerate}\setcounter{enumi}{1}
\item\label{item:B-SecondThing} The function $\log r_{\omega}$ is locally Lipschitz on $U$.
\end{enumerate}
The above two conditions \eqref{item:FirstThing} and \eqref{item:B-SecondThing} tell us (cf. Definition \ref{def:admissibility}) that the pair $\left(\omega_{0},\omega_{1}\right)$ is $\left(\theta,p\right)-$admissible for every $\theta\in(0,1)$ and $p\in[1,\infty)$. 
This permits us to invoke Theorem \ref{thm:difficult_inclusion_on_W_0}, which applies to any open subset $U$ of $\mathbb{R}^{d}$ and whose conclusions \eqref{eq:difficult_inclusion_subset} and \eqref{eq:difficult_inclusion_norm} are exactly the above mentioned inclusion and inequality which are required to complete the proof of Theorem \ref{thm:main}.
\end{proof}
\begin{proof}[Proof of Theorem \ref{thm:main_for_W_0}]
The first of the two inclusions in \eqref{eq:inclusion_for_W_0_main_simplified} has already been established as part of the proof of Theorem \ref{thm:main}.
The additional condition imposed in Theorem \ref{thm:main_for_W_0} for some $p\in[1,\infty)$ is exactly what is required to enable Theorem \ref{thm:SimultaneousApprox} to be applied in order to establish the second inclusion in \eqref{eq:inclusion_for_W_0_main_simplified} for that value of $p$ and for all $\theta\in(0,1)$.
\end{proof}
\begin{proof}[Proof of Theorem \ref{thm:main_for_U_including_R_d}]
The proof of Theorem \ref{thm:main} shows that the inclusions \eqref{eq:easy_inclusion_thm} and \eqref{eq:hard_inclusion_thm} hold also here. So, to prove the first part of Theorem \ref{thm:main_for_U_including_R_d} we simply substitute \eqref{eq:Wdensity01} and \eqref{eq:Wdensity02} in \eqref{eq:easy_inclusion_thm} and \eqref{eq:hard_inclusion_thm} to obtain \eqref{eq:main_for_U_including_R_d}. 
To prove the second part of the theorem, in which it is assumed that $U=\mathbb{R}^{d}$, we simply apply Theorem \ref{thm:W_p_is_W_p_0} to give us that \eqref{eq:Wdensity01} and \eqref{eq:Wdensity02} both hold for all $p\in[1,\infty)$.
\end{proof}
\begin{proof}[Proof of Theorem \ref{thm:Stein_Weiss_for_W}]
Since $g(x)$ is continuous and the exponential function is continuous and strictly positive, we conclude that if $\omega_1$ (or $\omega_0$) satisfies the compact boundedness condition, then so does $\omega_0$ (or $\omega_1$). Next, we note that
$$\rw(x)=e^{g(x)},$$
which implies, as $g(x)$ is Lipschitz and the exponential function is $C^1$ on $\R$, that $\rw(x)$ is locally Lipschitz. 
Thus we have shown that the conditions of Theorem \ref{thm:main} are satisfied, as claimed in part \eqref{item:conditions_are_valid} of Theorem \ref{thm:Stein_Weiss_for_W}. We next show part \eqref{item:SW_on_Rd}, namely that \eqref{eq:SteinWeissOnRd} holds to within equivalence of norms
when $U=\mathbb{R}^{d}$. Theorem \ref{thm:main} (or even simply Theorem \ref{thm:easy_inclusion}) gives us ``half'' of \eqref{eq:SteinWeissOnRd},
i.e. the inclusion ``$\subset$'' and the corresponding norm inequality. So we only need to show that 
$$W^{1,p}\left(\mathbb{R}^{d},\omega_{\theta}\right)\subset\left[W^{1,p}\left(\mathbb{R}^{d},\omega_{0}\right),W^{1,p}\left(\mathbb{R}^{d},\omega_{1}\right)\right]_{\theta}$$
and to obtain the reverse norm inequality corresponding to this reverse inclusion. In view of \eqref{eq:hard_inclusion_thm} and of Theorem \ref{thm:main} and Theorem \ref{thm:W_p_is_W_p_0} we can in turn reduce this task to showing that the inclusion 
\begin{equation}\label{eq:PseudoStum}
W^{1,p}\left(\mathbb{R}^{d},\omega_{\theta}\right)\subset\mathcal{W}^{p}\left(\mathbb{R}^{d},\theta,r_{\omega}\right)
\end{equation}
holds and is continuous.\\
For any given element $\phi\in W^{1,p}\left(\omega_{\theta}\right)$
we see that 
$$ \left\Vert \phi\right\Vert _{L^{p}\left(\omega_{\theta}\left\vert \nabla\log\left(r_{\omega}\right)\right\vert ^{p}\right)}^{p}=\int_{U}\left\vert \phi(x)\right\vert ^{p}\left\vert \nabla g(x)\right\vert ^{p}\omega_{\theta}(x)dx$$
$$\leq\left\Vert \nabla g\right\Vert _{L^{\infty}}^{p}\left\Vert \phi\right\Vert _{L^{p}\left(\omega_{\theta}\right)}^{p}.$$
This implies that $\phi$ is automatically in $\mathcal{W}^{p}\left(\theta,r_{\omega}\right)$ and 
\begin{equation}\label{eq:Khadash}
\begin{gathered}
\left\Vert \phi\right\Vert _{\mathcal{W}^{p}\left(\theta,r_{\omega}\right)}=\left(\left\Vert \phi\right\Vert _{W^{1,p}\left(\omega_{\theta}\right)}^{p}+\left\Vert \phi\right\Vert _{L^{p}\left(\omega_{\theta}\left\vert \nabla\log\left(r_{\omega}\right)\right\vert ^{p}\right)}^{p}\right)^{\frac{1}{p}}\\
\leq\left(1+\left\Vert \nabla g\right\Vert _{L^{\infty}}^{p}\right)^{\frac{1}{p}}\left\Vert \phi\right\Vert _{W^{1,p}\left(\omega_{\theta}\right)}.
\end{gathered}
\end{equation}
This establishes \eqref{eq:PseudoStum} and the required associated norm inequality and thus takes care of part \eqref{item:SW_on_Rd}. (In fact \eqref{eq:Khadash} also shows that the three norms appearing in \eqref{eq:main_norm} are equivalent to each other.) Finally we establish the remaining parts \eqref{item:SW_on_W0} and \eqref{item:SW_on_U} of Theorem \ref{thm:Stein_Weiss_for_W} by using exactly the same considerations, together with Theorems \ref{thm:main_for_W_0} and \ref{thm:main_for_U_including_R_d}.
\end{proof}

\par In the last two sections of the paper we will discuss additional topics related to our study, including an example of when $\W^p\pa{\theta,\rw}$ is not $W^{1,p}\pa{\omega_\theta}$, and some potential applications and open problems.
\section{More about the space $\W^p\pa{\theta,\rw}$}\label{sec:more_about_WW_p}
In this section we will discuss two more properties of the space $\W^p\pa{\theta,\rw}$. Namely, we will answer the following two questions:
\begin{ques}\label{ques:WW_p=W_1_pZZ} 
Is it possible for $\W^p\pa{\theta,\rw}$ to not be $W^{1,p}\pa{\omega_\theta}$?
\end{ques}
\begin{ques}\label{ques:smaller_subspace_of_WW_pZZ} 
Can we find a smaller, yet ``big enough'', subspace of $\W^p\pa{\theta,\rw}$ that will allow us to only consider the weight $\omega_\theta$? 
\end{ques}
Our answers will give us more information about the space $\rpa{W^{1,p}\pa{U,\omega_0},W^{1,p}\pa{U,\omega_1}}_{\theta}$.
\subsection{Can $\W^p\pa{\theta,\rw}$ be smaller than $W^{1,p}\pa{\omega_\theta}$?}\label{subsec:W_p_is_not_W_1p} 
We consider it very unlikely that the Stein-Weiss like formula \eqref{eq:OurFormula} could hold for \textit{all} choices of the set $U$ and weight functions $\omega_{0}$ and $\omega_{1}$. Here we shall begin some attempts towards finding a counterexample.\\
One might expect that when $U$ has a very complicated structure or when some or all of the hypotheses of Theorem \ref{thm:main} do not hold, this could prevent \eqref{eq:OurFormula} from holding. But the examples which we are about to present hint that maybe \eqref{eq:OurFormula} can fail even when $U$ is simply all of $\mathbb{R}^{d}$ and the above mentioned hypotheses do hold.\\
From this point onwards we will take $U$ to be $\R^d$ and remove almost all mention of the underlying set from our notation.\\
As can be seen from \eqref{eq:main_for_U_including_R_d} in Theorem \ref{thm:main_for_U_including_R_d}, a sufficient condition for obtaining a Stein-Weiss like theorem when $U=\R^d$ is that 
$$\W^p\pa{\theta,\rw}=W^{1,p}\pa{\omega_\theta}$$
and indeed (since obviously $\W^p\pa{\theta,\rw}\subset W^{1,p}\pa{\omega_\theta}$) this condition is effectively what we have used in the proof of Theorem \ref{thm:Stein_Weiss_for_W}. (Cf. \eqref{eq:PseudoStum}.) So a natural first step in searching for the above mentioned counterexample is to look for pairs of weight functions for which
\begin{equation}\label{eq:inclusion_not_equal}
\W^p\pa{\theta,\rw}\underset{\not=}{\subset}W^{1,p}\pa{\omega_\theta}.
\end{equation}
We shall describe such pairs. In fact, for each choice of $\theta\in(0,1)$, we will find such a pair which satisfies the hypotheses of Theorem \ref{thm:main}. But we should emphasize that at this stage we do not know whether these weight functions, which satisfy \eqref{eq:inclusion_not_equal}, will also turn out to provide a counterexample to \eqref{eq:OurFormula}.
To simplify the discussion we shall make do with presenting explicit formulas for our weight functions only in the case $p=1$. However it is not difficult to modify our examples so that they apply when $p\in(1,\infty)$.\\
Let us start by stating the relevant properties that a function, say $\phi$, which is in $W^{1,p}\pa{\omega_\theta}$ but not in $\W^p\pa{\theta,\rw}$ must have: Since $\phi\in W^{1,p}\pa{\omega_\theta}$ we have that $\phi$ has a weak gradient on $\R^d$ and
\begin{equation}\label{eq:counter_phi_in_W}
\int_{\R^d}\abs{\phi(x)}^p \omega_\theta(x)dx<\infty,\quad \int_{\R^d}\abs{\nabla \phi(x)}^p \omega_\theta(x)dx<\infty.
\end{equation}
On the other hand, since $\phi\not\in \W^p\pa{\theta,\rw}$ we must have that 
\begin{equation}\label{eq:counter_phi_not_in_WW}
\int_{\R^d}\abs{\phi(x)}^p \abs{\nabla \log\pa{\rw(x)}}^p \omega_\theta(x)dx=\infty.
\end{equation}
We shall simply choose $\phi\equiv 1$. Then the gradient condition in \eqref{eq:counter_phi_in_W} is automatically satisfied, and conditions \eqref{eq:counter_phi_in_W} and \eqref{eq:counter_phi_not_in_WW} become the following two conditions on $\omega_0$ and $\omega_1$:
 \begin{equation}\label{eq:counter_example_process}
\omega_\theta\in L^1\pa{\R^d},\quad\quad \int_{\R^d}\abs{\nabla \log\pa{\frac{\omega_0(x)}{\omega_1(x)}}}^p \omega_\theta(x)dx=\infty.
\end{equation}
Let us simplify things further by also choosing $\omega_0\equiv 1$. In this case
$$\omega_\theta(x)=\omega_1(x)^{\theta},$$ 
and we can see that showing that there is a weight function $\omega_\theta$ which satisfies \eqref{eq:counter_example_process} in this case is equivalent to finding a weight function $\omega$ (which will play the role of $\omega_\theta$) such that:
\begin{equation}\label{eq:counter_example_process_simple}
\int_{\R^d}\omega(x)dx<\infty,\quad\quad \int_{\R^d}\abs{\nabla \omega(x)}^p \omega(x)^{1-p} dx=\infty.
\end{equation}
It is clear that finding such a function $\omega$ will in fact show that \eqref{eq:counter_example_process} can be satisfied for every choice of $\theta\in(0,1)$ and for the particular value of $p$ appearing in both \eqref{eq:counter_example_process_simple} and \eqref{eq:counter_example_process}.\\
At this stage we shall specialise to the case $p=1$. In this case \eqref{eq:counter_example_process_simple} is simply the statement that $\omega$ is a function in $L^1\pa{\R^d}$ whose gradient is not in $L^1\pa{\R^d,\C^d}$.
For various obvious examples of integrable functions $\omega$ which have some monotone decay at infinity (for example, like the reciprocal of a polynomial, or of an exponential) it turns out that \eqref{eq:counter_example_process_simple} does not hold. However, as we shall now see, if we modify such a function by adding a term which oscillates rapidly near infinity and is suitably bounded, then the resulting function $\omega$ does satisfy \eqref{eq:counter_example_process_simple}:
Consider the strictly positive $C^1$ function
\begin{equation}\label{eq:counter_example_weights}
\begin{gathered}
\omega(x)=\frac{1}{\pa{1+\abs{x}^2}^\alpha}\pa{1+\sin\pa{\abs{x}^\beta}+\frac{1}{1+\abs{x}^2}}\\
=\frac{1}{\pa{1+\abs{x}^2}^\alpha}\pa{1+\frac{1}{1+\abs{x}^2}}+\frac{\sin\pa{\abs{x}^\beta}}{\pa{1+\abs{x}^2}^\alpha}.
\end{gathered}
\end{equation} 
We shall now show that one can easily choose positive numbers $\alpha$ and $\beta$ for which $\omega$ will satisfy \eqref{eq:counter_example_process_simple} for $p=1$. In fact, a similar computation shows that for any $p\in (1,\infty)$, one can find an $\alpha$ and a $\beta$ for which \eqref{eq:counter_example_weights} will satisfy \eqref{eq:counter_example_process_simple}.\\
Clearly, if $\alpha > \frac{d}{2}$ then $\omega(x)\in L^1\pa{\R^d}$ since
$$\omega(x) \leq \frac{3}{\pa{1+\abs{x}^2}^\alpha}.$$
Next, as
$$\nabla \omega(x) =\frac{-2\alpha}{\pa{1+\abs{x}^2}^{\alpha+1}}\pa{1+\frac{1}{1+\abs{x}^2}} x -\frac{2x}{\pa{1+\abs{x}^2}^{\alpha+2}}$$
$$-\frac{2\alpha\sin\pa{\abs{x}^\beta}}{\pa{1+\abs{x}^2}^{\alpha+1}}x +\frac{\beta\abs{x}^{\beta-2}\cos\pa{\abs{x}^\beta}}{\pa{1+\abs{x}^2}^{\alpha}}x $$
and
$$\abs{\frac{1}{\pa{1+\abs{x}^2}^{\alpha+1}}\pa{1+\frac{1}{1+\abs{x}^2}} x} \leq \frac{1}{\pa{1+\abs{x}^2}^{\alpha}},$$
$$\abs{\frac{2x}{\pa{1+\abs{x}^2}^{\alpha+2}}}\leq \frac{1}{\pa{1+\abs{x}^2}^{\alpha+1}}$$
and
$$\abs{\frac{\sin\pa{\abs{x}^\beta}}{\pa{1+\abs{x}^2}^{\alpha+1}}x }\leq \frac{1}{\pa{1+\abs{x}^2}^{\alpha}},$$
we see that if $\alpha>\frac{d}{2}$ then $\nabla \omega \in L^1\pa{\R^d,\C^d}$ if and only if
$$I_{\alpha,\beta}:=\int_{\R^d}\abs{\frac{\abs{x}^{\beta-2}\cos\pa{\abs{x}^\beta}}{\pa{1+\abs{x}^2}^{\alpha}}x}dx<\infty.$$
We now compute
$$I_{\alpha,\beta}=\int_{\R^d}\frac{\abs{x}^{\beta-1}\abs{\cos\pa{\abs{x}^\beta}}}{\pa{1+\abs{x}^2}^{\alpha}}dx$$
$$=\abs{\mathbb{S}^{d-1}}\int_0^\infty \frac{r^{\beta+d-2}\abs{\cos\pa{r^\beta}}}{\pa{1+r^2}^\alpha}dr \geq \frac{\abs{\mathbb{S}^{d-1}}}{2^\alpha}\int_1^\infty r^{\beta+d-2-2\alpha}\abs{\cos\pa{r^\beta}}dr $$
$$\underset{y=r^\beta}{=}\frac{\abs{\mathbb{S}^{d-1}}}{2^\alpha \beta}\int_1^\infty y^{\frac{d-1-2\alpha}{\beta}}\abs{\cos\pa{y}}dr $$
If, for example, we choose $\alpha=d$ and $\beta=d+1$, then this last integral is infinite. Since we then also have  $\alpha>\frac{d}{2}$ as required in the previous part of this discussion, this completes our proof that \eqref{eq:inclusion_not_equal} can hold for $p=1$ and for every $\theta\in(0,1)$ for suitable choices of weight functions.
\begin{remark}\label{rem:ExoticWeightFunctionsZZ}
As we already mentioned, there exist other examples of pairs of weight functions which satisfy \eqref{eq:inclusion_not_equal} for $p\in(1,\infty)$. Some of them nevertheless also  satisfy \eqref{eq:OurFormula}. We refer the reader to the appendix where we give an example of a pair of weight functions (even in $C^{\infty}\pa{\mathbb{R}^{d}}$) which in fact satisfies both \eqref{eq:inclusion_not_equal} and \eqref{eq:OurFormula} for every $p\in[1,\infty)$ and $\theta\in(0,1)$.
\end{remark}

\subsection{A simpler subspace of $\left[W^{1,p}\left(U,\omega_{0}\right),W^{1,p}\left(U, \omega_{1}\right)\right]_{\theta}$}\label{subsec:simple_subspace}
The main result of this paper shows that, under relatively mild conditions, the interpolation space 
$$\left[W^{1,p}\left(U,\omega_{0}\right),W^{1,p}\left(U,\omega_{1}\right)\right]_{\theta}$$ 
lies between $W^{1,p}\left(U,\omega_{\theta}\right)$ and a smaller subspace $\mathcal{W}^{p}\left(U,\theta,r_{\omega}\right)$, which
can be expressed as $W^{1,p}\left(U,\omega_{\theta}\right)\cap X$ (with an appropriate norm) where $X$ is the semi-normed weighted space $L^{p}\left(U,\omega_{\theta}\left\vert \nabla\log r_{\omega}\right\vert ^{p}\right)$
(see Definition \ref{def:SemiNorm}).\\
For some applications, such as those which we will shortly discuss in \ref{subsec:application_evolution}, it will be
convenient to find a more simply defined space $\mathcal{X}$, which is embedded continuously
in $X$. Such a space will therefore yield the chain of inclusions 
$$
W^{1,p}\left(U,\omega_{\theta}\right)\cap\mathcal{X}\subset\left[W^{1,p}\left(U,\omega_{0}\right),W^{1,p}\left(U,\omega_{1}\right)\right]_{\theta}\subset W^{1,p}\left(U,\omega_{\theta}\right).
$$
A simple $\mathcal{X}$ to choose, if possible, would be a standard weighted Lebesgue space with respect to the weight function $\omega_{\theta}$ \textit{alone}. We shall now discuss one possible way which can enable us to choose an $\mathcal{X}$ of this kind, however with an exponent larger than $p$. It requires us to impose one more condition on our weight functions.
\begin{definition} \label{def:M-theta-q}
Let $\omega_{0}$ and $\omega_{1}$ be weight functions on an open subset $U$ of $\mathbb{R}^{d}$ and let $\omega_{\theta}$ and $r_{\omega}$ be as defined throughout this paper (by \eqref{eq:omegatheta} and \eqref{eq:rw}). Suppose that $\log r_{\omega}$ has a weak gradient on $U$. For each $q\in[1,\infty)$ and each $\theta\in[0,1]$ let $M\left(U,\theta,r_{\omega},q\right)$ be the quantity 
\begin{equation}\label{eq:condition_q_with_omega_theta}
M\left(U,\theta,r_{\omega},q\right):=\left(\int_{U}\left\vert \nabla\log\left(r_{\omega}(x)\right)\right\vert ^{q}\omega_{\theta}(x)dx\right)^{\frac{1}{q}}.
\end{equation}
(Note that here $\theta$ may also equal $0$ or $1$.) We will often use the abbreviated notation $M\left(\theta,q\right)$ for $M\left(U,\theta,r_{\omega},q\right)$.
\end{definition}

\begin{lemma}\label{lem:L_q_w_is_in_L_p_w|rw|} 
Let $U$, $\omega_{0}$, $\omega_{1}$, $\omega_{\theta}$ and $r_{\omega}$ have the properties stated in Definition \ref{def:M-theta-q}. Suppose, furthermore, that there exist $q\in(p,\infty)$ and $\theta\in(0,1)$ for which $M\left(\theta,q\right)<\infty$.
Then, in terms of the notation of Definition \ref{def:SemiNorm}, we have that
$$L^{\frac{qp}{q-p}}\left(U,\omega_{\theta}\right)\subset L^{p}\left(U,\omega_{\theta}\left\vert \nabla\log\left(r_{\omega}\right)\right\vert ^{p}\right)$$
and 
\begin{equation}\label{eq:norm_control_L_q_over_L_p_with_omega|rw|}
\left\Vert \cdot\right\Vert _{L^{p}\left(U,\omega_{\theta}\left\vert \nabla\log\left(r_{\omega}\right)\right\vert ^{p}\right)}\leq M\left(\theta,q\right)\left\Vert \cdot\right\Vert _{L^{\frac{qp}{q-p}}\left(U,\omega_{\theta}\right)}
\end{equation}
for these values of $q$ and $\theta$. 
\end{lemma}

\begin{proof} Using H\"older's inequality with the exponent $\frac{q}{p}$
we find that 
$$\int_{U}\left\vert \phi(x)\right\vert ^{p}\left\vert \nabla\log\left(r_{\omega}(x)\right)\right\vert ^{p}\omega_{\theta}(x)dx$$
$$\leq\left(\int_{U}\left\vert \nabla\log\left(r_{\omega}(x)\right)\right\vert ^{q}\omega_{\theta} (x)dx\right)^{\frac{p}{q}}\left(\int_{U}\left\vert \phi(x)\right\vert ^{\frac{pq}{q-p}}\omega_{\theta}(x)dx\right)^{\frac{q-p}{q}}. $$
This proves both the inclusion and \eqref{eq:norm_control_L_q_over_L_p_with_omega|rw|}.
\end{proof} 
With this in hand we can deduce the following: 
\begin{theorem}\label{thm:main_with_smaller_space}
Let $\omega_{0}$ and $\omega_{1}$ be a pair of weight functions on $U$ that satisfy the compact boundedness condition. Assume furthermore that 
$$r_{\omega}(x)=\frac{\omega_{0}(x)}{\omega_{1}(x)}$$
is locally Lipschitz on $U$, and that $U$, $\omega_{0}$ and $\omega_{1}$ are such that the conditions of Theorem \ref{thm:main_for_U_including_R_d} hold. Then, if $M\left(\theta,q\right)$, defined by \eqref{eq:condition_q_with_omega_theta}, is finite for some $q>p$, we have that 
$$
W^{1,p}\left(U,\omega_{\theta}\right)\cap L^{\frac{qp}{q-p}}\left(U,\omega_{\theta}\right)\subset\left[W^{1,p}\left(U,\omega_{0}\right),W^{1,p}\left(U,\omega_{1}\right)\right]_{\theta}.
$$
Moreover, there exists a constant $C_{p}>0$ which depends only on $p$,  such that for any $\phi\in W^{1,p}\left(U,\omega_{\theta}\right)\cap L^{\frac{qp}{q-p}}\left(U,\omega_{\theta}\right)$, 
\begin{equation}\label{eq:main_norm_smaller_space}
\begin{split} & \left\Vert \phi\right\Vert {}_{\left[W^{1,p}\left(U,\omega_{0}\right),W^{1,p}\left(U,\omega_{1}\right)\right]_{\theta}}\\
\leq C_p & \left(\left\Vert \phi\right\Vert _{W^{1,p}\left(U,\omega_{\theta}\right)}^{p}+M\left(\theta,q\right)^{p}\left\Vert \phi\right\Vert _{L^{\frac{qp}{q-p}}\left(U,\omega_{\theta}\right)}^{p}\right)^{\frac{1}{p}}
\end{split}
\end{equation}
\end{theorem} 
\begin{proof} 
This theorem is an immediate consequence of Theorem \ref{thm:main_for_U_including_R_d} and Lemma \ref{lem:L_q_w_is_in_L_p_w|rw|}. The inequality \eqref{eq:main_norm_smaller_space} follows immediately from  \eqref{eq:main_norm} of Theorem \ref{thm:main} and \eqref{eq:norm_control_L_q_over_L_p_with_omega|rw|}. 
\end{proof} 
Obviously, a modification of the above theorem can be made to accommodate the more general case of Theorem
\ref{thm:main} when we start considering appropriate closures.\\
 We conclude this subsection, and section, with the following remark:
\begin{remark}\label{rem:simpler_conditions_for_M_rw_qZZ} 
A simple application of H\"older's inequality shows that if $M\left(0,q\right)$ and $M\left(1,q\right)$ are finite then 
$$M\left(\theta,q\right)\leq M\left(0,q\right)^{1-\theta}M\left(1,q\right)^{\theta}<\infty\quad \text{for every} \;\theta\in(0,1).$$
 \end{remark}

\section{Final Remarks}\label{sec:final}
In this last section of our work, we reflect on possible generalisations of our main results via notions of equivalence of weight functions, and possible applications to evolution equations. Then, finally, we discuss some further possibilities for research related to the issue of interpolation of weighted Sobolev spaces.
\subsection{Equivalence of weights}\label{subsec:equivalence_of_weights} As our main results of this work depend strongly on the underlying weights of our respective spaces, one natural question can arise:
\begin{ques}\label{ques:equivalence_of_weightsZZ}
 Can one change the weight under consideration to a 'better one' (i.e. suitable for our theorems) without changing the underlying Sobolev space? 
\end{ques}
The answer to that is in the affirmative - in the obvious case when one deals with \textit{equivalent weights}.
\begin{definition}\label{def:equivalent_weightsZZ}
Adopting standard terminology, we say that two strictly positive measurable functions $\omega$ and $\eta$ are equivalent, if there exists a constant $A>0$ such that
$$
A^{-1}\eta(x) \leq \omega(x) \leq A\eta(x).
$$
\end{definition}
It is immediate to check that if $\omega_0$ and $\omega_1$ satisfy the compact boundedness condition, and are equivalent, respectively, to $\eta_0$ and $\eta_1$, then the latter functions also satisfy the same condition. Moreover, $W^{1,p}\pa{\omega}$ and $W^{1,p}\pa{\eta}$ are isomorphic to each other for equivalent weights. (Remark \ref{rem:MuchMoreGeneralThanContinuous}
has some relevance in this context.)\\
Accordingly, one can immediately state and almost immediately prove more general versions of our main theorem and its main consequences, in which the weight functions $\omega_{0}$ and $\omega_{1}$ are required merely to be equivalent to other weight functions which satisfy the requirements in the current statements of those theorems.
\subsection{Possible applications to evolution equations}\label{subsec:application_evolution} Many linear evolution equations of the form
\begin{equation}\label{eq:linear_eq}
\partial_t f(t,x)=\Li (f)(t,x)
\end{equation} 
give rise to an evolution semigroup, $\br{T(t)}_{t\geq 0}$, which acts smoothly on several spaces simultaneously. Moreover, in fields such as parabolic PDEs and kinetic theory, one frequently encounters evolution semigroups whose action on the initial datum of the problem produces functions which have increased smoothness, decay and integrability (hypercontractivity). One good example of semigroups with such behaviour are the so-called Markov semigroups. We refer the reader to \cite{BGL} for more information on the subject.\\
In many cases, one is interested in the decay, or convergence to equilibrium as $t$ tends to $\infty$, of the solution at time $t$ of \eqref{eq:linear_eq} with respect to an appropriate norm, such as $\norm{\cdot}_{W^{1,p}\pa{\R^d,\omega}}$. \\
Under suitable hypotheses, Theorem \ref{thm:main_for_U_including_R_d}, together with observations that we made in \S\ref{subsec:simple_subspace}, allows us to interpolate rates of decay/convergence between two different weighted Sobolev spaces to an intermediate space of the same kind.
More explicitly, if we know that for a given semigroup $\left\{ T(t)\right\} _{t\ge0}$ the solution to \eqref{eq:linear_eq} decays or converges to an equilibrium at rates that are given by the norms of $T(t)$ on appropriate subspaces of both $W^{1,p}\pa{U,\omega_0}$ and $W^{1,p}\pa{U,\omega_1}$, then it may be possible to deduce analogous information about the norm of $T(t)$ on a suitable subspace of $W^{1,p}\pa{U,\omega_\theta}$, for $\theta\in(0,1)$.\\
We shall now give one example of a setting where this can be done. Here we assume that the norms of $T(t)$ decay to zero. The underlying open set for all of the spaces is to be understood to be $\mathbb{R}^{d}$. 
\begin{theorem}\label{thm:application_evolutionZZ} 
Let $\left(\omega_{0},\omega_{1}\right)$ be a pair of weight functions that satisfy the compact boundedness
condition on $\mathbb{R}^{d}$, and such that $r_{\omega}$, defined by \eqref{eq:rw}, is locally Lipschitz on $\mathbb{R}^{d}$. Consider a semigroup of operators $\left\{ T(t)\right\} _{t>0}$ such that, for some fixed $p\in[1,\infty),$ 
\begin{equation}\label{eq:TtIsBounded}
T(t): \left(W^{1,p}\left(\omega_{0}\right),W^{1,p}\left(\omega_{1}\right)\right)\rightarrow\left(W^{1,p}\left(\omega_{0}\right),W^{1,p}\left(\omega_{1}\right)\right)\,\text{ for each }\; t>0.
\end{equation}
Assume furthermore that there exist $q\in(p,\infty)$, $\theta\in(0,1)$, $t_{0}>0$ and a class of initial data $\mathcal{G}$  such that for every $g\in\mathcal{G}$ 
$$T(t_{0})g\in W^{1,p}\left(\omega_{\theta}\right)\cap L^{\frac{qp}{q-p}}\left(\omega_{\theta}\right).$$
Suppose furthermore that $M\left(\theta,q\right)$, defined in \eqref{eq:condition_q_with_omega_theta}, is finite 
for these values of $\theta$ and $q$. Then we have, for every $g\in\mathcal{G}$ and $t>t_{0}$,
that $T(t)g\in W^{1,p}\left(\omega_{\theta}\right)$ and also that
$$\left\Vert T(t)g \right\Vert _{W^{1,p}\left(\omega_{\theta}\right)}\leq\mathcal{C}\left\Vert T\left(t-t_{0}\right)\right\Vert _{\mathcal{B}\left(W^{1,p}\left(\omega_{0}\right)\right)}^{1-\theta}\left\Vert T\left(t-t_{0}\right)\right\Vert _{\mathcal{B}\left(W^{1,p}\left(\omega_{1}\right)\right)}^{\theta}$$
$$\left(\left\Vert T(t_{0})g \right\Vert _{W^{1,p}\left(\omega_{\theta}\right)}+M\left(\theta,q\right)\left\Vert T(t_{0})g \right\Vert _{L^{\frac{qp}{q-p}}\left(\omega_{\theta}\right)}\right)$$
where the constant $\mathcal{C}$ depends only on $p$. 
\end{theorem}

\begin{proof}
We start by noting that, due to Theorem \ref{thm:main_with_smaller_space}, for each $g\in\mathcal{G}$ we have that 
$$T(t_{0})g\in\left[W^{1,p}\left(\omega_{0}\right),W^{1,p}\left(\omega_{1}\right)\right]_{\theta}.$$\\
Recalling Theorem \ref{thm:easy_inclusion} and the semigroup property
$$T(t)=T(t-t_{0})T(t_{0})$$
which holds for every $t\geq t_{0}$, we see that for each $g\in\mathcal{G}$ and $t>t_{0}$ we have that $T(t)g\in\left[W^{1,p}\left(\omega_{0}\right),W^{1,p}\left(\omega_{1}\right)\right]_{\theta}\subset W^{1,p}\left(\omega_{\theta}\right)$. Furthermore, using this and \eqref{eq:easy_inclusion_norm} together with \eqref{eq:TtIsBounded} and Theorem \ref{thm:interpolation_thm_basic}, we obtain that
$$\left\Vert T(t)g\right\Vert _{W^{1,p}\left(\omega_{\theta}\right)}\leq\left\Vert T(t)g\right\Vert _{\left[W^{1,p}\left(\omega_{0}\right),W^{1,p}\left(\omega_{1}\right)\right]_{\theta}}=\left\Vert T\left(t-t_{0}\right)\left(T(t_{0})g\right)\right\Vert _{\left[W^{1,p}\left(\omega_{0}\right),W^{1,p}\left(\omega_{1}\right)\right]_{\theta}}$$
$$\leq\left\Vert T\left(t-t_{0}\right)\right\Vert _{\mathcal{B}\left(W^{1,p}\left(\omega_{0}\right)\right)}^{1-\theta}\left\Vert T\left(t-t_{0}\right)\right\Vert _{\mathcal{B}\left(W^{1,p}\left(\omega_{1}\right)\right)}^{\theta}\left\Vert T(t_{0})g\right\Vert _{\left[W^{1,p}\left(\omega_{0}\right),W^{1,p}\left(\omega_{1}\right)\right]_{\theta}}.$$
Using \eqref{eq:main_norm_smaller_space} yields the inequality required to complete the proof of this theorem.
\end{proof}

\subsection{Future research and open problems}\label{subsec:future} We feel excited about future possibilities in continuing to try and understand whether or not one can obtain a Stein-Weiss like theorem for weighted Sobolev spaces under more general hypotheses on the underlying set and the weight functions. In particular, some of the problems which we think are worth investigating are the following:
\begin{ques*}
Can one find an open set $U$ and weights $\omega_0$ and $\omega_1$ on it for which
$$\W^p\pa{U,\theta,\rw}\underset{\not=}{\subset}\rpa{W^{1,p}\pa{U,\omega_0},W^{1,p}\pa{U,\omega_1}}_\theta ?$$
\end{ques*}
\begin{ques*}
Can one find necessary and sufficient conditions under which 
$$W^{1,p}_0\pa{U,\omega}=W^{1,p}\pa{U,\omega} \quad\text{and}\quad \W^p_0\pa{U,\theta,\rw}=\W^p\pa{U,\theta,\rw}?$$
\end{ques*}
Let us mention that pages 550-554 of \cite{KO} contain a discussion of the closure of $C_{c}^{\infty}(U)$ (there denoted by $C_{0}^{\infty}(U)$) in various weighted Sobolev spaces. A different but somewhat related topic is considered on pages 43-49 of \cite{Kufner}.
\begin{ques*}
Can one obtain similar results to Theorems \ref{thm:main}, \ref{thm:main_for_W_0}, \ref{thm:main_for_U_including_R_d} and \ref{thm:Stein_Weiss_for_W} for weighted homogeneous Sobolev spaces $\dot{W}^{1,p}\pa{U,\omega}$ (i.e. where we only consider the norm of the weak gradient)?\\
(The answer is evidently yes for each $p\in[1,\infty)$ in the very special case where $d=1$ and $U=\mathbb{R}$ and both of the weight functions satisfy the compact boundedness condition. In that case the maps $\phi\mapsto\phi'$ and its inverse $\psi\mapsto\int_{0}^{x}\psi(t)dt$ give us an isometric identification between the two couples 
$$\left(\dot{W}^{1,p}\left(\mathbb{R},\omega_{0}\right),\dot{W}^{1,p}\left(\mathbb{R},\omega_{1}\right)\right)\mbox{\,\,and\,\,}\left(L^{p}\left(\mathbb{R},\omega_{0}\right),L^{p}\left(\mathbb{R},\omega_{1}\right)\right)$$
and so we can apply the ``classical'' Stein-Weiss theorem. A more detailed explanation can be found in the appendix.)
\end{ques*}
We are confident that more can be done (and asked), and will be done in the next few years.
 
\appendix
\section{Additional Proofs}\label{app:additionalZZ}
In this appendix we include some additional proofs that we felt would hinder the flow of the main body of the paper. 
\subsection{An alternative characterization of the space $W_{loc}^{1,p}(U)$}\label{sub:AlternativeCharacterizationZZ}
In this part of the appendix we will provide some rather straightforward arguments which can be used to justify the following claim, made immediately after Definition \ref{def:LocalSobolev}:
\begin{claim}\label{claim:alternative_loc_space}
 $f\in W_{loc}^{1,p}(U)$ if and only if, for every open subset $V$ of $U$ whose closure is a compact subset of $U$, the restriction of $f$ to $V$ is an element of $W^{1,p}(V)$.
\end{claim} 
In order to prove this claim, we require the following notation: For a given open set $U$, and each $n\in\mathbb{N}$ we define the open set 
$$ U_{n}=\left\{ x\in U\,|\, \text{dist}(x,\partial U)>1/n\right\} \cap B_{n}(0)$$
 (where $B_n(0)$ is the open ball of radius $n$ in $\mathbb{R}^{d}$ centred at the origin). The closure of $U_{n}$ is contained in $U$ and in $\overline{B_{n}(0)}$ and is therefore a compact subset of $U$.
\begin{proof}[Proof of Claim \ref{claim:alternative_loc_space}]
Suppose that the function $f:U\to\mathbb{C}$ has the property that, for each $n\in\mathbb{N}$, its restriction $f_{n}:=\left.f\right|_{U_{n}}$ to $U_{n}$ has a weak
gradient on $U_{n}$. Let us denote this weak gradient by $g_{n}$. We do not yet need to suppose also that $f_{n}\in W^{1,p}(U_{n})$ but later, when we do,
we will also be able to assert that $g_{n}\in L^{p}(U_{n},\mathbb{C}^{d})$.\\
Due to the uniqueness of the weak gradient, and the fact that $U_{n}\subset U_{n_{1}}\subset U$ for all $n_{1}>n$, we see that if $n_{1}>n$ we have that $g_{n_{1}}|_{U_{n}}=g_{n}$. For each $n\in\mathbb{N}$ let $v_{n}:U\to\mathbb{C}^{d}$ be the function which coincides with $g_{n}$ on $U_{n}$ and equals $\vec{0}$ on $U\setminus U_{n}$. Thus we have that
\begin{equation}\label{eq:oompz}
\begin{gathered}
v_{n_{1}}(x)=g_{n_{1}}(x)=\left.g_{n_{1}}(x)\right|_{U_{n}}=g_{n}(x)=v_{n}(x)\\
\text{for all }\;x\in U_{n}\text{ and all}\;n_{1}>n\text{ and all }\;n\in\mathbb{N}.
\end{gathered}
\end{equation}
This shows that, for each fixed $n\in\mathbb{N}$, the pointwise limit $\lim_{k\to\infty}v_{k}(x)$ exists for all $x\in U_{n}$. Since $U=\bigcup_{n\in\mathbb{N}}U_{n}$ this in turn implies that this same pointwise limit $\lim_{k\to\infty}v_{k}(x)$ exists for \textit{all} $x\in U$ and so we can define a new function $g:U\to\mathbb{C}^{d}$
by setting 
$$g(x):=\lim_{k\to\infty}v_{k}(x)\quad \text{for all }\;x\in U.$$ 
In view of (\ref{eq:oompz}) we have that, for each $n\in\mathbb{N}$,
\begin{equation}\label{eq:ggn}
g(x)=g_{n}(x)\mbox{ for all\,}x\in U_{n}
\end{equation}
We will now show that $f$ has a weak gradient on $U$ and that $g$ is that gradient. Given an arbitrary $\phi\in C_{c}^{\infty}(U)$ let $K$ be a compact set contained in $U$ such that $\phi$ and therefore also $\nabla\phi$ both vanish on $U\setminus K$. There exists $n\in\mathbb{N}$ such that $K\subset U_{n}$, and therefore the restriction of $\phi$ to $U_{n}$ is in $C_{c}^{\infty}(U_{n})$. We also have that 
\begin{equation}\label{eq:phivanish}
\phi(x)=0\text{ and }\;\nabla\phi(x)=\vec{0}\quad \text{for all }\;x\in U\setminus U_{n}.
\end{equation}
Consequently, using (\ref{eq:phivanish}), then the definitions of $f_{n}$ and $g_{n}$ and then (\ref{eq:ggn}) and then (\ref{eq:phivanish})
again, we obtain that 
\begin{align*}
\int_{U}f\nabla\phi dx & =\int_{U_{n}}f\nabla\phi dx=\int_{U_{n}}f_{n}\nabla\phi dx=-\int_{U_{n}}\phi g_{n}dx\\
 & =-\int_{U_{n}}\phi gdx=-\int_{U}\phi gdx.
\end{align*}
Since the preceding argument holds for every $\phi\in C_{c}^{\infty}(U)$, this shows that $g$ indeed is the (necessarily unique) weak gradient of $f$ on $U$. \\
Suppose now that $f$ has the property that, for every open subset $V$ of $U$ such that $\overline{V}$ is a compact subset of $U$, the restriction of $f$ to $V$ is an element of $W^{1,p}(V)$. Then the preceding argument shows that $f$ has a weak gradient $\nabla f$ on $U$ which equals the function $g$ obtained above. Furthermore, the functions $f_{n}$ and $g_{n}$ defined as above now satisfy $f_{n}\in W^{1,p}(U_{n})$ and $g_{n}\in L^{p}(U_{n},\mathbb{C}^{d})$ and they
coincide with the restrictions of $f$ and of $g$ respectively to the set $U_{n}$. Thus we have $\chi_{U_{n}}f\in L^{p}(U)$ and $\chi_{U_{n}}\nabla f\in L^{p}(U,\mathbb{C}^{d})$. Let $K$ be an arbitrary compact subset of $U$. For some $n\in\mathbb{N}$ we have that $K\subset U_{n}$ and so $\chi_{K}f\in L^{p}(U)$ and
$\chi_{K}\nabla f\in L^{p}(U,\mathbb{C}^{d})$. This shows that $f\in W_{loc}^{1,p}(U)$.\\
We now prove the reverse implication.\\
Suppose that $f\in W_{loc}^{1,p}(U)$ and that $\nabla f$ denotes its weak gradient on $U$.  Let $V$ be an arbitrary open subset of $U$ whose closure $K:=\overline{V}$
is a compact subset of $U$. \\
Let $g=\left.f\right|_{V}$ and let $\nabla g$ denote the $\mathbb{C}^{d}$-valued function defined \textit{only on} $V$ which is the weak gradient of $g$ on $V$. The fact that $f\in W_{loc}^{1,p}(U)$ implies that $\chi_{K}f$ and $\chi_{K}\nabla f$ are elements, respectively of $L^{p}(U)$ and $L^{p}(U,\mathbb{C}^{d})$. The restriction of $\chi_{K}f$ to $V$ must of course equal $g$. Furthermore, by the definition of weak gradients and by their uniqueness, the restriction to $V$ of $\chi_{K}\nabla f$ must equal $\nabla g$. The fact that $f\in W_{loc}^{1,p}(U)$ implies that $\chi_{K}f$ and $\chi_{K}\nabla f$ are elements, respectively of $L^{p}(U)$ and $L^{p}(U,\mathbb{C}^{d})$. Therefore $g$ and $\nabla g$ are in $L^{p}(V)$ and $L^{p}(V,\mathbb{C}^{d})$ respectively.\\
In other words $g\in W^{1,p}(V)$, and this completes our proof of Claim \ref{claim:alternative_loc_space}.
\end{proof}
\subsection{From locally Lipschitz to globally Lipschitz - A proof of Fact \eqref{fact:LocLipschitzGoesGlobal}}\label{sub:LocalToGlobalZZ}
In this part of the appendix we will prove the following claim, which was formulated as the fact \eqref{fact:LocLipschitzGoesGlobal} just after the statement of Lemma \ref{lem:range}:
\begin{claim}\label{claim:known_factZZ}
If the function $\psi:U\to\mathbb{C}$ is locally Lipschitz and has compact support, then it is Lipschitz on all of $U$.
\end{claim}
\begin{proof}
Suppose that $\psi:U\to\mathbb{C}$ is locally Lipschitz and has compact support. Then there exists a compact subset $K$ of $U$ such that $\psi$ vanishes on $U\setminus K$.
The number $\delta:=\text{dist}(K,\mathbb{R}^{d}\setminus U)$ is strictly positive and we let $K_{1}$ be the set 
$$K_1=\left\{ x\in\mathbb{R}^{d}\,:\,\text{dist}(x,K)\le \frac{\delta}{2}\right\}.$$
Obviously $K_{1}\subset U$. The compactness of $K$ implies that $K_{1}$ is closed and bounded and therefore also compact. Our hypotheses on $\psi$ and the compactness of $K_{1}$ ensure that the quantities 
$$C_{1}:=\sup\left\{ \frac{\left|\psi(x)-\psi(y)\right|}{\left|x-y\right|}:x,\, y\in K_{1},\, x\ne y\right\} $$
and 
$$C_{2}:=\sup\left\{ \left|\psi(x)\right|:x\in K_{1}\right\} $$ 
are both finite. To complete the proof we will show that
\begin{equation}
\left|\psi(x)-\psi(y)\right|\le\left|x-y\right|\max\left\{ C_{1},\frac{2C_{2}}{\delta}\right\} \label{eq:WWnH}
\end{equation}
for every $x$ and $y$ in $U$. The left side of \eqref{eq:WWnH} equals $0$ whenever $x$ and $y$ are both in $U\setminus K_{1}$ and it is bounded by $C_{1}\left|x-y\right|$ whenever $x$ and $y$ are both in $K_{1}$. Thus we only need to consider what happens when $x\in K_{1}$ and $y\in U\setminus K_{1}$.\\ 
\begin{itemize}
\item If $x\notin K$ then we again have $\left|\psi(x)-\psi(y)\right|=0$.
\item If $x\in K$ then $\left|x-y\right|>\delta/2$ and so 
$$ \left|\psi(x)-\psi(y)\right|=\left|\psi(x)\right|\le C_{2}=\frac{2C_{2}}{\delta}\cdot\frac{\delta}{2}\le\frac{2C_{2}}{\delta}\left|x-y\right|.$$
\end{itemize}
We conclude that \eqref{eq:WWnH} indeed holds for all $x$ and $y$ in $U$ and our proof is complete.
\end{proof}

\subsection{A 
Stein-Weiss 
like theorem in the case where $\rw$ is not Lipschitz}\label{subsec:SW_with_no_lipZZ}
In this subsection of the appendix we will focus on showing that there exist weight functions, $\omega_0$ and $\omega_1$, for which one can obtain a Stein-Weiss like theorem, (i.e. for which \eqref{eq:OurFormula} is valid) even though $\rw$, defined in \eqref{eq:rw}, is not Lipschitz. This will show that Theorem \ref{thm:Stein_Weiss_for_W} gives a sufficient yet not necessary condition for obtaining such a result.\\
Our starting point for finding 
such an example
is the following simple lemma:
\begin{lemma}\label{lemapp:ObviousEquivalence}
Let $\rho_{0}$ and $\rho_{1}$ be weight functions on $\mathbb{R}^{d}$ which satisfy the hypotheses of Theorem \ref{thm:Stein_Weiss_for_W}. (In particular this implies
that the function $\log\left(\frac{\rho_{0}}{\rho_{1}}\right)$ is Lipschitz on $\mathbb{R}^{d}$.) Let $\omega_{1}$ be a weight function on $\mathbb{R}^{d}$ which satisfies 
\begin{equation}\label{eq:crO-Cr}
c\rho{}_{1}(x)\le\omega_{1}(x)\le C\rho_{1}(x)\text{ for all\,}x\in\mathbb{R}^{d}
\end{equation}
for two positive constants $c$ and $C$. If we set $\omega_0=\rho_0$, then we have that
$$
\left[W^{1,p}\pa{\mathbb{R}^{d},\omega_{0}},W^{1,p}\pa{\mathbb{R}^{d},\omega_{1}}\right]_{\theta}=W^{1,p}\pa{\mathbb{R}^{d},\omega_{\theta}}
$$
to within equivalence of norms.
\end{lemma}
\begin{proof}
This is an immediate consequence of Theorem \ref{thm:Stein_Weiss_for_W} combined with the observations made in Subsection \ref{subsec:equivalence_of_weights}.\end{proof}
Lemma \ref{lemapp:ObviousEquivalence} gives us a possible approach for constructing our desired example: We start with a pair of weights that satisfy our Stein-Weiss like theorem, Theorem \ref{thm:Stein_Weiss_for_W}, and modify one of the weights in a way that keeps the new weight equivalent to the original weight, but violates the Lipschitz condition on $\rw$. We shall do that now.\\
Consider the weight functions $\rho_{0}\equiv1$ and $\rho_{1}(x)=e^{-\sqrt{1+\abs{x}^{2}}}$. It is obvious that $\rho_{0}$ and $\rho_{1}$ both satisfy the compact boundedness condition and that  
$$\log\left(\frac{\rho_{0}(x)}{\rho_{1}(x)}\right)=\sqrt{1+|x|^{2}}$$
is Lipschitz on $\mathbb{R}^{d}$. Thus, we conclude that $\rho_0$ and $\rho_1$ satisfy the conditions of Theorem \ref{thm:Stein_Weiss_for_W}.\\
Now, let $\omega_{1}$ be the function
$$\omega_{1}(x)=e^{-\sqrt{1+|x|^{2}}-\sin\left(e^{|x|^{2}}\right)}.$$
Clearly $\omega_1$ satisfies \eqref{eq:crO-Cr} with $c=1/e$ and $C=e$.\\
Since $\rw=\omega_0 / \omega_1=1/\omega_1$ we see that
$$ \nabla\log\left(\rw(x)\right)=\nabla\left(\sqrt{1+|x|^{2}}+\sin\left(e^{|x|^{2}}\right)\right)=\frac{x}{\sqrt{1+|x|^{2}}}+2xe^{|x|^{2}}\cos\left(e^{|x|^{2}}\right).$$
The above is unbounded on $\R^d$, which gives us the example we needed.\\
In fact, one can do more with this example. One can even use it to show that this pair of weight functions 
has
the property that:
\begin{center}
\textit{$\mathcal{W}^{p}\left(\mathbb{R}^{d},\theta,r_{\omega}\right)$ is strictly smaller than $W^{1,p}\pa{\mathbb{R}^{d},\omega_{\theta}}$\\ 
for every $p\in[1,\infty)$ and $\theta\in(0,1)$.}
\end{center}
We leave the proof of this claim to the reader.
\par This example also leads us naturally to ask:
\begin{ques*}
Can one find weight functions $\omega_0$ and $\omega_1$ on $\R^d$, that are not equivalent to weights that satisfy the conditions of Theorem \ref{thm:Stein_Weiss_for_W} but for which we nevertheless have \eqref{eq:OurFormula}? 
\end{ques*}
We'd like to conclude this subsection by noting that a theme common to this example and the one presented in Subsection \ref{subsec:W_p_is_not_W_1p} is that they both use weight functions which are the product of a ``well behaved'' function with a highly oscillatory but bounded function.

\subsection{A comment about homogeneous weighted Sobolev spaces of univariate functions}\label{subsec:homogeneous_sobolevZZ}
In this last subsection we will show that, when the open set $U$ is $\R$, one can easily obtain a Stein-Weiss like theorem for the homogeneous weighted Sobolev spaces,  $\dot{W}^{1,p}(\mathbb{R},\omega)$, as a direct consequence of the regular Stein-Weiss theorem (as we claimed at the end of Subsection \ref{subsec:future}).\\
We start with a definition
\begin{definition}\label{def:homo_sobolevZZ}
For each weight function $\omega:\mathbb{R}\to(0,\infty)$ and each $p\in\left[1,\infty\right]$ let $\dot{W}^{1,p}(\mathbb{R},\omega)$ be the space of equivalence classes modulo constants of measurable functions $f:\mathbb{R}\to\mathbb{C}$ which have a weak derivative $f'$ in $L^{p}\left(\mathbb{R},\omega\right)$. 
This space is normed by $\left\Vert f\right\Vert _{\dot{W}^{1,p}\pa{\mathbb{R},\omega}}=\left\Vert f'\right\Vert _{L^{p}\pa{\mathbb{R},\omega}}$. 
\end{definition}
Suppose that $\omega$ satisfies the compact boundedness condition, or at least that it satisfies the weaker condition \eqref{eq:KO-Property}
which we used in Lemma \ref{lem:banach_space_and_loc}. 
Then every function $g$ in $L^{p}(\mathbb{R},\omega)$ is locally integrable and therefore the function
$$G(x):=\int_{0}^{x}g(t)dt$$ 
is defined for all $x\in\mathbb{R}$ and is absolutely continuous on every bounded interval. Moreover, its pointwise derivative $G'$ exists almost everywhere and coincides with $g$ almost everywhere. The fact that $g$ is also the weak derivative of $G$ is immediate due to the fact that we are on $\R$ and $G$ is absolutely continuous. \\
From all the above we see that the map $T$ defined by 
$$Tg(x)=\int_{0}^{x}g(t)dt$$
is a continuous linear map from $L^{p}\pa{\mathbb{R},\omega}$ onto $\dot{W}^{1,p}\pa{\mathbb{R},\omega}$.
The inverse mapping of $\dot{W}^{1,p}(\mathbb{R},\omega)$ onto $L^{p}(\mathbb{R},\omega)$
is of course simply the derivative map $D$ defined by $Df=f'$. It is also clear that $T$ and $D$ are both isometries. Since $L^{p}\left(\mathbb{R},\omega\right)$ is complete, we conclude that $\dot{W}^{1,p}(\mathbb{R},\omega)$ is complete for all weight functions $\omega$ which satisfy \eqref{eq:KO-Property}. 

Now, given $p_{0}$ and $p_{1}$ in $[1,\infty)$ and weight functions $\omega_{0}$ and $\omega_{1}$ on $\mathbb{R}$ which both satisfy the compact boundedness condition (as we saw, we can also make do with somewhat weaker assumptions), we see that $\left(\dot{W}^{1,p_{0}}\pa{\mathbb{R},\omega_{1}},\dot{W}^{1,p_{1}}\pa{\mathbb{R},\omega_{1}}\right)$ is a Banach couple. We use Theorem \ref{thm:interpolation_thm_basic} to obtain that the operators
$T$ and $D$ satisfy
$$T:\left[L^{p_{0}}\left(\mathbb{R},\omega_{0}\right),L^{p_{1}}\left(\mathbb{R},\omega_{1}\right)\right]_{\theta} \to \left[\dot{W}^{1,p_{0}}(\mathbb{R},\omega_{1}),\dot{W}^{1,p_{1}}(\mathbb{R},\omega_{1})\right]_{\theta}$$
and
$$D:\left[\dot{W}^{1,p_{0}}(\mathbb{R},\omega_{1}),\dot{W}^{1,p_{1}}(\mathbb{R},\omega_{1})\right]_{\theta} \to \left[L^{p_{0}}\left(\mathbb{R},\omega_{0}\right),L^{p_{1}}\left(\mathbb{R},\omega_{1}\right)\right]_{\theta},$$
with norms that do not exceed $1$.\\
Since the weight function $\omega_{\theta,p_{\theta}}$ defined by \eqref{eq:pwtheta} also has the compact boundedness property, we see that $T$ and $D$ are also isometries, respectively, of $L^{p_{\theta}}\pa{\mathbb{R},\omega_{\theta,p_{\theta}}}$ onto $\dot{W}^{1,p_{\theta}}\pa{\mathbb{R},\omega_{\theta,p_{\theta}}}$
and of $\dot{W}^{1,p_{\theta}}\pa{\mathbb{R},\omega_{\theta,p_{\theta}}}$ onto $L^{p_{\theta}}\pa{\mathbb{R},\omega_{\theta,p_{\theta}}}$. Combining these facts with the formula \eqref{eq:GenCalderThm} shows that 
$$\left[\dot{W}^{1,p_{0}}\pa{\mathbb{R},\omega_{1}},\dot{W}^{1,p_{1}}\pa{\mathbb{R} ,\omega_{1}}\right]_{\theta}=\dot{W}^{1,p_{\theta}}\pa{\mathbb{R},\omega_{\theta,p_{\theta}}},$$
with equality of norms.

\end{document}